\documentclass[11pt]{amsart}

\usepackage[english]{babel}
\usepackage{amsmath, amsthm, amsfonts,stmaryrd,amssymb}
\numberwithin{equation}{section}
\usepackage{bbm}
\usepackage{accents}
\usepackage{tikz-cd}
\usepackage{mathrsfs}                      
\usepackage{mathtools}                     
\usepackage{booktabs}
\usepackage{todonotes}
\usepackage[all]{xy}
\usepackage[utf8]{inputenc}
\usepackage{mathtools}                     
\usepackage{todonotes} 
\usepackage{framed}
\usepackage{graphicx}   
\usepackage[margin=3cm]{geometry}
\usepackage{hyperref} 
\usepackage{xcolor}
\hypersetup{
    colorlinks=true,
    linkcolor=blue,  
    urlcolor=cyan,
    pdftitle={},
    pdfauthor={},
}

\usepackage{subfig}
\usepackage{pgfplots}
\usepackage{cleveref}
\usepackage{thmtools}
\usepackage{stmaryrd}
\usepackage{enumerate}
\usepackage{dsfont}
\usepackage{cancel}
\usepackage[abs]{overpic}

\theoremstyle{plain}
\newtheorem{theorem}{Theorem}[section]

\newtheorem{remark}[theorem]{Remark}
\newtheorem{lemma}[theorem]{Lemma}

\newtheorem{proposition}[theorem]{Proposition}

\theoremstyle{definition}
\newtheorem{definition}[theorem]{Definition}
\newtheorem{example}[theorem]{Example}

\def\Om{\Omega}

\def\ed{\mathrm{d}}

\def\p{\partial}

\def\R{{\mathbb R}}

\def\exp{\operatorname{exp}}

\newcommand{\mres}{\mathbin{\vrule height 1.6ex depth 0pt width
0.13ex\vrule height 0.13ex depth 0pt width 1.3ex}}

\let\bs=\boldsymbol
\let\mb=\mathbb
\let\mc=\mathcal

\let\mr=\mathrm

\def\x{\mathbf{x}}

\def\brho{\boldsymbol{\rho}}
\def\bF{\boldsymbol{F}}
\def\bphi{\boldsymbol{\phi}}

\def\Tc{\Cs}
\def\Ec{\mathcal{E}}

\def\Lc{\mathcal{L}}

\def\Wc{{W}}

\def\Pc{\mathcal{P}}
\def\Fc{\mathcal{F}}
\def\Gc{\mathcal{G}}
\def\Ac{\mathcal{A}}
\def\Uc{\mathcal{U}}
\def\Rd{\mathbb{R}^d}

\usepackage[ruled,vlined]{algorithm2e}
 
\DeclareMathOperator*{\argmin}{argmin}

\def\d{\text{d}}

\def\t{\text{t}}

\def\rhoab{\rho^\alpha}

\def\brhoab{\brho^\alpha}
\def\bphiab{\bphi^\alpha}

\def\grad{\nabla}
\def\n{\bs{n}}

\def\PcO{\Pc(\Om)}

\def\intO{\int_{\Om}}

\def\Cs{\mathcal{T}}
\def\Es{\overline{\Sigma}}
\def\EsIn{\Sigma}

\def\edge{\sigma}
\def\EsK{\Es_K}
\def\EsInK{\EsIn_K}

\def\PCs{\mathbb{P}_{\Cs}}

\def\Fs{\mathbb{F}_\Cs}

\def\one{\boldsymbol{1}}

\def\indf{\mathds{1}}

\def\C{\textit{C}\,} 
\def\rhoe{\varrho} 

\def\bh{{\boldsymbol h}}
\def\bmu{\boldsymbol{\mu}}
\def\bnu{\boldsymbol{\nu}}

\def\ed{\mathrm{d}}
\def\extra{\mathsf{E}_\alpha}

\usepackage{array,multirow,makecell}
\setcellgapes{1pt}
\makegapedcells
\newcolumntype{R}[1]{>{\raggedleft\arraybackslash }b{#1}}
\newcolumntype{L}[1]{>{\raggedright\arraybackslash }b{#1}}
\newcolumntype{C}[1]{>{\centering\arraybackslash }b{#1}}
\usepackage{pifont}
\newcommand{\cmark}{\ding{51}}%

\def\scheme{EVBDF2} 

\title[From geodesic extrapolation to a BDF2 scheme for Wasserstein gradient flows]{From geodesic extrapolation to a variational BDF2 scheme for Wasserstein gradient flows}

\begin{document}
	
	\date{\today}
	
	\author[T.\ O. Gallou\"et]{Thomas O.\ Gallou\"et}
	\address{Thomas O. Gallou\"et  (\href{mailto:thomas.gallouet@inria.fr}{\tt thomas.gallouet@inria.fr}), Team Mokaplan, Inria Paris  75012 Paris, CEREMADE, CNRS, UMR 7534, Université Paris-Dauphine, PSL University, 75016 Paris, France} 
	\author[A. Natale]{Andrea Natale}
	\address{Andrea Natale (\href{mailto:andrea.natale@inria.fr}{\tt andrea.natale@inria.fr}), Inria, Univ. Lille, CNRS, UMR 8524 - Laboratoire Paul Painlevé, F-59000 Lille, France
	} 
	\author[G. Todeschi]{Gabriele Todeschi}
	\address{Gabriele Todeschi (\href{mailto:gabriele.todeschi@univ-grenoble-alpes.fr}{\tt gabriele.todeschi@univ-grenoble-alpes.fr}), Univ. Grenoble-Alpes, ISTerre, F-38058 Grenoble, France
	} 
	
	\begin{abstract} 
		We introduce a time discretization for Wasserstein gradient flows based on the classical Backward Differentiation Formula of order two. The main building block of the scheme is the notion of geodesic extrapolation in the Wasserstein space, which in general is not uniquely defined. We propose several possible definitions for such an operation, and we prove convergence of the resulting scheme to the limit PDE, in the case of the Fokker-Planck equation.
		For a specific choice of extrapolation we also prove a more general result, that is convergence towards EVI flows.
		Finally, we propose a variational finite volume discretization of the scheme which numerically achieves second order accuracy in both space and time.
	\end{abstract}
	
	\maketitle
	
	
	\vspace{1em}

	\vspace{1em}
	
	\section{Introduction}\label{sec:LJKO2intro}
	
	In this paper we are concerned with the construction of second-order in time discretizations for the following system of PDEs,
	describing the time evolution of a density $\rhoe: [0,T]\times \Omega \rightarrow \mathbb{R}_+ $ on a convex compact domain $\Omega$ and over the time interval $[0,T]$:
	\begin{equation}\label{eq:model}
		\partial_t \rhoe - \mathrm{div} \left(\rhoe \nabla \frac{\delta \mathcal E}{\delta \rho} (\rhoe) \right) = 0  \quad \text{on }  (0,T)\times \Omega \,,
	\end{equation}
	with initial and boundary conditions:
	\begin{equation}\label{eq:bcs}
		\rhoe(0,\cdot) = \rho_0\,, \quad \rhoe \nabla \frac{\delta \mathcal E}{\delta \rho} (\rhoe) \cdot n_{\partial \Omega} = 0 \quad \text{on }  (0,T) \times \partial \Omega\,,
	\end{equation}
	for a given initial density $\rho_0$, and where $n_{\partial \Omega}$ denotes the outward pointing normal to $\partial \Omega$. In equation \eqref{eq:model}, $\mathcal{E}:L^1(\Omega; \mathbb{R}_+) \rightarrow \mathbb{R}$ is a functional of the density and describes the energy of the system. Different choices for $\mathcal{E}$ yield different equations modeling a wide range of phenomena. Typical examples are the Fokker-Planck equation \cite{jordan1998fokkerplanck}, the porous medium equation \cite{otto2001geometry} or the Keller-Segel equation \cite{blanchet2013KellerSegel}, but also more complex cases such as multiphase flows \cite{cances2017multiphase,laurenccot2013gradient,cances2019two} or crowd motion models \cite{santambrogio2018crowd} can be considered.
	
	Since the density satisfies the continuity equation with zero boundary flux, its total mass is conserved. Moreover, the energy decreases along the evolution:
	\[
	\frac{\ed}{\ed t} \mathcal{E}(\rhoe(t,\cdot) ) \leq 0\,.
	\]
	This behaviour is a consequence of the fact that system \eqref{eq:model}, under suitable assumptions on the energy, can be interpreted as a gradient flow in the space of probability measures $\mc{P}(\Omega)$ equipped with the Wasserstein distance $W_2$. This interpretation is well-known since the pioneering work of Jordan, Kinderlehrer and Otto \cite{jordan1998fokkerplanck}, who showed that one recovers the Fokker-Planck equation when following the steepest descent curve of an entropy functional with respect to the Wasserstein metric. Such result is best explained in the time-discrete setting: given a uniform decomposition $0 = t_0 <t_1 < \ldots < t_N= T$ of the interval $[0,T]$ with time step $\tau \coloneqq t_{n+1} -t_n$, consider the sequence $(\rho_n)_n$ defined for $1\leq n\leq N$ by
	\begin{equation}\label{eq:jko}
		\rho_{n} = \argmin_{\rho \in \mathcal{P}(\Omega)} \frac{W^2_2(\rho,\rho_{n-1})}{2\tau} + \mc{E}(\rho)\,,
	\end{equation}
	where the energy is given by
	\begin{equation}\label{eq:energyfp}
		\mathcal{E}(\rho) = \int_\Omega V\rho + \rho \log \rho \,,
	\end{equation}
	with $V:\Omega\rightarrow \mathbb{R}$ being a Lipschitz function, if $\rho$ is absolutely continuous with respect to the Lebesgue measure and $+\infty$ otherwise. Then, one can show that the discrete curve $t\mapsto \tilde{\rhoe}(t)$, defined by $\tilde{\rhoe}(t,\cdot) = \rho_{n-1}$ for $t\in(t_{n-1}, t_{n}]$ and $1 \leq n \leq N$, converges uniformly in the $W_2$ distance to the unique solution of the Fokker-Planck equation
	\begin{equation}\label{eq:fp}
		\partial_t \rhoe - \mathrm{div} (\rhoe \nabla V) - \Delta \rhoe  = 0  \quad \text{on }  (0,T)\times \Omega \,,
	\end{equation}
	satisfying \eqref{eq:bcs}.
	
	The numerical scheme defined in equation \eqref{eq:jko} is known as JKO scheme and it allows one to interpret many different models as Wasserstein gradient flows. It also provides a convenient framework both for the analysis of such models (e.g., to prove {existence of solutions or} exponential convergence towards steady states) \cite{ambrosio2008gradient,santambrogio2017euclidean}, and for the design of numerical discretizations \cite{benamou2016augmented,carrillo2019primal,cances2020LJKO,leclerc2020lagrangian,carlier2017convergence}. In fact, reproducing the JKO scheme at the discrete level generally  implies energy stability even in very degenerate settings. Moreover in the case of convex energies one can use robust convex optimization tools that, e.g., can easily take into account the positivity constraint on the density or even other type of strong constraints (as in the case of incompressible immiscible multiphase flows in porous media, see Section \ref{ssec:multiphase}). 
	
	Since the JKO scheme is a variational version of the implicit Euler scheme, it is an order one method. Recently, several higher-order alternatives to the JKO scheme have been proposed, but it is not trivial to translate them into a fully-discrete setting (see \cite{Matthes2019bdf2,Legendre2017VIM}, and Section \ref{ssec:vimbdf2} below for a detailed description of such approaches). In fact, to the best of our knowledge, there exists no viable fully-discrete approach able to compute with second order accuracy general Wasserstein gradient flows while preserving (to some extent) the underlying variational structure.	
	
	In this paper we contribute to this quest by reformulating the classical multi-step scheme based on the Backward Differentiation Formula of order two (BDF2) as the composition of two inner steps: a geodesic extrapolation step, and a standard JKO step. We refer to the resulting scheme as Extrapolated Variational BDF2 (\scheme) scheme. As the extrapolation step is not uniquely defined (since Wasserstein geodesics may not be globally defined in time), we provide several natural notions of extrapolation and for some of these we provide convergence guarantees for the resulting scheme.
	For a particular choice of extrapolation, which unfortunately is not covered by our theory, we also propose a simple and efficient (space-time) discretization.
	Importantly, we find numerically that this does indeed produce second-order accurate solutions both in space and time.

	\subsection{Description of the BDF2 approach and main results} In the Euclidean setting, the gradient flow associated to a smooth real-valued convex function $F:\Rd\rightarrow\R$ and a starting point $x_0\in\Rd$, is the unique solution to the Cauchy problem
	\begin{equation}\label{eq:gradF}
	\left \{
	\begin{array}{ll}
		x'(t) = - \nabla F(x(t))\,, & \quad \forall\, t>0\,,\\
		x(0) = x_0\,. \\
	\end{array}
	\right.
	\end{equation}
	The BDF2 scheme applied to such a system, with time step $\tau>0$, can be written as follows: given $x_0,x_1 \in \mathbb{R}^d$, for $n\geq 2$ find $x_n\in \mathbb{R}^d$ satisfying
	\begin{equation}\label{eq:BDF2_eu_OC}
	\frac{3}{2\tau} \Big(  x_n- \frac{4}{3} x_{n-1}+\frac{1}{3} x_{n-2} \Big) = -\nabla F(x_n)\,.
	\end{equation}
	This can be interpreted as an implicit Euler step, with starting point 
	\[
	x^\alpha_{n-1} \coloneqq  x_{n-2} + \alpha(x_{n-1}- x_{n-2}) =  x_{n-1} + \beta(x_{n-1}- x_{n-2}) \,,
	\]  where $\alpha =4/3$ and $\beta = \alpha -1=1/3$, and with time step $(1-\beta) \tau = 2\tau/3$. In turn, $x^\alpha_{n-1}$ coincides with the Euclidean extrapolation at time $\alpha$, from $x_{n-2}$ (at time $0$) to $x_{n-1}$ (at time $1$), with respect to a fictitious time variable (see Figure \ref{fig:time} for a graphical representation of the time intervals involved in the scheme).
	
	\begin{figure}
	\begin{overpic}[scale=.8]{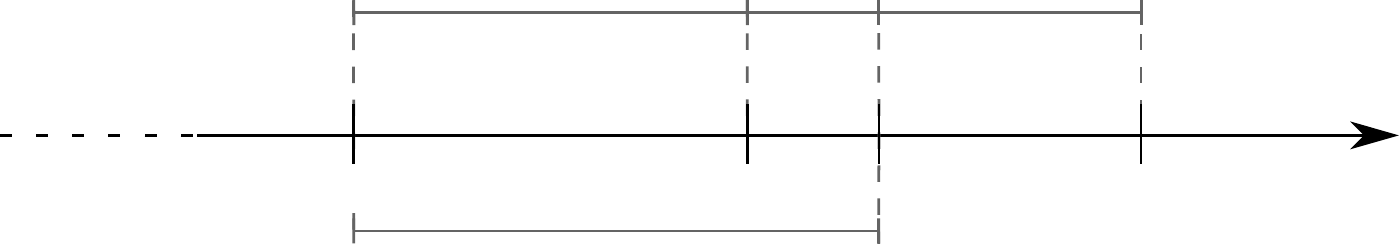}
		\put(214,42){$(1-\beta)\tau$}
		\put(182,42){$\beta\tau$}
		\put(125,42){$\tau$}
		\put(70,11){$t_{n-2}$}
		\put(163,11){$t_{n-1}$}
		\put(259,11){$t_n$}
		\put(140,-6){$\alpha \tau$}
	\end{overpic}
	\caption{A graphical representation of the time intervals involved in the definition of the EVBDF2 scheme.} \label{fig:time}
	\end{figure}
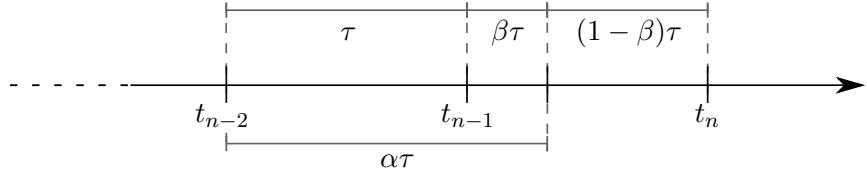
	
	In order to define a counterpart to the BDF2 scheme \eqref{eq:BDF2_eu_OC} for Wasserstein gradient flows, one needs to replace the Euclidean extrapolation at time $\alpha>1$ by an analogous operation in the space of probability measures equipped with the $W_2$ metric. In this paper, we will represent such an operation by a map $\extra:\mc{P}_2(\mathbb{R}^d)\times  \mc{P}_2(\mathbb{R}^d)  \rightarrow \mc{P}_2(\mathbb{R}^d)$ (where $\mc{P}_2(\mathbb{R}^d)$ is the set of probability measures on $\mathbb{R}^d$ with finite second moments), which we will refer to as an $\alpha$-extrapolation operator. Given such a map, we define the \scheme\ scheme as follows: given $\rho_0,\rho_1 \in \mathcal{P}(\Omega)$, for $n\geq 2$ find $\rho_n \in \mc{P}(\Omega)$ satisfying
	\begin{equation}\label{eq:bdf2metric}
	\displaystyle \rho_{n} \in \argmin_{\rho \in \mathcal{P}(\Omega)} \frac{W^2_2(\rho,\rho^\alpha_{n-1})}{2(1-\beta)\tau} +\mathcal{E}(\rho)\,,
	\quad \rho_{n-1}^\alpha = \extra (\rho_{n-2},\rho_{n-1})\,,
	\end{equation}
	where here $\mc{E}:\mc{P}(\Omega) \rightarrow \mathbb{R}$ is defined on the whole space $\mc{P}(\Omega)$.
	
	The extrapolation operator $\extra$ plays a crucial role in the scheme, but it is not trivial to propose an appropriate definition for it due to the structure of $W_2$ geodesics on $\mc{P}_2(\mathbb{R}^d)$.
	To clarify this, recall that 
	a (globally length-minimizing) geodesic with respect to the $W_2$ metric is a curve $\omega: [t_0,t_1] \rightarrow \mc{P}_2(\mathbb{R}^d)$ such that 
	\begin{equation}\label{eq:geodesic}
	W_2(\omega(s_0),\omega(s_1)) = \frac{|s_1-s_0|}{|t_1-t_0|} W_2(\omega(t_0),\omega(t_1))\,,
	\end{equation} 
	for all $s_0,s_1 \in (t_0,t_1)$. Given two measures $\mu_0,\mu_1 \in\mc{P}_2(\mathbb{R}^d)$ there always exists a geodesic connecting the two.
	Furthermore, due to Brenier's theorem, supposing that $\mu_0$ is absolutely continuous with respect to the Lebesgue measure, there exists a unique geodesic $\omega:[0,1]\rightarrow \mathcal{P}_2(\mathbb{R}^d)$ such that $\omega(0) = \mu_0$ and $\omega(1)=\mu_1$, and this has a very  
	simple expression: 
	\begin{equation}\label{eq:brenierth}
	\omega(t) = ((1-t) \mathrm{Id} + t \nabla u)_\# \mu_0\,,
	\end{equation}
	where $\mathrm{Id}$ is the identity map on $\Rd$ and $u:\Rd \rightarrow \mathbb{R}$ is a convex function. This means that particles travel on straight lines along the interpolation, without colliding into each other. However, for a given $\alpha>1$, there may exist no geodesic defined on $[0,\alpha]$ that coincide on $[0,1]$ with $\omega$. This is because following their straight trajectories particles may collide immediately after time $t=1$, even if both $\mu_0$ and $\mu_1$ have smooth and strictly positive densities. This means that one cannot use such geodesic extensions to define the extrapolation operator $\extra$ in a unique way. Therefore, instead of focusing on a particular definition, we only require a uniform stability bound on the extrapolation which we will need to prove the convergence of the scheme. In particular, we will focus on extrapolation operators that are dissipative in the following sense:
	
	\begin{definition}[Dissipative extrapolations] \label{def:dissipation} An extrapolation operator $\extra$ is $\theta$-dissipative if it satisfies
	\begin{equation}\label{eq:dissipation}
		W_2(\mu_1, \extra(\mu_0,\mu_1)) \leq \theta W_2(\mu_0,\mu_1)\,, 
	\end{equation}
	for any $\mu_0,\mu_1 \in \mc{P}_2(\mathbb{R}^d)$ and for a constant $\theta\geq 0$.
	\end{definition}
	Note that by equation \eqref{eq:geodesic}, if the extrapolation is consistent with the geodesic extension when this exists, then we must have $\theta\geq \alpha -1 \eqqcolon \beta$.
	Upon adding a further consistency assumption on the extrapolation given in equation \eqref{eq:consistency} below (see Remark \ref{rem:roleassumptions} for more comments on the role of our main assumptions), we can establish the following convergence result:
	
	\begin{theorem}\label{th:convergencefp} Let $\rho_0\in\mc{P}(\Omega)$ and $\mc{E}$ given by \eqref{eq:energyfp}. For any given $N\geq 1$, let $(\rho_n)_{n=0}^N$ be the discrete solution defined by the scheme \eqref{eq:bdf2metric} for given $\rho_1 \in \mc{P}(\Omega)$ (dependent on $N$), with time step $\tau = T/N$, and with $\extra$ being a $\theta$-dissipative extrapolation operator with $0\leq \beta=\alpha-1<1$ and $\theta<1/2$, and such that for all $\mu_0,\mu_1 \in \mc{P}(\Omega)$ and $\varphi \in \C^\infty_c(\Rd)$ verifying $\nabla \varphi \cdot n_{\partial \Omega} =0$ on $\partial\Omega$,
	\begin{equation}\label{eq:consistency}
		\left| \int_{\Rd} \varphi\, ( \extra(\mu_0,\mu_1) - \alpha \mu_1 + \beta \mu_0) \right|\leq C_\varphi W^2_2(\mu_0,\mu_1)\,,
	\end{equation}
	where $C_\varphi>0$ only depends on $\alpha$, $\varphi$ and $\Omega$. Suppose that $W^2_2(\rho_0,\rho_1) \leq C \tau$,  for  a constant $C>0$ independent of $\tau$, and that $\mc{E}(\rho_1)\leq \mc{E}(\rho_0)$. 
	Then, the curve $t\mapsto \tilde{\rho}_\tau(t)$ defined by $\tilde{\rho}_\tau(t) \coloneqq \rho_{n-1}$ for all $t\in(t_{n-1}, t_{n}]$ and $1 \leq n \leq N$, converges as $N\rightarrow \infty$, uniformly in the $W_2$ distance, to a distributional solution to the Fokker-Planck equation on $[0,T]\times \Omega$ and initial conditions given by $\rho_0$.
	\end{theorem}
	
	Of course, in order to achieve second order accuracy, we must set $\alpha =4/3$ and require in addition that, if there exists a geodesic $\omega:[0,\alpha]\rightarrow\mc{P}(\Omega)$ such that $\omega|_{[0,1]}$ is a geodesic from $\mu_0$ to $\mu_1$, then $\extra(\mu_0,\mu_1)$ must coincide with $\omega(\alpha)$. Importantly, we will show that there exist several different ways to define such an operator, providing therefore different convergent approaches.
	We highlight that there is no inconsistency between the scheme \eqref{eq:bdf2metric}, defined on $\mc{P}(\Omega)$, and an extrapolation operator $\extra$ valued in $ \mc{P}_2(\mathbb{R}^d)$. In fact, both for theoretical or numerical reasons, one may be led to define an extrapolation operator on the whole space to avoid issues with the boundary of $\Omega$. Nevertheless, scheme \eqref{eq:bdf2metric} is well-defined and, as long as the consistency assumption \eqref{eq:consistency} is satisfied, the convergence result of Theorem \ref{th:convergencefp} holds.
	
	One approach for producing an operator $\extra$, which enjoys a particularly rich structure, consists in reproducing the variational characterization of the linear extrapolation in the metric setting.
	Given two points $x_0, x_1 \in \mathbb{R}^d$, the Euclidean extrapolation at time $\alpha$ from $x_0$ to $x_1$ is the point $x_\alpha = \alpha x_1 - \beta x_0$ with $\beta =\alpha-1$. This can be obtained as the unique solution to
	\begin{equation}\label{eq:euclideanvar}
	x_\alpha = \underset{x\in \mathbb{R}^d}{\mathrm{argmin}} ~ \alpha |x-x_1|^2 - \beta |x-x_0|^2\,.
	\end{equation}
	Similarly, we define the metric extrapolation in the Wasserstein space as follows:
	\begin{equation}\label{eq:metricextra}
	\extra(\mu_0,\mu_1) \coloneqq \underset{\rho \in \mathcal{P}_2(\mathbb{R}^d)}{\mathrm{argmin}} ~ \alpha W_2^2(\rho,\mu_1) - \beta W_2^2(\rho,\mu_0)\,.
	\end{equation} 
	Problem \eqref{eq:metricextra} is not a convex optimization problem in the classical sense. To see this, consider the following simple counterexample. In dimension $d=1$, take \[
	\mu_0 = (\delta_{-1} + \delta_{1})/2, \quad  \mu_1 = \delta_0, \quad \nu_0=\delta_{-1}, \quad \nu_1=\delta_{1}. \]Along the interpolation $\nu(t) = (1-t)\nu_0+t\nu_1$, the first term of the functional in \eqref{eq:metricextra} is constant whereas the second one is concave.
	Nonetheless, we will show that problem \eqref{eq:metricextra} always admits a unique solution (see Proposition \ref{prop:existence}) and it also satisfies the assumptions in Theorem \ref{th:convergencefp}. Furthermore, exploiting the variational formulation of the metric extrapolation \eqref{eq:metricextra}, we can prove a more general convergence result using the Evolution Variational Inequality (EVI) characterization of gradient flows in metric spaces. More precisely, we prove the following result:
	\begin{theorem}\label{th:convergenceevi}
	Let $\rho_0\in\mc{P}(\Omega)$ and $\mc{E}:\mc{P}(\Omega) \rightarrow \mathbb{R}$ being a $\lambda$-convex energy in the generalized geodesic sense, for $\lambda\in\R_+$. For any given $N\geq 1$, let $(\rho_n)_{n=0}^N$ be the discrete solution defined by the scheme \eqref{eq:bdf2metric} for given $\rho_1 \in \mc{P}(\Omega)$ (dependent on $N$), with time step $\tau = T/N$, and with $\extra$ being the metric extrapolation \eqref{eq:metricextra} with $\beta = \alpha -1$. Suppose that $W^2_2(\rho_0,\rho_1) \leq C \tau$,  for  a constant $C>0$ independent of $\tau$, and that $\mc{E}(\rho_1)\leq \mc{E}(\rho_0)$.  Then, the curve $t\mapsto \tilde{\rho}_\tau(t)$ defined by $\tilde{\rho}_\tau(t) \coloneqq  \rho_{n-1}$ for $t\in(t_{n-1}, t_{n}]$ and $1 \leq n \leq N$, converges as $N \rightarrow \infty$, uniformly in the $W_2$ distance, to the unique absolutely continuous curve $\rhoe:[0,T]\rightarrow\mc{P}(\Omega)$ satisfying $\rhoe(0) = \rho_0$ and such that 
	for any $\nu\in\PcO$ it holds
	\begin{equation*}
		\frac{\d}{\d t} \frac{1}{2} \Wc_2^2(\varrho(t),\nu) \le \Ec(\nu)-\Ec(\varrho(t)) - \frac{\lambda}{2} \Wc_2^2(\varrho(t),\nu), \quad \forall t\in(0,T) \,.
	\end{equation*}
	\end{theorem}
	\noindent
	Remarkably, problem \eqref{eq:metricextra} admits a convex dual formulation, see Remark \ref{rem:metricextra_dual}.

	\subsection{Relation with previous works and numerical implementation issues}\label{ssec:vimbdf2}
	
	Going back to the discretization of system \eqref{eq:gradF}, each step of the BDF2 scheme \eqref{eq:BDF2_eu_OC} can also be obtained as the optimality conditions of the following problem:
	\begin{equation}\label{eq:BDF2_eu}
	x_n = \argmin_{x\in\Rd} \alpha \frac{\left|x-x_{n-1}\right|^2}{2(1-\beta)\tau} - \beta \frac{\left|x-x_{n-2}\right|^2}{2(1-\beta)\tau} +  F(x) \,.
	\end{equation}
	This suggests defining a similar formulation in Wasserstein space as follows
	\begin{equation}\label{eq:BDF2_mat}
	\rho_n \in \argmin_{\rho\in\PcO} \alpha \frac{\Wc_2^2(\rho,\rho_{n-1})}{2(1-\beta)\tau} - \beta\frac{\Wc_2^2(\rho,\rho_{n-2})}{2(1-\beta)\tau} + \Ec(\rho) \,.
	\end{equation}
	This approach has been proposed by Matthes and Plazotta \cite{Matthes2019bdf2,plazotta2018fokkerplanck}, who proved equivalent versions of Theorem \ref{th:convergencefp} and \ref{th:convergenceevi}. Even if in the Euclidean setting the analogue problems to \eqref{eq:BDF2_mat} and \eqref{eq:bdf2metric} yield the same solutions, one can check that this is not the case in the Wasserstein space (see, e.g., the example in Figure \ref{fig:equivalence}). However, just as for the metric extrapolation problem \eqref{eq:metricextra},  \eqref{eq:BDF2_mat} is not a convex optimization problem in the classical sense. For this reason, it is not easy to provide a numerical implementation of \eqref{eq:BDF2_mat} when $d\geq 2$. The same is true for the \scheme\ scheme \eqref{eq:bdf2metric} when using the metric extrapolation. Nonetheless, the advantage of using the \scheme\ scheme is that one has some freedom in choosing the extrapolation operator, which makes it more amenable to computations.
	
	\begin{figure}
	\begin{overpic}[scale=.9]{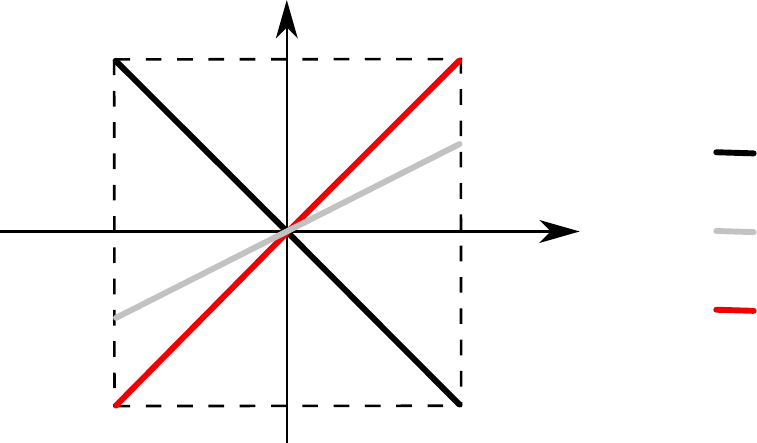}
		\put(205,74){$\rho_{n-2}$}
		\put(205,52){$\rho_{n-1}$}
		\put(142,44){$x$}
		\put(82,107){$y$}
		\put(205,30){$\rho_{n-1}^\alpha$}
	\end{overpic}
	\caption{An example for which the schemes \eqref{eq:bdf2metric} and \eqref{eq:BDF2_mat} provide different results, e.g., for the energy given by the convex indicator function of the set  $\{\mu : \mu(\mathbb{R}^d\setminus\{x=0\}) =0\}$. In the figure $\rho_{n-2}$, $\rho_{n-1}$ and $\rho_{n-1}^\alpha$ are uniformly distributed on the segments $(t,-t)$, $(t,(1-\beta)t/\alpha)$ and $(t,t)$ for $t\in[-1,1]$, respectively (in this case the geodesic from $\rho_{n-2}$ to $\rho_{n-1}$ on the time interval $[0,1]$ can be extendend up to time $\alpha$, yielding $\rho_{n-1}^\alpha$). For the scheme \eqref{eq:bdf2metric} the measure $\rho_n$ is uniformly distributed on the segment  $(0,t)$ for $t\in[-1,1]$, whereas for the scheme \eqref{eq:BDF2_mat} the measure $\rho_n$ can be obtained as the extrapolation of the projections of $\rho_{n-2}$ and $\rho_{n-1}$ on the axis $y$, and can be shown to have a strictly smaller support.
	}\label{fig:equivalence}
	\end{figure}
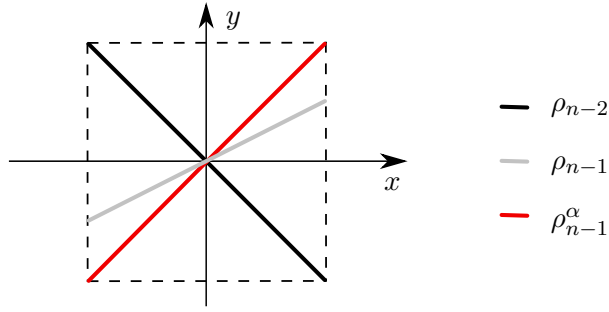
	
	Another second-order variation of the JKO scheme was proposed by Legendre and Turinici  \cite{Legendre2017VIM}, and it is based on the implicit midpoint rule, which applied to system \eqref{eq:gradF} leads to the scheme: for $n\geq1$ find $x_n \in \mathbb{R}^d$ satisfying
	\[
	\frac{1}{\tau}(x_n-x_{n-1}) = -\grad F\Big(\frac{x_n+x_{n-1}}{2}\Big) \,,
	\]
	which can be obtained as the optimality conditions of the problem
	\begin{equation}\label{eq:VIM_eu}
	x_n = \argmin_{x\in\Rd} \frac{ \left|x-x_{n-1}\right|^2}{2\tau} + 2 F\Big(\frac{x+x_{n-1}}{2}\Big) \,.
	\end{equation}
	Translating such a scheme to the Wasserstein setting yields the Variational Implicit Midpoint (VIM) scheme proposed in \cite{Legendre2017VIM}: for $n\geq1$ find $\rho_n \in \PcO$ satisfying
	\begin{equation}\label{eq:VIM}
	\rho_n \in \argmin_{\rho\in\PcO} \frac{\Wc_2^2(\rho,\rho_{n-1})}{2\tau}  + 2 \Ec(\rho_{n-1/2}) \,,
	\end{equation}
	where ${\rho}_{n-1/2}$ is the midpoint of the (not necessarily unique) geodesic between $\rho$ and $\rho_{n-1}$.
	Also in this case, it is not evident how to implement such a scheme, as it requires an explicit formula for the midpoint given the initial and final measures. This may also lead to convexity issues. Notice however that in the same spirit of our formulation of the BDF2 scheme, the implicit midpoint scheme can be formulated in the following alternative way: for $n\geq1$ find $\rho_n \in \PcO$ satisfying 
	\begin{equation}\label{eq:VIM2}
	\rho_n = {\sf E}_2 (\rho_{n-1},\rho_{n-1/2})\,, \quad \rho_{n-1/2} \in \argmin_{\rho\in\PcO} \frac{\Wc_2^2(\rho,\rho_{n-1}) }{\tau} +  \Ec(\rho) \,,
	\end{equation}
	where ${\sf E}_2 (\rho_{n-1},\rho_{n-1/2})$ denotes the extrapolation at time $\alpha =2$ of a geodesic from $\rho_{n-1}$ (at time 0) to $\rho_{n-1/2}$ (at time 1). In general, this leads to a different discrete solution than the one obtained with \eqref{eq:VIM}, although the two schemes coincide if there exists a unique geodesic extension from $\rho_{n-1}$  to $\rho_{n-1/2}$ which stays globally length-minimizing up to time 2 for all $n$.
	Nevertheless, the behavior of scheme \eqref{eq:VIM2} is radically different from that of the \scheme\ \eqref{eq:bdf2metric}, due to the different way JKO steps and extrapolations are performed. Namely, the order of the operations as well as the length of the steps play a crucial role. We will investigate this phenomenon numerically by considering a fully-discrete version of the VIM scheme and show that in general this approach may lead to persistent oscillations in the solution (Section \ref{ssec:comparison}).
	
	Providing a fully discrete version of problem \eqref{eq:model}, via the \scheme\ scheme \eqref{eq:bdf2metric}, comes with an additional challenge since the chosen space discretization should also be second-order accurate in space, in order to exploit the increased accuracy of the time discretization.
	We propose a discretization in the Eulerian framework of finite volumes. Specifically, we implement Two Point Flux Approximation (TPFA) finite volumes, which have been extensively analyzed lately for the discretization of optimal transport and Wasserstein gradient flows \cite{gladbach2018scaling,erbar2020computation,natale2021computation,cances2020LJKO,natale2020FVCA}.
	Following these last two works in particular, we propose a scheme in which the Wasserstein distance is locally linearized, at each step of the scheme, in order to decrease the computational complexity of the approach, without dropping the second-order accuracy in time. In addition, we propose one possible discrete version of the extrapolation in this setting, which can be implemented in a robust way, and we verify numerically the second-order accuracy of the resulting approach.
	
	We stress that the space discretization of the \scheme\ scheme that we propose, even if maintaining its variational structure, relies on substantial simplifications of the original problem. As a consequence, our theoretical results do not apply directly, and further work is required for a fully discrete convergence proof. Given this, the numerical results presented in Section \ref{sec:numerics} are only preliminary and they are mainly meant to demonstrate the feasibility of the approach.
	
	\section{Preliminaries and notation}\label{sec:prem}
	
	Let $\mc{P}_2(\mathbb{R}^d)$ be the space of probability measures with finite second moments. Given  $\mu_0, \mu_1 \in \mc{P}_2(\mathbb{R}^d)$, we denote by $W_2(\mu_0,\mu_1)$ the $L^2$-Wasserstein distance between $\mu_0$ and $\mu_1$ (see, e.g., 
	Chapter 5 in \cite{santambrogio2015optimal}). This can be defined via the following minimization problem:
	\begin{equation}\label{eq:w2}
	W_2^2(\mu_0,\mu_1)\coloneqq \min_{\substack{\gamma \in \Pi(\mu_0,\mu_1)}} \int {|x-y|^2}\,\ed \gamma(x,y) \,,
	\end{equation}
	where $\Pi(\mu_0,\mu_1)$ is the set of probability measures on $\mathbb{R}^d\times \mathbb{R}^d$ with marginals $\mu_0$ and $\mu_1$. This problem always admits a solution $\gamma^*$, although it is not necessarily unique, which we refer to as an optimal transport plan from $\mu_0$ to $\mu_1$. By linearity of the constraint and of the function minimized in \eqref{eq:w2}, one can easily check that the function $W_2^2$ is jointly convex with respect to its arguments (with respect to the linear structure of $\mc{P}_2(\mathbb{R}^d)$). We will refer to the space of probability measures $\mc{P}_2(\mathbb{R}^d)$ 
	equipped with the metric $W_2$ as the Wasserstein space.
	
	Problem \eqref{eq:w2} admits an alternative dynamical formulation, which was introduced by Benamou and Brenier in \cite{benamou2000computational}, and which reads as follows:
	\begin{equation}\label{eq:w2dynamic}
	W_2^2(\mu_0,\mu_1) = (t_1 - t_0) \min_{(\omega, v) \in \mathcal{C}}  \int_{t_0}^{t_1} \ed t \int \omega(t) |v(t,\cdot)|^2
	\end{equation}
	where $\mc{C}$ is the set of curves $(\omega,v)$ with finite total kinetic energy, with $\omega:[t_0,t_1] \rightarrow \mc{P}_2(\mathbb{R}^d)$ and $ v: [t_0,t_1] \rightarrow L^2(\omega(t);\mathbb{R}^d)$, satisfying weakly the continuity equation 
	\begin{equation}\label{eq:continuity}
	\partial_t \omega + \mathrm{div}(\omega v ) = 0
	\end{equation}
	with zero flux boundary conditions (i.e.\ $\omega \,v \cdot n_{\partial \Omega} = 0 $), and initial and final conditions $
	\omega(t_0) = \mu_0 $, $ \omega(t_1) = \mu_1$. The minimum in \eqref{eq:w2dynamic} is always achieved although there might be multiple minimizers. In particular, one can use formula \eqref{eq:w2dynamic} to deduce that the Wasserstein space is a geodesic space and the minimizers $\omega$ are geodesics.
	
	By the optimality conditions of problem \eqref{eq:w2dynamic}, a curve $\omega$ is a geodesic if and only if there exists a potential $\phi:[t_0,t_1]\times \mathbb{R}^d \rightarrow \mathbb{R}$ that verifies:
	\begin{enumerate}
	\item $\phi(t_0, \cdot)$ is a continuous $(-(t_1-t_0)^{-1})$-convex function, i.e.\ such that the so-called Brenier potential 
	\begin{equation}\label{eq:u} 
		x\mapsto u(x) \coloneqq  (t_1-t_0) \phi(t_0,x) + \frac{|x|^2}{2} \quad \text{is convex}\,;
	\end{equation}
	\item the potential $\phi$ is the unique viscosity solution of the Hamilton-Jacobi equation 
	\begin{equation}\label{eq:hj}
		\displaystyle \partial_t \phi + \frac{|\nabla \phi |^2}{2} =  0\,,
	\end{equation}
	or equivalently, it verifies the Hopf-Lax representation formula,
	\begin{equation}\label{eq:hamiltonjacobiintro}
		\phi(t,x) = \inf_{y\in\mathbb{R}^d} \frac{|x-y|^2}{2(t-t_0)} + \phi(t_0,y)\,;
	\end{equation}
	{\item $\nabla \phi(t,\cdot) \in L^2(\omega(t);\mathbb{R}^d)$ for a.e.\ $t\in[t_0,t_1]$ and $(\omega,\nabla \phi) \in \mc{C}$.}
	\end{enumerate}
	We say that a function $\phi$ verifying these condition is an optimal potential from $\mu_0$ to $\mu_1$ on the time interval $[t_0,t_1]$.
	Furthermore, for any optimal potential $\phi$, it holds:
	\begin{equation}\label{eq:formulationHJ}
	\frac{W_2^2(\mu_0,\mu_1)}{2(t_1-t_0)} = \int \phi(t_1,\cdot) \mu_1-\int \phi(t_0,\cdot) \mu_0\,.
	\end{equation}
	
	Because of the semi-convexity of $\phi(t_0,\cdot)$, the maps $X(t,\cdot)$, defined a.e.\ by
	\begin{equation}\label{eq:xflow}
	X(t,\cdot) \coloneqq \mathrm{Id} + (t-t_0) \nabla \phi(t_0,\cdot)
	\end{equation}
	are injective for all $t\in[t_0,t_1)$ (as the gradient of a strongly convex function), and the resulting curve of maps $X:[t_0,t_1]\times \mathbb{R}^d \rightarrow \mathbb{R}^d$ is the Lagrangian flow of the time-dependent vector field $\nabla \phi(t,\cdot)$, i.e., for a.e.\ $x\in\mathbb{R}^d$, $X(\cdot,x)$ solves the flow equation
	\[
	\frac{\ed}{\ed t} X(t,x) = \nabla \phi(t,X(t,x)),\quad X(t_0,x) = x \,.
	\]
	
	If $\mu_0$ is absolutely continuous, given an optimal potential $\phi$ and the associated Lagrangian flow $X$ defined by \eqref{eq:xflow}, one can easily verify that the curve 
	\begin{equation}\label{eq:rhopush}
	\omega(t) = X(t,\cdot)_\# \mu_0
	\end{equation}
	solves the continuity equation with velocity $\nabla \phi$ and boundary conditions $\omega(0)=\mu_0$ and $\omega(1)= \mu_1$ (in distributional sense), and therefore it is a geodesic. Moreover, using the absolute continuity of $\mu_0$, one can also show that the initial potential $\phi(t_0,\cdot)$ is uniquely defined $\mu_0$-a.e., and no other geodesic curve exists connecting $\mu_0$ and $\mu_1$.
	Note also that  from \eqref{eq:rhopush}, one can recover Brenier's result \eqref{eq:brenierth} with the Brenier potential $u$ as in \eqref{eq:u}, and also verify the equivalence with formulation \eqref{eq:w2}. As a matter of fact, in this case the optimal transport plan is also unique and is given by $\gamma^* = (\mathrm{Id},\nabla u)_\# \mu_0$, where the map $\nabla u$ is the so-called optimal transport map from $\mu_0$ to $\mu_1$. On the other hand, for any convex function $u$, setting $\phi(0,\cdot)$ via \eqref{eq:u}, the curve $\omega$ defined in \eqref{eq:rhopush} is a geodesic between $\mu_0$ and $(\nabla u)_\# \mu_0$ (and the unique one, if $\mu_0$ is absolutely continuous).

	\section{Analysis of the \scheme\ scheme}\label{sec:LJKO2analysis}
	
	In this section we collect the main properties of the \scheme\ discretization \eqref{eq:bdf2metric}, and in particular we prove Theorem \ref{th:convergencefp}, which establishes the convergence of the discrete flow generated by the scheme to the linear Fokker-Planck equation. Throughout the section, $(\rho_n)_n$ denotes a sequence of measures generated by the \scheme\ scheme \eqref{eq:bdf2metric}, where $\sf{E}_\alpha$ is a $\theta$-dissipative extrapolation, with $\theta<1/2$.
	
	\subsection{Well-posedness and classical estimate}\label{ssec:well-posedness} We start by stating some a priori bounds, which are valid for a general class of energies.
	In particular, in this paragraph, we only assume that $\mc{E}$ is lower semi-continuous with respect to the weak-* topology. Since $\mc{P}(\Omega)$ is compact for this topology (we recall that we assume $\Omega$ compact) this also implies that $\mc{E}$ is bounded from below. Problem \eqref{eq:bdf2metric} therefore admits a minimizer at each step $n$.
	
	\begin{lemma}\label{lem:almostED}
	At each step $n$, the solution $\rho_{n}$ satisfies the following inequality
	\begin{equation}\label{eq:almostED}
		(1-\theta)\frac{\Wc_2^2(\rho_{n},\rho_{n-1})}{2(1-\beta)\tau} + \Ec(\rho_{n}) \leq \theta \frac{\Wc_2^2(\rho_{n-1},\rho_{n-2})}{2(1-\beta)\tau}  + \Ec(\rho_{n-1})\,.
	\end{equation}
	\end{lemma}
	\begin{proof}
	Due to the optimality of $\rho_{n}$ and using \eqref{eq:dissipation}, we can write
	\[
	\begin{aligned}
		\frac{\Wc_2^2(\rho_{n},\rhoab_{n-1})}{2(1-\beta)\tau}   + \Ec(\rho_{n}) & \le \frac{\Wc_2^2(\rho_{n-1},\rhoab_{n-1}) }{2(1-\beta)\tau}   + \Ec(\rho_{n-1}) \\
		&\le  \frac{\theta^2}{2(1-\beta)\tau} \Wc_2^2(\rho_{n-1},\rho_{n-2}) + \Ec(\rho_{n-1})\,.
	\end{aligned}
	\]
	If $\theta=0$ this coincides with \eqref{eq:almostED}. If $\theta>0$, observe that
	by the triangular and Young's inequalities, for any $c>0$,
	\[
	\Wc_2^2(\rho_{n},\rho_{n-1}) \le \Big(1+\frac{1}{c}\Big)\Wc_2^2(\rho_{n},\rhoab_{n-1})+(1+c)\Wc_2^2(\rho_{n-1},\rhoab_{n-1})\,.
	\]
	Setting $c=\theta^{-1}-1$ in this last inequality and using again \eqref{eq:dissipation}, we can estimate the left-hand side from below using
	\[
	\begin{aligned}
		\frac{\Wc_2^2(\rho_{n},\rhoab_{n-1})}{2(1-\beta)\tau}   &\ge \frac{1}{2(1-\beta)\tau} \left(\frac{c}{c+1} \Wc_2^2(\rho_{n},\rho_{n-1}) -c\Wc_2^2(\rho_{n-1},\rhoab_{n-1})\right)   \\
		&\ge \frac{1-\theta}{2(1-\beta)\tau} \Wc_2^2(\rho_{n},\rho_{n-1}) - \frac{(1-\theta)\theta}{2(1-\beta)\tau}\Wc_2^2(\rho_{n-1},\rho_{n-2})  \,.
	\end{aligned}
	\]
	Rearranging, we obtain \eqref{eq:almostED}.
	\end{proof}
	
	Note that if we take $\beta =0$, i.e.\ we remove the extrapolation step, we can take $\theta=0$ in \eqref{eq:almostED} and recover the standard dissipation estimate for the JKO scheme.

	\begin{lemma}\label{lem:bound_W}
	Let $C_1>0$ be a constant such that $\Wc_2^2(\rho_1,\rho_0)\le C_1\tau$ and $\Ec(\rho_1)\leq \Ec(\rho_0)$. Then, it holds:
	\begin{equation}\label{eq:bound_W}
		\frac{1}{\tau} \sum_{n=0}^{N_{\tau}} \Wc_2^2(\rho_{n},\rho_{n-1}) \le C
	\end{equation}
	for a constant $C>0$ depending only on $C_1$, $\beta$, $\theta$, $\mc{E}$ and $\rho_0$.
	\end{lemma}
	\begin{proof}
	Summing over $n$ the inequality \eqref{eq:almostED} we obtain
	\begin{equation}\label{eq:EDlong}
		\frac{1-2\theta}{2(1-\beta)\tau} \sum_{n=0}^{N} \Wc_2^2(\rho_{n},\rho_{n-1}) \le  \Ec(\rho_1)-\Ec(\rho_{n}) + 		\frac{\theta}{(1-\beta)\tau}  \Wc_2^2(\rho_1,\rho_0)\,,
	\end{equation}
	Then, since $\theta<1/2$  and thanks to the lower bound on the energy and the assumption $\mc{E}(\rho_1)\leq \mc{E}(\rho_0)$, we have
	\[
	\frac{1}{\tau} \sum_{n=0}^{N} \Wc_2^2(\rho_{n},\rho_{n-1})  \le 		\frac{2(1-\beta)}{1-2\theta} \Big( \Ec(\rho_0)- \inf \Ec \Big) +\frac{2\theta}{1-2\theta}C_1 \,.
	\]
	\end{proof}
	
	\begin{remark} For a given $\rho_0$, one can always choose $\rho_1$ so that the constant $C_1$ above is independent of $\tau$ and $\mc{E}(\rho_1)\leq \mc{E}(\rho_0)$, which are also the assumptions in the statements of Theorems \ref{th:convergencefp} and \ref{th:convergenceevi}. For example, it is sufficient to take $\rho_1$ as the solution obtained after a finite number $N_0\in \mathbb{N}$ of JKO steps with time step $\tau/N_0$ and initial condition given by $\rho_0$, with $\Ec(\rho_0)<\infty$. In fact, in this case, by the same proof as for Lemma \ref{lem:bound_W} (with $\beta=\theta=0$), one can take $C_1 = 2 (\Ec(\rho_0)- \inf \Ec )$.
	\end{remark}

	\subsection{Convergence towards the Fokker-Planck equation}\label{ssec:LJKO2convergence}
	
	Given a Lipschitz continuous exterior potential $V\in W^{1,\infty}(\Om)$, the Fokker-Planck equation is given by
	\begin{equation}\label{eq:FokkerPlanckLJKO2}
	\partial_t \rhoe = \Delta \rhoe + \mathrm{div} (\rhoe \nabla V) \quad \text{in } (0,T)\times\Om \,,
	\end{equation}
	complemented with no-flux boundary conditions $(\grad \rhoe +\rhoe \grad V) \cdot n_{\partial \Omega} = 0$ on $\p \Om$ and an initial condition $\rhoe(0,\cdot)=\rho_0\in\PcO$. Equation \eqref{eq:FokkerPlanckLJKO2} can be interpreted as a Wasserstein gradient flow with respect to the energy functional $\mc{E}:\mc{P}(\Omega) \rightarrow \mathbb{R}$ given by
	\begin{equation}\label{eq:FokkerPlanckEnergy}
	\Ec(\rho) = \Uc(\rho) + \intO \rho V \,,
	\end{equation}
	where the internal energy $\Uc:\mc{P}(\Omega) \rightarrow \mathbb{R}$ (the entropy) is defined by
	\begin{equation}\label{eq:entropy}
	\Uc(\rho) \coloneqq
	\begin{cases}\displaystyle
		\intO  \log\left(\frac{\ed\rho}{\ed x}\right) \ed \rho \quad &\text{if $\rho \ll \ed x\mres \Omega$} \,, \\
		+\infty \quad &\text{otherwise} \,,
	\end{cases}
	\end{equation}
	where $\ed x\mres \Omega$ denotes the restriction of the Lebesgue measure to the domain $\Omega$. 
	Since the function $x\mapsto x\log x$ is strictly convex and superlinear, the energy $\mc{E}$ is also strictly convex on its domain (with respect to the linear structure of $\mc{P}(\Omega)$) and lower semi-continuous (with respect to the weak-* topology: see, e.g., Proposition 7.7 in \cite{santambrogio2015optimal}). Since $W_2^2$ is continuous and convex in its arguments, there exists a unique solution $\rho_{n}$ to problem \eqref{eq:bdf2metric} at each step $n$, and this is furthermore absolutely continuous with respect to $\ed x \mres \Omega$. Moreover, both Lemmas  \ref{lem:almostED} and \ref{lem:bound_W} apply.
	
	As in the previous paragraph, we assume that $\sf{E}_\alpha$ is a $\theta$-dissipative extrapolation with $\theta<1/2$, and $(\rho_n)_n$ denotes a sequence of measures generated by the associated \scheme\ scheme \eqref{eq:bdf2metric}. 
	Although the discrete flow does not move by strictly minimizing the energy at each step (see Lemma \ref{lem:almostED}), we will show that it converges to the maximal slope curve of $\Ec$. For this, we will rely on the same arguments as in the original work of Jordan, Kinderlehrer, and Otto \cite{jordan1998fokkerplanck} for the JKO scheme.
	
	Relying on the estimate \eqref{eq:bound_W}, the compactness arguments for obtaining a limit curve are rather standard.
	We introduce two density curves on the interval $[0,T]$, given by
	\begin{equation}\label{eq:interpolations}
	\begin{aligned}
		&\rhoe_{\tau}(t) = \sum_{n=1}^{N} \rho_{n-1} \indf_{(t_{n-1},t_{n}]} \,, \quad \rho_{\tau}(0) = \rho_0 \,, \\
		&\tilde{\rhoe}_{\tau}(t) =  \sum_{n=1}^{N} \tilde{\rhoe}_{n}(t) \indf_{(t_{n-1},t_{n}]} \,, \quad \tilde{\rho}_{\tau}(0) = \rho_0 \,,
	\end{aligned}	
	\end{equation}
	with $t \mapsto \tilde{\rhoe}_{n}(t)$ being the geodesic curve between $\rho_{n-1}$ and $\rho_{n}$ on the time interval $[t_{n-1},t_{n}]$ (i.e.\ the minimizer of problem \eqref{eq:w2dynamic} on this interval). Let  $\tilde{v}_n$ be the associated optimal vector field as in problem \eqref{eq:w2dynamic} for all $1\leq n \leq N$. By definition of $\tilde{\rhoe}_\tau$, we have that 
	\[
	\p_t \tilde{\rhoe}_{\tau} +\mathrm{div} (\tilde{\rhoe}_{\tau}\tilde{v}_{\tau}) = 0 
	\]
	in the distributional sense on $(0,T) \times\Om$,
	where $\tilde{v}_{\tau}$ is the vector field defined by $\tilde{v}_{\tau}|_{(t_{n-1},t_n]} = \tilde{v}_n$ for all $1\leq n \leq N$. Moreover, on each interval $[t_{n-1},t_{n}]$ it holds:
	\[
	\Wc_2^2(\rho_{n},\rho_{n-1}) = \tau \int_{t_{n-1}}^{t_{n}}\intO \tilde{\rhoe}_{\tau} |\tilde{v}_{\tau}|^2 \,.
	\]
	The curve $\rhoe_{\tau}$ is a piecewise constant measure-valued curve whereas $\tilde{\rhoe}_{\tau}$ is a (absolutely) continuous one, interpolating the discrete densities.
	
	\begin{proposition}\label{prop:conv_rho}
	For a given $\rho_0$ and any given $N\geq 1$, let $\rho_\tau$ be the curve defined as in equation \eqref{eq:interpolations}, with $\rho_1$ being such that $W^2_2(\rho_0,\rho_1) \leq C \tau$,  for  a constant $C>0$ independent of $\tau$, and $\mc{E}(\rho_1)\leq \mc{E}(\rho_0)$. Then, the sequence $(\rhoe_{\tau})_{\tau}$ converges uniformly in the $\Wc_2$ distance to an absolutely continuous curve $\rhoe:[0,T]\rightarrow \mc{P}(\Omega)$.
	\end{proposition}
	\begin{proof}
	
	The sequence of curves $(\tilde{\rhoe}_\tau)_{\tau\in\R_+}$, defined from $[0,T]$ to the (compact) space $\mc{P}(\Omega)$ equipped with the Wasserstein distance, is uniformly H\"older continuous.
	Indeed, for any $r,s\in (0,T], s>r$, denote $N_r,N_s$ the two integers such that $r\in(t_{N_r},t_{N_r+1}], s\in(t_{N_s},t_{N_s+1}]$.  By the dynamical formulation of the Wasserstein distance \eqref{eq:w2dynamic}, it holds
	\begin{equation}\label{eq:holder_estim}
		\begin{aligned}
			\Wc_2(\tilde{\rhoe}_{\tau}(s),\tilde{\rhoe}_{\tau}(r)) &\le  |s-r|^{\frac{1}{2}} \left(\int_{r}^{s}\intO \tilde{\rhoe}_{\tau} |\tilde{v}_{\tau}|^2\right)^{\frac{1}{2}} \le |s-r|^{\frac{1}{2}} \left( \sum_{n=N_r}^{N_s} \int_{t_{n}}^{t_{n+1}}\intO \tilde{\rhoe}_{\tau} |\tilde{v}_{\tau}|^2\right)^{\frac{1}{2}} \\
			&= |s-r|^{\frac{1}{2}} \left( \sum_{n=N_r}^{N_s} \frac{1}{\tau}\Wc_2^2(\rho_{n},\rho_{n+1})\right)^{\frac{1}{2}} \le C |s-r|^{\frac{1}{2}}
		\end{aligned}
	\end{equation}
	where in the last inequality we used the estimate \eqref{eq:bound_W}. By the generalized Ascoli-Arzelà theorem, the sequence converges uniformly in $\Wc_2$, up to a subsequence, to a limit curve $\rhoe$. As the inequality \eqref{eq:holder_estim} passes to the limit, $\rhoe$ is also an absolutely continuous curve with respect to the Wasserstein metric.
	Finally, for any $r\in[0,T]$,
	\[
	\Wc_2(\rhoe_{\tau}(r),\tilde{\rhoe}_{\tau}(r)) =	\Wc_2(\tilde{\rhoe}_{\tau}(t_{N_r}),\tilde{\rhoe}_{\tau}(r)) \le \sqrt{\tau} \left( \int_{t_{N_r}}^{t_{N_r+1}}\intO \tilde{\rhoe}_{\tau} |\tilde{v}_{\tau}|^2 \right)^{1/2}\le C \sqrt{\tau} \,,
	\]
	by the same computations.
	Therefore, the piecewise continuous curve $\rhoe_\tau$ converges uniformly with order $\sqrt{\tau}$ to the same limit curve $\rhoe$.
	
	\end{proof}
	
	To characterize the limit curve $\rhoe$ we will rely on the optimality conditions of the minimization problem in \eqref{eq:bdf2metric}, which is equivalent to a single JKO step.
	Consider an absolutely continuous measure $\rho$ and a smooth vector field ${\xi}$ tangent to the boundary of $\Om$. We define $\omega$ as the absolutely continuous curve solution to
	\begin{equation}\label{eq:curve_var}
	\partial_s \omega + \mathrm{div} (\omega {\xi}) = 0\,, \quad \text{in } (-\delta,\delta)\times\Om\,, \quad \omega(0) = \rho\,,
	\end{equation}
	for $\delta>0$.
	The variations of the energy and the Wasserstein distance along curves defined in this way can be computed explicitly as follows.
	
	\begin{lemma}\label{lem:variations_ac}
	Consider two measures $\rho\in \mc{P}(\Omega),\, \nu\in\mc{P}_2(\mathbb{R}^d)$, with $\rho$ absolutely continuous, and denote by $\gamma$ the optimal transport plan from $\rho$ to $\nu$.
	For any ${\xi}\in\C^\infty_c(\Rd;\Rd)$ with ${\xi}\cdot n_{\partial \Omega} = 0$ on $\partial \Om$, let $\omega$ be the curve of measures defined by \eqref{eq:curve_var} with $\omega(0)=\rho$. It holds:
	\begin{equation}\label{eq:var_W2}
		\frac{\d \Wc_2^2(\omega(s),\nu)}{\d s} \Big|_{s=0} = 2 \int_{\Rd\times \Rd} (x-y)\cdot{\xi}({x}) \, \d \gamma(x,y) \,,
	\end{equation}
	\begin{equation}\label{eq:var_energy}
		\frac{\d \Ec(\omega(s))}{\d s} \Big|_{s=0} = -\intO  \mathrm{div} ({\xi} ({x}))  \ed \rho({x}) + \intO \grad V({x}) \cdot {\xi}({x}) \ed \rho({x}) \,.
	\end{equation}
	
	\end{lemma}
	\begin{proof}
	See \cite[Corollary 10.2.7]{ambrosio2008gradient} and \cite[Theorem 5.30]{villani2003topics}.
	\end{proof}
	
	We are now ready to prove Theorem \ref{th:convergencefp} which states the convergence of the sequence of curves $(\rhoe_\tau)_{\tau}$ towards a distributional solution of equation \eqref{eq:FokkerPlanckLJKO2}. Specifically, we need to prove that, for all $\varphi\in\C^{\infty}_c([0,T)\times\Rd)$ such that $\grad\varphi\cdot n_{\partial \Omega}=0$ on $\p\Om$, the limit curve $\rhoe$ satisfies:
	\begin{equation}\label{eq:FokkerPlanck_dist}
	-\int_0^T\intO \p_t \varphi \rhoe-\intO \varphi(0)\rhoe(0) 	-\int_0^T \intO \Delta \varphi \rhoe + \int_0^T\intO \grad V \cdot \grad \varphi \rhoe = 0 \,.
	\end{equation}
	
	\begin{proof}[Proof of Theorem \ref{th:convergencefp}]
	Let us define for all $\rho \in \mc{P}(\Omega)$,
	\begin{equation}
		\Gc(\rho_{n-1},\rho_{n-2};\rho) \coloneqq \frac{ \Wc_2^2(\rho,\rhoab_{n-1}) }{2(1-\beta)\tau}+ \Ec(\rho) \,,
	\end{equation}
	which is minimized by $\rho_n$, by the definition of the scheme \eqref{eq:bdf2metric}. Consider a smooth function $\varphi\in\C^{\infty}_c([0,T)\times\Rd)$ such that $\grad\varphi\cdot n_{\partial \Omega}=0$ on $\p\Om$. We define the sequence $(\varphi_{n})_n\subset\C_c^{\infty}(\Rd)$ as $\varphi_{n}=\varphi(t_{n},\cdot)$.
	Consider then a curve $\omega$ defined as in \eqref{eq:curve_var} with $\omega(0)=\rho_{n}$ and ${\xi}= \grad \varphi_{n-2}$.
	Denoting by ${\gamma}_{n}$ the optimal transport plan from $\rho_{n}$ to $\rhoab_{n-1}$, and using \eqref{eq:var_W2}-\eqref{eq:var_energy} as well as the optimality of $\rho_n$, we obtain
	\begin{equation}\label{eq:BDF2_mod_OC}
		\begin{multlined}
			\frac{\d \Gc(\rho_{n-1},\rho_{n-2}; \omega(s))}{\d s}\Big|_{s=0} = \frac{1}{(1-\beta)\tau} \int_{\Rd \times \Rd} (x-x_\alpha)\cdot \grad \varphi_{n-2}({x}) \d {\gamma}_{n} (x,x_\alpha) \\ -\intO  \Delta \varphi_{n-2}({x}) \ed \rho_{n}({x})  + \intO \grad V({x}) \cdot \grad \varphi_{n-2} ({x}) \ed \rho_{n}({x}) = 0 \,.
		\end{multlined}
	\end{equation}
	Thanks to Proposition \ref{prop:conv_rho} and the regularity of $\varphi$, we immediately have
	\begin{equation*}
		\Bigg|
		\sum_{n=2}^{N} \tau \left(-\intO \Delta \varphi_{n-2} \rho_{n} + \intO \grad V \cdot \grad \varphi_{n-2} \rho_{n} \right)
		-\bigg( -\int_0^T \intO \Delta \varphi \rhoe + \int_0^T\intO \grad V \cdot \grad \varphi \rhoe \bigg)
		\Bigg| \longrightarrow 0 \,,
	\end{equation*}
	for $\tau\rightarrow0$. In order to prove that the measure $\rhoe$ is a distributional solution of equation \eqref{eq:FokkerPlanckLJKO2} we need to show that
	\begin{equation*}
		I_1 \coloneqq \Bigg|
		\sum_{n=2}^{N} \frac{1}{1-\beta} \int_{\Rd \times \Rd} (x-x_\alpha)\cdot \grad \varphi_{n-2}({x}) \d \gamma_{n} (x,x_\alpha) \,
		-\left(-\int_0^T\intO \p_t \varphi \rhoe-\intO \varphi(0)\rhoe(0) \right)
		\Bigg| \longrightarrow 0 \,,
	\end{equation*}
	as well.
	We can bound the latter quantity as $I_1\le I_2+I_3$, where $I_2 = \sum_{n=2}^{N} I_2^n $ with
	\[
	I_2^n \coloneqq \Bigg|
	\frac{1}{1-\beta} \int_{\Rd \times \Rd} (x-x_\alpha)\cdot \nabla\varphi_{n-2}({x}) \d \gamma_{n}(x,x_\alpha) 
	-\frac{1}{1-\beta} \int_{\Rd} (\rho_{n}-\alpha\rho_{n-1}+\beta\rho_{n-2}) \varphi_{n-2}\Bigg| \,,
	\]
	and
	\begin{equation*}
		I_3 \coloneqq \left| \sum_{n=2}^{N} \frac{1}{1-\beta} \int_{\Rd} (\rho_{n}-\alpha\rho_{n-1}+\beta\rho_{n-2}) \varphi_{n-2}  -
		\left(-\int_0^T\intO \p_t \varphi \rhoe-\intO \varphi(0)\rhoe(0) \right)
		\right| \,.
	\end{equation*}
	Integrating by parts the discrete derivative in this last term,
	\begin{multline*}
		\sum_{n=2}^{N} \frac{1}{1-\beta}\int_{\Rd} (\rho_{n}-\alpha\rho_{n-1}+\beta\rho_{n-2}) \varphi_{n-2} = \\
		\begin{aligned}
			&=\sum_{n=2}^{N} \frac{1}{1-\beta} \int_{\Rd} (\varphi_{n-2}-(\alpha\varphi_{n-1}-\beta\varphi_{n})) \rho_{n} + \frac{1}{1-\beta} \int_{\Rd} \beta \varphi_0\rho_0  +(\beta\varphi_1-\alpha\varphi_0)\rho_1\,.
		\end{aligned}
	\end{multline*}
	Then, since $\alpha = 1+\beta$, and thanks to the smoothness of the function $\varphi$ and Proposition \ref{prop:conv_rho}, we obtain $I_3\le C\tau$ for some constant $C$ independent of $\tau$.

	Let us focus then on the term $I_2$. Adding and subtracting $(1-\beta)^{-1}\int_{\Rd}(\rho_{n}-\rhoab_{n-1})\varphi_{n-2}$ at each step $n$, we obtain
	\begin{align}\label{eq:consistencyterm}
		I_2^n
		&\le
		\begin{aligned}[t]
			&\frac{1}{1-\beta} \left| \int_{\Rd\times\Rd} (x-x_\alpha)\cdot\nabla\varphi_{n-2}({x}) \d \gamma_{n}(x,x_\alpha) - \int_{\Rd} (\rho_{n}-\rhoab_{n-1})\varphi_{n-2} \right| \\
			&+\frac{1}{1-\beta} \left| \int_{\Rd} (\alpha \rho_{n-1}-\beta \rho_{n-2}-\rhoab_{n-1}) \varphi_{n-2} \right| \eqqcolon\frac{1}{1-\beta}( I_4^n+ I_5^n) \,.
		\end{aligned}  
	\end{align}
	Rewriting 
	\[
	\int_{\Rd} (\rho_{n}-\rhoab_{n-1})\varphi_{n-2} = \int_{\Rd\times\Rd} (\varphi_{n-2}({x})-\varphi_{n-2}(x_\alpha)) \d {\gamma}_{n}(x,x_\alpha) \,,
	\]
	we can bound $I_4^n$ as
	\[
	\begin{aligned}
		I_4^n &=\left| \int_{\Rd\times\Rd}  \varphi_{n-2}({x}) - \varphi_{n-2}(x_\alpha) - (x-x_\alpha)\cdot \nabla\varphi_{n-2}({x}) \d {\gamma}_{n}(x,x_\alpha) \right| \\
		&\le \frac{1}{2} ||\text{Hess}(\varphi_{n-2})||_{\infty} \left( \int_{\Rd\times\Rd} |x-x_\alpha|^2 \d {\gamma}_{n}(x,x_\alpha) \right) \\
		&=  \frac{1}{2} ||\text{Hess}(\varphi_{n-2})||_{\infty} \Wc_2^2(\rho_{n},\rhoab_{n-1}) \\
		&\le ||\text{Hess}(\varphi_{n-2})||_{\infty} \left(\Wc_2^2(\rho_{n},\rho_{n-1})+\Wc_2^2(\rho_{n-1},\rhoab_{n-1})\right) \\
		&\le ||\text{Hess}(\varphi_{n-2})||_{\infty} \Big(\Wc_2^2(\rho_{n},\rho_{n-1})+\theta^2 \Wc_2^2(\rho_{n-1},\rho_{n-2})\Big) \,,
	\end{aligned}
	\]
	where we used the dissipation estimate \eqref{eq:dissipation}. Similarly by the consistency assumption \eqref{eq:consistency} on the extrapolation, there exists a constant $C_\varphi$ only depending on $\varphi$ and $\Omega$ such that
	\[
	I_5^n \leq C_\varphi W^2_2(\rho_{n-1},\rho_{n-2})\,.
	\]
	Using the bound \eqref{eq:bound_W}, the estimates above imply that there exists a constant $C>0$ such that $I_2 \leq C \tau$.
	The whole term $I_1$ is therefore converging to zero and $\rhoe$ satisfies equation \eqref{eq:FokkerPlanck_dist}.
	
	\end{proof}
	
	\begin{remark}\label{rem:roleassumptions} 
		The $\theta$-dissipativity and consistency assumptions play different roles in our proof of convergence. One the one hand, $\theta$-dissipativity is essentially used to get a stable scheme (Lemma \ref{lem:almostED}) and obtain compactness (Lemma \ref{lem:bound_W}). On the other hand, the consistency assumption is necessary to obtain a consistent discretization  of the time derivative (appearing in $I_5^n$ in \eqref{eq:consistencyterm}) and recover the correct PDE in the limit.
	\end{remark}

	\section{Extrapolation in Wasserstein space}\label{sec:extrapolation}

	In this section we consider the issue of defining geodesic extrapolations in the Wasserstein space. In particular, we propose several notions of extrapolation operators $\sf{E}_\alpha$, which in some cases verify the assumptions of Theorem \ref{th:convergencefp}, and discuss their relationship.
	We consider the extrapolation problem on the whole space $\mc{P}_2(\mathbb{R}^d)$. This allows us to be more general and to simplify the exposition, in particular avoiding issues with the boundary.
	On the other hand, some of the proposed definitions may be adapted so that the extrapolation of two measures in $\mc{P}(\Omega)$ stays in $\mc{P}(\Omega)$ (see Remark \ref{rem:omega}).
	We stress that this last property is not required in our definition of the EVBDF2 scheme \eqref{eq:bdf2metric}, but it can be useful to produce a fully-discrete scheme (see Section \ref{sec:discreteextra}) or an intrinsic formulation. See Section \ref{sec:boundeddom} for more considerations on this issue.
	
	As recalled in the introduction, a globally-minimizing geodesic with respect to the $W_2$ metric is a curve $\omega: [t_0,t_1] \rightarrow \mc{P}_2(\mathbb{R}^d)$ such that 
	\begin{equation}\label{eq:geodesicglob}
	W_2(\omega(s_0),\omega(s_1)) = \frac{|s_1-s_0|}{|t_1-t_0|} W_2(\omega(t_0),\omega(t_1))\,,
	\end{equation}
	for all $s_0,s_1 \in (t_0,t_1)$. We say that $\omega: [t_0,t_1] \rightarrow \mc{P}_2(\mathbb{R}^d)$ is a locally-minimizing geodesic if for all $t\in(t_0,t_1)$ there exists an open interval $J \ni t$ such that \eqref{eq:geodesicglob} holds for all $s_0,s_1 \in J \cap (t_0,t_1)$.
	From the discussion in Section \ref{sec:prem}, given two measures $\mu_0, \mu_1 \in \mc{P}_2(\mathbb{R}^d)$, if $\mu_0$ is absolutely continuous there exists a unique globally length-minimizing geodesic connecting the two, which is given by
	\begin{equation}\label{eq:brenier2}
	\omega(t) = ((1-t) \mathrm{Id} + t \nabla u)_\# \mu_0
	\end{equation}
	for $t\in[0,1]$, where $u$ is a uniquely defined convex function $\mu_0$-a.e.\ (up to an additive constant). As a matter of fact, we have for all $s_0,s_1 \in (0,1)$,
	\begin{equation}\label{eq:geoineq}
	\begin{aligned}
		W^2_2(\omega(s_0),\omega(s_1))& \leq \int_{\mathbb{R}^d} |(1-s_0) x + s_0 \nabla u(x) - (1-s_1) x - s_1 \nabla u(x)|^2 \ed \mu_0(x)\\ &=|s_1 - s_0|^2 W_2^2(\mu_0,\mu_1),
	\end{aligned}
	\end{equation}
	where for the first inequality we used as competitor the plan $((1-s_0) \mathrm{Id} + s_0 \nabla u,  (1-s_1) \mathrm{Id} + s_1 \nabla u)_\# \mu_0$, and for the second equality the optimality of the plan $(\mathrm{Id},\nabla u)_\#\mu_0$ for the transport problem from $\mu_0$ to $\mu_1$. On the other hand, for $s_1>s_0$, by the triangular inequality and \eqref{eq:geoineq}
	\[
	\begin{aligned}
	W_2(\mu_0,\mu_1)& \leq W_2(\mu_0,\omega(s_0)) + W_2(\omega(s_0),\omega(s_1)) +W_2(\omega(s_1),\mu_1) \\
	& \leq (s_0 + 1 - s_1) W_2(\mu_0,\mu_1) + W_2(\omega(s_0),\omega(s_1))\,,
	\end{aligned}
	\]
	and therefore the inequality in \eqref{eq:geoineq} is an equality.
	Moreover, by similar calculations one can verify that for any $\alpha\geq 1$ the curve $t\in[0,\alpha] \mapsto \omega(t)$, still defined as in \eqref{eq:brenier2}, is a globally length-minimizing geodesic if and only if $u$ is $\beta/\alpha$-convex, i.e.\ the function
	\begin{equation}\label{eq:convexityextra}
	x \mapsto \alpha u(x) - \beta \frac{|x|^2}{2}  \quad \text{is convex},
	\end{equation}
	with $\beta =\alpha -1$. 
	However, in general, there is no guarantee that $u$ is strongly-convex even if $\mu_0$ and $\mu_1$ have smooth and strictly positive densities and for arbitrarily small $\beta$, as shown by the following example.
	
	\begin{example}[Contraction flow]
	Take $u = \frac{\beta}{2\alpha}|\cdot|^2$, for $\alpha >1$ and $\beta = \alpha-1$. Then, for any absolutely continuous $\mu_0 \in \mc{P}_2(\mathbb{R}^2)$ and $\mu_1 = (\nabla u)_\# \mu_0$, there exists a unique globally length-minimizing geodesic on $(-\infty,\alpha]$ such that $\omega(0) = \mu_0$ and $\omega(1)=\mu_1$, which is given by \eqref{eq:brenier2}. On the other hand, since all trajectories cross at time $\alpha$ (i.e.\ $(1-\alpha) \mathrm{Id} + \alpha \nabla u = 0$), there exists no geodesic on $(-\infty,\alpha']$ (either local or global) with $\alpha'>\alpha$ satisfying the same property.  
	\end{example}
	
	In general, globally length-minimizing geodesic extensions may not exist even if particle trajectories do not cross. In this case, however, locally length-minimizing extensions may still exist as shown in the next example.
	
	\begin{example}[Shear flow]\label{ex:shear} For $d=2$, let
	\[
	\mu_0 = \frac{1}{2}(\delta_{z} + \delta_{-z}) \,, \quad \mu_1 =  \frac{1}{2}(\delta_{z -v} + \delta_{-z+v})
	\]
	where $z = (1,1)$ and  $v=(1,0)$. In this case, there exists a unique geodesic $\omega:\mathbb{R}\rightarrow \mathcal{P}_2(\mathbb{R}^2)$ which is locally length-minimizing, and such that  $\omega(0) = \mu_0$ and $\omega(1)= \mu_1$, which is given by
	\begin{equation}\label{eq:shear}
		\omega(t) =  \frac{1}{2}(\delta_{z - tv} + \delta_{-z+tv}).
	\end{equation}
	However, $\omega$ is globally length-minimizing only when restricted on $(-\infty,2]$.
	\end{example}
	
	In order to define our scheme, we need an extrapolation operator which is well-defined even when the geodesic extension (either globally or locally length-minimizing) does not exist. In the following we will introduce different possible definitions and describe their properties.
	
	\subsection{Free-flow extrapolations} One possible strategy for defining an extrapolation consists in disregarding the convexity condition on the Brenier potential in \eqref{eq:convexityextra}, and allowing particles to cross each other while keeping their straight trajectories at constant speed. If $\mu_0\in \mc{P}_2(\mathbb{R}^d)$ is absolutely continuous, this amounts to defining, for any $\mu_1 \in \mc{P}_2(\mathbb{R}^d)$ and $\alpha>1$,
	\begin{equation}\label{eq:lagextra}
	{\sf E}_\alpha(\mu_0,\mu_1) = ((1-\alpha) \mathrm{Id} + \alpha \nabla u)_\# \mu_0 \,,
	\end{equation}
	where $u$ is a Brenier potential from $\mu_0$ to $\mu_1$ (uniquely defined $\mu_0$-a.e.). If $\mu_0$ is not absolutely continuous, there may exist multiple geodesics and optimal transport plans from $\mu_0$ to $\mu_1$. In general, we say that an extrapolation operator $\sf{E}_\alpha$ yields a \emph{free-flow extrapolation} if, denoting by $\Gamma(\mu_0,\mu_1)$ the set of optimal plans from $\mu_0$ to $\mu_1$, one has:
	\begin{equation}\label{eq:freeflow}
	\forall \mu_0,\mu_1 \in \mc{P}_2(\mathbb{R}^d)\,, ~ \exists\, \gamma^* \in \Gamma(\mu_0,\mu_1)~ :~ {\sf E}_\alpha(\mu_0,\mu_1) =(\pi_{\alpha})_\# \gamma^* , 
	\end{equation}
	where $\pi_\alpha: \mathbb{R}^d\times \mathbb{R}^d \rightarrow \mathbb{R}^d$ is the map defined by $\pi_\alpha(x,y) = x+ \alpha (y-x)$.
	By construction, when the geodesic induced by $\gamma^*$ in \eqref{eq:freeflow} admits a locally (or globally) length-minimizing geodesic extension, the resulting free-flow extrapolation is always consistent with it (for example, free-flow extrapolations yield the curve \eqref{eq:shear} in the case of Example \ref{ex:shear}). Furthermore, such extrapolation operators are admissible for our scheme in the sense of Theorem \ref{th:convergencefp}, as shown by the following proposition.
	
	\begin{proposition} \label{prop:lagextra} Any free-flow extrapolation operator ${\sf E}_\alpha: \mc{P}_2(\mathbb{R}^d)\times \mc{P}_2(\mathbb{R}^d) \rightarrow \mc{P}_2(\mathbb{R}^d)$, i.e.\ any map satisfying \eqref{eq:freeflow}, is $\beta$-dissipative with $\beta = \alpha -1$, and in addition it verifies the consistency assumption \eqref{eq:consistency} for all $\varphi\in \C^\infty_c(\mathbb{R}^d)$.
	\end{proposition}
	\begin{proof} For simplicity, we only consider the case where $\mu_0$ is absolutely continuous. Let $\nabla u$ the optimal transport map from $\mu_0$ to $\mu_1$.
	To prove the dissipativity, let $\bar{\gamma} = (\nabla u, (1-\alpha) \mathrm{Id} + \alpha \nabla u)_\# \mu_0$. Then $\bar{\gamma} \in \Pi(\mu_1, {\sf E}_\alpha(\mu_0,\mu_1))$ and by equation \eqref{eq:w2}, 
	\[
	W^2_2(\mu_1, {\sf E}_\alpha(\mu_0,\mu_1)) \leq \int |x-y|^2 \ed \bar{\gamma}(x,y)  = (1-\alpha)^2 \int | \mathrm{Id} - \nabla u|^2 \mu_1 = \beta^2W^2_2(\mu_0,\mu_1)\,.
	\]
	For the consistency, let $\varphi \in \C^\infty_c(\mathbb{R}^d)$ and observe that, by the definition of pushforward,
	\[
	\int \varphi\, ( \extra(\mu_0,\mu_1) - \alpha \mu_1 + \beta \mu_0)  = \int \left[ \varphi\big( (1-\alpha)x + \alpha \nabla u(x)\big) -\alpha \varphi\big(\nabla u(x)\big) +\beta \varphi(x) \right]\ed \mu_0(x) \,.
	\]
	Using the Taylor expansion of $\varphi$ around the point $x$ in the integral on the right-hand side, we find
	\[
	\left|\int \varphi\, ( \extra(\mu_0,\mu_1) - \alpha \mu_1 + \beta \mu_0)\right| \leq \frac{\alpha\beta}{2} \|\mathrm{Hess}(\varphi)\|_\infty W^2_2(\mu_0,\mu_1)\,.
	\]
	In the general case where $\mu_0$ is not absolutely continuous, the proof is analogous replacing transport maps by optimal plans.
	\end{proof}

	\subsection{Extrapolation with collisions}\label{sec:collisions} 
	
	Free-flow extrapolations are the simplest way to extend geodesics after their maximal time of existence, but they are purely Lagrangian and they cannot be easily implemented in an Eulerian setting. Here we describe an alternative route to construct an extrapolation operator which prevents particles to cross, and which is based on viscosity solutions of the Hamilton-Jacobi equation. The resulting operator can be implemented in a robust way, but unfortunately it falls outside the hypotheses of the convergence results presented in this work. In Section \ref{sec:LJKO2scheme}, we will describe a possible implementation (in the case of a compact domain $\Omega$) and verify numerically that it leads to a second-order scheme.
	
	Given $\mu_0, \mu_1 \in\mc{P}_2(\mathbb{R}^d)$, let us suppose that the optimal potential $\phi$ for the transport from $\mu_0$ to a given measure $\mu_1$ on the time interval $[0,1]$, is such that 
	\begin{equation}\label{eq:phi0}
	\phi(0,\cdot) \text{ is globally Lipschitz.}
	\end{equation}
	Then, the curve $\omega:[0,\infty) \rightarrow \mc{P}_2(\mathbb{R}^d)$ satisfying
	\begin{equation}\label{eq:omegahj}
	\omega(t) =  \left[\nabla \mathrm{co} \left ((1-t) \frac{|\cdot|^2}{2} + t  u\right)\right]_\# \mu_0 \,,
	\end{equation}
	where $u = |\cdot|^2/2 + \phi(0,\cdot)$ is a Brenier potential from $\mu_0$ to $\mu_1$, and where $\mr{co}$ denotes the convex hull, is well-defined. We remark that \eqref{eq:omegahj} coincides at time $t=\alpha$ with the free-flow extrapolation \eqref{eq:lagextra} as long as the convexity condition \eqref{eq:convexityextra} holds. On the other hand, if such condition is not verified, taking the convex envelope in \eqref{eq:omegahj} guarantees that the flow stays monotone and particles cannot cross.
	
	If \eqref{eq:phi0} holds, one also has that the Hamilton-Jacobi equation \eqref{eq:hj} with initial condition $\phi(0,\cdot)$ has a unique viscosity solution, which is given by the Hopf-Lax formula
	\begin{equation}\label{eq:hopf}
	\phi(t,\cdot)=\mathcal{H}_t(\phi(0,\cdot)), \quad  \mathcal{H}_t(\phi(0,\cdot))(x) \coloneqq \inf_{y\in\mathbb{R}^d} \frac{|x-y|^2}{2t} + \phi(0,y)\,.
	\end{equation}
	Note that the evolution of the density transported by the velocity field $\nabla \phi(t,\cdot)$ (via the continuity equation) is also well-defined since so is its Lagragian flow \cite{khanin2010particle,bogaevsky2004matter}. In the following lemma we show that equations \eqref{eq:hopf} and \eqref{eq:omegahj} are closely related.

	\begin{lemma} 
	Let $\phi:[0,\infty) \times \mathbb{R}^d\rightarrow \mathbb{R}$ be the unique viscosity solution to the Hamilton-Jacobi equation, or equivalently verifying \eqref{eq:hopf} for $t>0$, with $\phi(0,\cdot)$ being a Lipschitz function, and denote $u \coloneqq \phi(0,\cdot) + \frac{|\cdot|^2}{2}$. Let $\mu_0\in\mathcal{P}_2(\mathbb{R}^d)$ be an absolutely continuous measure and $\omega :[0,\infty) \rightarrow \mathcal{P}_2(\mathbb{R}^d)$ be the curve defined by \eqref{eq:omegahj}
	for all $t \geq 0$. Then,
	\begin{enumerate} 
		\item for all $t \geq 0$,
		$\omega(t)$ solves
		\begin{equation}\label{eq:varextra}
			\min_{\mu\in\mathcal{P}_2(\mathbb{R}^d)} \frac{W^2_2(\mu_0, \mu)}{2t} - \int \phi(t, \cdot) \mu\,;
		\end{equation}
		\item if $d=1$, $\omega$ is  a weak solution to the continuity equation with velocity $\nabla \phi(t,\cdot)$. 
	\end{enumerate}
	\end{lemma}
	
	\begin{proof}
	Concerning the first point, by the optimality conditions of problem \eqref{eq:varextra} \cite[Example 7.21]{santambrogio2015optimal} one can verify that:
	\[
	\frac{W_2^2(\mu_0,\mu)}{2{t}} = \int \phi(t,\cdot) \mu -\int \mathcal{H}_t(-\phi(t,\cdot))  \mu_0\,.
	\]
	Therefore, the optimal transport map from $\mu_0$ to the optimal measure $\mu$ is the gradient of $\frac{|\cdot|^2}{2}-t\mathcal{H}_t(-\phi(t,\cdot))=\frac{|\cdot|^2}{2}-t\mathcal{H}_t(-\mathcal{H}_t(\phi(0,\cdot)))$.
	Noting that for any function $\psi$ it holds
	\begin{equation}\label{eq:hopf-leg}
		\begin{aligned}
			\frac{|y|^2}{2}-t\mathcal{H}_t(\psi)(y)
			&=\frac{|y|^2}{2}-\inf_x \frac{|x-y|^2}{2} +t \psi(x) \\
			&=\sup_x y \cdot x -\left(\frac{|x|^2}{2}+t\psi(x)\right)=\left(\frac{|\cdot|^2}{2}+t\psi(\cdot)\right)^*(y), \quad \forall y\,,
		\end{aligned}
	\end{equation}
	we conclude by applying twice \eqref{eq:hopf-leg}.
	
	For the second part, we refer to Proposition 4.1 in \cite{ben2003system}, where an explicit expression for the measure transported by the flow is provided.
	\end{proof}
	
	\begin{remark}\label{rem:viscosity+continuity}
	For $d>1$, the curve \eqref{eq:omegahj} does not coincide in general with the solution of the continuity equation with velocity $\nabla \phi(t,\cdot)$. This is because \eqref{eq:omegahj} completely disregards the dynamics of mass within the shocks, which may be non-trivial \cite{ben2003system,bogaevsky2004matter}. 
	\end{remark}
	
	There are two main problems with using \eqref{eq:omegahj} to define an extrapolation operator, i.e.\ setting $\extra(\mu_0,\mu_1)=\omega(\alpha)$. First, the initial potential $\phi(0,\cdot)$ is uniquely defined only $\mu_0$-a.e., however the value of the potential outside the support of $\mu_0$ does affect the final measure $\omega(\alpha)$ for $\alpha>1$. Second, because of the same reason one can easily construct solutions that are not dissipative in the sense of Definition \ref{def:dissipation}: for example, one can take $\mu_0 =\mu_1$ with compact support and select an initial potential outside the support in such a way that $\omega(\alpha)$ (defined as in the previous lemma) is different from $\mu_1$. 
	
	\begin{remark}[Extrapolation via pressureless fluids]\label{rem:pressureless}
	With the same notation as above, one could construct geodesic continuations also by looking for solutions $\omega:[0,\infty)\rightarrow \mc{P}_2(\mathbb{R}^d)$, $v:[0,\infty)\rightarrow L^2(\omega(t);\mathbb{R}^d)$, of the following system of PDEs:
	\begin{equation}\label{eq:pressureless}\left\{
		\begin{array}{l}
			\partial_t \omega + \mathrm{div}(\omega v) = 0\,,\\
			\displaystyle \partial_t (\omega  v) + \mathrm{div}\left( \omega  {v\otimes v}\right) =  0\,,\\
		\end{array}\right.
	\end{equation}
	with initial conditions given by
	\[
	\omega (0) = \mu_0\,, \quad v(0,\cdot) = \nabla \phi(0,\cdot)\,.
	\]
	System \eqref{eq:pressureless} describes the evolution of a pressureless fluid with given initial density and velocity. In fact, any sufficiently regular solution $(\omega,v)$ of problem \eqref{eq:w2dynamic} on the time interval $[0,1]$ also solves \eqref{eq:pressureless}, since the absence of shocks implies that the Hamilton-Jacobi equation is equivalent to the conservation of momentum, i.e.\ the second equation in \eqref{eq:pressureless}. Moreover, dissipative solutions to such system, i.e.\ for which the kinetic energy $\mc{K}:[0,\infty) \rightarrow \mathbb{R}_+$ given by
	\[
	\mc{K}(t)\coloneqq \int \omega(t)|v(t)|^2
	\]
	is nonincreasing,
	provide a dissipative notion of extrapolation, since by equation  \eqref{eq:w2dynamic}, for any $\alpha =1+\beta>1$
	\[
	W^2_2(\mu_1,\omega(\alpha)) \leq \beta \int_1^\alpha \ed t \int \omega(t) |u(t)|^2  \leq  \beta^2  \int_0^1 \ed t \int \omega(t) |u(t)|^2  = \beta^2 W^2_2(\mu_0,\mu_1)\,.
	\]
	Such solutions can be constructed by requiring a \emph{sticky collision} condition, which enforces particles to share the same position after their collision. 
	In dimension higher than one, few results exist on the well-posedness of system \eqref{eq:pressureless}, so we will not consider this case in detail. On the other hand, in dimension one, sticky solutions to system \eqref{eq:pressureless}  have been widely studied in the literature. In particular,  Brenier and Grenier \cite{brenier1998one} showed that one can construct solutions to \eqref{eq:pressureless} using the unique entropy solution of a scalar conservation law, and in particular a solution to \eqref{eq:pressureless} is given by the curve
	\[ 
	\omega(t) = \tilde{X}(t,\cdot)_\# \mu_0,
	\] 
	with 
	\begin{equation}\label{eq:pushenvelope}
		\tilde{X}(t,x) \coloneqq  (\partial_x \mathrm{co}\, \psi(t,\cdot)) \circ F_0(x)\,, \quad \psi(t,s) \coloneqq \int_{0}^s X(t, F_0^{[-1]}(s')) \, \ed s'\,, 
	\end{equation}
	where $X$ is defined as in \eqref{eq:rhopush}, and $F_0^{[-1]}: [0,1] \rightarrow \overline{\mathbb{R}}$ is the quantile function of $\mu_0$, i.e.\ the pseudo-inverse of its cumulative distribution function $F_0:x \rightarrow \int_{-\infty}^x \ed \mu_0(x)$. Note that as long as the geodesic can be extended $\psi(t,\cdot)$ stays convex (as it is the integral of a monotone function) and therefore the definitions for $X(t,\cdot)$ and $\tilde{X}(t,\cdot)$, respectively in \eqref{eq:rhopush} and \eqref{eq:pushenvelope}, coincide.  We will show that in this case the resulting notion of extrapolation coincides with that provided by the metric extrapolation, which is discussed in detail in the next section.
	\end{remark}
	
	\subsection{Metric extrapolation}\label{ssec:metricextra}
	
	In analogy with the Euclidean case (see equation \eqref{eq:euclideanvar}), one can adopt a variational definition for the extrapolation, which we refer to as \emph{metric extrapolation}, and which is defined for all $\alpha>1$ and for all $\mu_0,\mu_1 \in\mc{P}_2(\mathbb{R}^d)$ by
	\begin{equation}\label{eq:metricextraF}
	\extra(\mu_0,\mu_1) \coloneqq \underset{\rho \in \mathcal{P}_2(\mathbb{R}^d)}{\mathrm{argmin}} ~ \mc{F}(\mu_0,\mu_1; \rho) \,, \quad \mc{F}(\mu_0,\mu_1; \rho) \coloneqq \alpha W_2^2(\rho,\mu_1) - \beta W_2^2(\rho,\mu_0)\,,
	\end{equation}
	where $\beta =\alpha -1$. In Proposition \ref{prop:existence}  we will show that problem \eqref{eq:metricextraF} admits indeed a unique solution, which justifies the definition of the metric extrapolation.
	
	\begin{remark}\label{rem:omega} Alternatively, one can define the metric extrapolation as in equation \eqref{eq:metricextraF} via a minimization on probability measures in $\mc{P}(\Omega)$ over a given compact domain $\Omega$. In this case, differently from the free-flow case \eqref{eq:lagextra}, the support of the extrapolated measures is always contained in ${\Omega}$. The results of this section hold also in this case without major changes.  
	\end{remark}
	
	First of all, we observe that by the triangular and Young's inequalities, for any $\rho, \mu_0, \mu_1 \in \mc{P}_2(\mathbb{R}^d)$
	\[
	W^2_2(\rho,\mu_0) \leq   \left(1+\frac{1}{\beta}\right) W^2_2(\rho,\mu_1) + (1+\beta) W^2_2(\mu_0,\mu_1)
	\]
	and therefore
	\begin{equation}\label{eq:boundbelow}
	\mc{F}(\mu_0,\mu_1; \rho) \geq  - \alpha \beta W^2_2(\mu_0,\mu_1).
	\end{equation}
	Then, if there exists a unique geodesic \eqref{eq:brenier2} from $\mu_0$ to $\mu_1$ and this can be continued up to time $\alpha$, i.e.\ if the associated Brenier potential $u$ is $\beta/\alpha$-convex, then the lower bound is attained only by $\rho=\omega(\alpha)$ with
	\[
	\omega(\alpha) =  ((1-\alpha) \mathrm{Id} + \alpha \nabla u)_\# \mu_0, 
	\]
	since by equation \eqref{eq:geodesic}
	\[
	W^2_2(\mu_0,\omega(\alpha)) = \alpha^2 W^2_2(\mu_0,\mu_1) \,, \quad 
	W^2_2(\mu_1,\omega(\alpha)) = \beta^2 W^2_2(\mu_0,\mu_1)\,.
	\]
	\begin{remark} Note that  if the geodesic extension is only locally (but not globally) minimizing, then it may not be recovered as a solution of problem \eqref{eq:metricextraF}: for instance, this is the case for the shear flow example \ref{ex:shear}, in which case one can compute the explicit solution to the metric extrapolation problem, which is represented in Figure \ref{fig:shear}.
	\end{remark} 
	
	\begin{figure}
	\begin{overpic}[scale=.9]{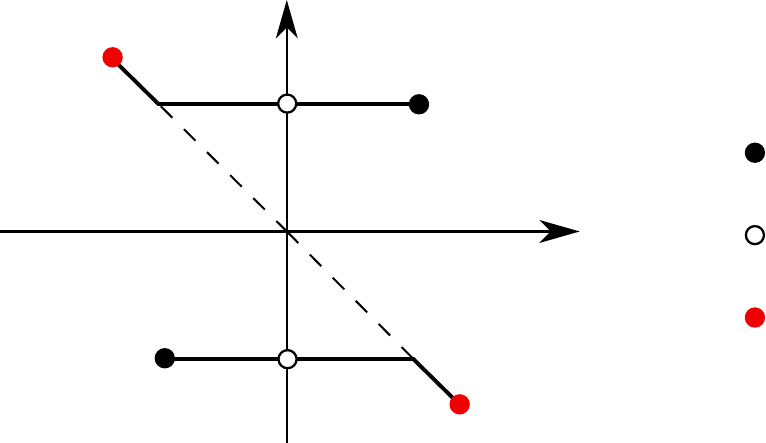}
		\put(205,74){$\mu_0$}
		\put(205,52){$\mu_1$}
		\put(142,44){$x$}
		\put(82,107){$y$}
		\put(205,30){$\mu_\alpha$}
	\end{overpic}
	\caption{Metric extrapolation in the setting of Example \ref{ex:shear}. The black solid line connecting the support of the three measures represents the trajectory followed by the extrapolated measure for different values of the parameter $\alpha$.}\label{fig:shear}
	\end{figure}
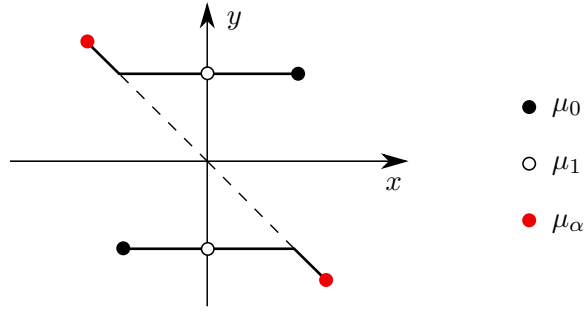
	
	Existence and uniqueness for minimizers of problem \eqref{eq:metricextraF} actually hold in general due to the fact that the functional $\Fc$ is strongly convex along particular curves known as generalized geodesics.
	To describe such curves, consider three measures $\nu_0, \nu_1, \nu_2 \in \mc{P}_2(\mathbb{R}^d)$, let $\gamma_{0,1}\in \mc{P}_2(\mathbb{R}^d\times \mathbb{R}^d)$ and  $\gamma_{0,2}\in \mc{P}_2(\mathbb{R}^d\times \mathbb{R}^d)$ optimal transport plans from $\nu_0$ to $\nu_1$ and from $\nu_0$ to {$\nu_2$}, respectively. A generalized geodesic from $\nu_1$ to $\nu_2$ with base $\nu_0$ is a curve $\omega:[0,1]\rightarrow\mathcal{P}_2(\mathbb{R}^d)$ satisfying, for all $\varphi\in C^0_b(\mathbb{R}^d)$,
	\[
	\int \varphi \omega(t)  = \int  \varphi(x_1(1-t) +x_2 t) \ed \gamma (x_0,x_1,x_2)
	\]
	where $\gamma \in \Pc_2(\mathbb{R}^d\times \mathbb{R}^d \times \mathbb{R}^d)$ is a plan verifying
	\begin{equation}\label{eq:gammagen}
	\begin{aligned} 
		\int \psi (x_0,x_1) \ed \gamma(x_0,x_1,x_2) &= \int  \psi(x_0,x_1) \ed \gamma_{0,1} (x_0,x_1) \,, \\
		\int \psi(x_0,x_2) \ed \gamma(x_0,x_1,x_2) &= \int \psi(x_0,x_2) \ed \gamma_{0,2} (x_0,x_2) \,,
	\end{aligned}
	\end{equation}
	for all $\psi \in C^0_b(\mathbb{R}^d\times \mathbb{R}^d)$.
	The existence of such a plan is a consequence of the so-called gluing lemma (Lemma  5.3.2 in \cite{ambrosio2008gradient}). In the case where $\nu_0$ is absolutely continuous, denoting by $T_{0,1}$ and $T_{0,2}$ the optimal transport plans from $\nu_0$ to $\nu_1$ and  from $\nu_0$ to $\nu_2$ respectively, there exists a unique generalized geodesic from $\nu_1$ to $\nu_2$ with base $\nu_0$ which is given by
	\begin{equation}\label{eq:gengeo}
	\omega(t) = ((1-t) T_{0,1} +t T_{0,2})_\#\nu_0\,.
	\end{equation}
	A functional $\mathcal{J}:\mathcal{P}_2(\mathbb{R}^d)\rightarrow \mathbb{R}$ is $\lambda$-convex along generalized geodesics based in $\nu_0$, if for all $\nu_1$ to $\nu_2$ and for all generalized geodesics $\omega :[0,1]\rightarrow\mathcal{P}_2(\mathbb{R}^d)$ from $\nu_1$ to $\nu_2$ with base $\nu_0$,
	\begin{equation}\label{eq:lambgenconv}
	\mathcal{J}(\omega(t)) \leq (1-t) \mathcal{J}(\nu_1) + t \mathcal{J}(\nu_2) - \lambda \frac{t(1-t)}{2} \int |x_1-x_2|^2 \ed \gamma(x_0,x_1,x_2)
	\end{equation}
	with $\gamma$ satisfying equation \eqref{eq:gammagen}. {We say that the functional $\mathcal{J}$ is $\lambda$-convex along generalized geodesics if the previous definition holds true for any $\nu_0\in\mathcal{P}_2(\mathbb{R}^d)$.}
	
	The following result was proven in \cite{Matthes2019bdf2} and provides the strong convexity of the functional $\mc{F}$ along generalized geodesics.
	
	\begin{lemma}[Theorem 3.4 in \cite{Matthes2019bdf2}]\label{lem:abconvex}
	For any $\mu_0,\mu_1 \in \mc{P}_2(\mathbb{R}^d)$, the functional $\Fc(\mu_0,\mu_1; \cdot):\mathcal{P}_2(\mathbb{R}^d) \rightarrow\mathbb{R}$ defined in \eqref{eq:metricextraF} is 2-convex along generalized geodesics based in $\mu_1$. In particular, for any $\mu_2, \mu_3 \in \mc{P}_2(\mathbb{R}^d)$ there exists a curve  $\omega:[0,1]\rightarrow \Pc_2(\mathbb{R}^d), \omega(0)=\mu_2, \omega(1)=\mu_3$, such that for all $t\in[0,1]$, it holds:
	\begin{equation}\label{eq:abconvex}
		\Fc(\mu_0,\mu_1; \omega(t)) \le (1-t) \Fc(\mu_0,\mu_1; \mu_2) + t \Fc(\mu_0,\mu_1; \mu_3)- t(1-t)W^2_2(\mu_2,\mu_3).
	\end{equation}
	\end{lemma}
		
	Lemma \ref{lem:abconvex} is the main ingredient to prove the following proposition.
	
	\begin{proposition}\label{prop:existence}
	The metric extrapolation problem \eqref{eq:metricextraF} admits a unique solution $\mu_\alpha$. Moreover, the metric extrapolation is $\beta$-dissipative, i.e.
	\begin{equation}\label{eq:betadisalp}
	W_2(\mu_1, \mu_\alpha) \leq \beta W_2(\mu_0,\mu_1)\,,
	\end{equation}
	and for all $\mu\in\mathcal{P}_2(\mathbb{R}^d)$,
	\begin{equation}
	\label{eq:rhoab_ineq}
	\Wc_2^2(\mu,\mu_\alpha) + \Fc(\mu_0,\mu_1; \mu_\alpha) \le \Fc(\mu_0,\mu_1; \mu )\,.
	\end{equation}
	\end{proposition}
	
	\begin{proof} 
	The functional $\Fc$ is strongly convex along generalized geodesics by Lemma \ref{lem:abconvex}, which implies uniqueness of the solution.
	Regarding existence, let $(\mu^n)_n$ be a minimizing sequence. We denote $m=\inf_{\mu\in\Pc_2(\Rd)} \Fc(\mu_0,\mu_1;\mu)$, which is finite due to \eqref{eq:boundbelow}, and we introduce $\Gc(\mu)=\Fc(\mu_0,\mu_1;\mu)-m$. Consider two measures $\mu^{n_1},\mu^{n_2}$ of the sequence and the generalized geodesic $\omega$ based in $\mu_1$ connecting them, as in Lemma \ref{lem:abconvex}. The inequality \eqref{eq:abconvex} for $t=\frac12$ provides
	\[
	\frac14 W_2^2(\mu^{n_1},\mu^{n_2})\le \frac12 \Gc(\mu^{n_1})+\frac12 \Gc(\mu^{n_2}) \,,
	\]
	which implies that the sequence is Cauchy in the Wasserstein space $(\Pc_2(\Rd),W_2)$. The Wasserstein space being complete \cite[Proposition 7.1.5]{ambrosio2008gradient}, the sequence converges to a measure $\mu_\alpha$, which is the minimizer since $\Fc$ is continuous.
	
	Inequality \eqref{eq:rhoab_ineq} derives again from Lemma \ref{lem:abconvex}. For a given $\mu\in\Pc_2(\mathbb{R}^d)$, consider a generalized geodesic $\omega$ as in Lemma \ref{lem:abconvex}, with $\omega(0) =\mu_\alpha$ and $\omega(1)=\mu$. By optimality of $\mu_\alpha$, it holds
	\[
	\begin{aligned}
	0 &\le \Fc(\mu_0,\mu_1; \omega(t))-\Fc(\mu_0, \mu_1; \mu_\alpha) \\
	&\le t\big(\Fc(\mu_0,\mu_1; \mu)-\Fc(\mu_0, \mu_1; \mu_\alpha )\big)- t(1-t)\Wc_2^2(\mu,\mu_\alpha)\,,
	\end{aligned}
	\]
	which, dividing by $t$ and taking the limit $t\rightarrow0$, gives \eqref{eq:rhoab_ineq}. Using \eqref{eq:boundbelow} on the left-hand side of \eqref{eq:rhoab_ineq} and then taking $\mu=\mu_1$, we obtain the estimate \eqref{eq:betadisalp}.
	
	\end{proof}
	
	In order to prove the consistency assumption we will use the following optimality conditions for problem \eqref{eq:metricextraF}.
	
	\begin{lemma}\label{lem:optgamma} Let $\mu_\alpha$ be the unique solution to problem \eqref{eq:metricextraF}. 
	There exist two optimal transport plans $\gamma_{0,\alpha}$ and $\gamma_{1,\alpha}$ from $\mu_0$ to $\mu_\alpha$ and from $\mu_1$ to $\mu_{\alpha}$, respectively, such that
	\begin{equation}\label{eq:rhoab_OC}
	\alpha \int (x_\alpha - x_1)\cdot {\xi}(x_\alpha) \d \gamma_{1,\alpha} (x_1,x_\alpha)-\beta \int (x_\alpha - x_0)\cdot {\xi}(x_\alpha) \d \gamma_{0,\alpha} (x_0,x_\alpha) = 0 \,,
	\end{equation}
	for any ${\xi}\in\C^\infty_c(\Rd;\Rd)$.
	\end{lemma}
	
	\begin{proof}
	Note that we cannot use directly Lemma \ref{lem:variations_ac} because $\mu_\alpha$ is not necessarily absolutely continuous. Therefore, in order to prove the result we construct a sequence of approximated smooth variational problems and pass to the limit in the optimality conditions.
	Let us define for $\varepsilon>0$,
	\begin{equation}\label{eq:Feps}
		\Fc_{\varepsilon}(\mu_0,\mu_1; \mu) \coloneqq \Fc(\mu_0,\mu_1; \mu) + \varepsilon \, \Uc(\mu|\nu) \,,
	\end{equation}
	where $\mc{U}(\cdot\,|\nu)$ denotes the relative entropy
	\begin{equation}\label{eq:relative_entropy}
		\Uc(\mu|\nu) \coloneqq
		\begin{cases} \displaystyle
			\int \log\left(\frac{\ed \mu}{\ed \nu}\right) \d \mu \quad &\text{if $\mu \ll \nu$} \,, \\
			+\infty \quad &\text{otherwise} \,,
		\end{cases}
	\end{equation}
	and $\nu= (2\pi)^{-d/2}\exp({-{|x|^2}/{2}})\ed x\in\mathcal{P}_2(\mathbb{R}^d)$. We introduce the regularized problem
	\begin{equation}\label{eq:defrhoabEps}
		\inf_{\mu\in\Pc_2(\Rd)} \Fc_{\varepsilon}(\mu_0,\mu_1; \mu) \,.
	\end{equation}
	
	Let $(\mu^n)_n$ be a minimizing sequence for \eqref{eq:defrhoabEps}. Due to Jensen's inequality the relative entropy is positive. Furthermore, it is convex along generalized geodesics \cite[Theorem 9.4.11]{ambrosio2008gradient}.
	Hence, reasoning as in Proposition \ref{prop:existence}, we obtain convergence in $W_2$ of $\mu^n$ to a measure $\mu_\alpha^\varepsilon$. The relative entropy is lower semi-continuous on the Wasserstein space $(\mc{P}_2(\mathbb{R}^d), W_2)$ \cite[Theorem 15.4]{ambrosio2021lectures} and therefore $\mu_\alpha^\varepsilon$ is the unique minimizer.
	
	Note that
	\[
	\int \log\left(\frac{\ed \mu}{\ed \nu}\right) \d \mu = \int \log\left(\frac{\ed\mu}{\ed x}\right) \ed \mu +\int \frac{|x|^2}{2} \ed \mu(x) + \frac{d}{2}\log(2\pi) \,, \quad \text{for $\mu\ll \nu\,$.}
	\]
	Therefore, by applying Lemma \ref{lem:variations_ac} (adapted to the case where $\Omega =\mathbb{R}^d$), we can write down the necessary optimality conditions of problem \eqref{eq:defrhoabEps}:
	\begin{multline}\label{eq:defrhoabEps_oc}
	\left.	\frac{\d \Fc_{\varepsilon}(\mu_0,\mu_1; \omega(s))}{\d s}\right|_{s=0} = 	2\alpha \int (x_\alpha - x_1)\cdot {\xi}(x_\alpha) \d \gamma^\varepsilon_{1,\alpha} (x_1,x_\alpha)\\-2\beta \int (x_\alpha - x_0)\cdot {\xi}(x_\alpha) \d \gamma^\varepsilon_{0,\alpha} (x_0,x_\alpha) +\varepsilon \int \big(x_\alpha\cdot\xi(x_\alpha)-\mathrm{div} ({\xi} (x_\alpha)) \big) \ed \mu_\alpha^{\varepsilon}(x_\alpha) = 0 \,,
	\end{multline}
	for any ${\xi}\in\C^\infty_c(\Rd;\Rd)$, where $\omega:(-\delta,\delta)\rightarrow \mc{P}_2(\mathbb{R}^d)$ is the curve of measures defined by \eqref{eq:curve_var} with $\omega(0)=\mu_\alpha$, and where we denote now by $\gamma^\varepsilon_{0,\alpha}$ and $\gamma^\varepsilon_{1,\alpha}$ the optimal transport plans from $\mu_0$ to $\mu_\alpha^\varepsilon$ and from $\mu_1$ to $\mu_\alpha^\varepsilon$, respectively.
	
	We want to show that the regularized functionals  $\Fc_\varepsilon(\mu_0,\mu_1;\cdot)$, interpreted as functionals on the Wasserstein space $(\mc{P}_2(\Rd),W_2)$, $\Gamma$-converges towards $\Fc(\mu_0,\mu_1;\cdot)$, in order to pass to the limit in the optimality conditions of problem  \eqref{eq:defrhoabEps}. 
	Since $\Fc$ is continuous with respect to $W_2$ convergence and $\Uc$ is positive, the $\Gamma$-$\liminf$ is obvious,
	\[
	\Fc(\mu_0,\mu_1; \mu) \le \liminf_{\varepsilon} \Fc(\mu_0,\mu_1; \mu_{\varepsilon}) \le  \liminf_{\varepsilon}\Fc_{\varepsilon}(\mu_0,\mu_1; \mu_\varepsilon) \,,
	\]
	for any $\mu_{\varepsilon}\rightarrow\mu$ in the Wasserstein sense. Concerning the $\Gamma$-$\limsup$, if $\Uc(\mu|\nu)<\infty$ we can take $\mu_{\varepsilon}=\mu$ as recovery sequence. Otherwise, since the set of absolutely continuous measures is dense in $\Pc_2(\Rd)$, 
	we can take a sequence of absolutely continuous measures $\mu_{\varepsilon}$ converging to $\mu$ with respect to the Wasserstein metric.
	Since $\Uc(\mu|\nu)=\infty$, up to a reparametrization we can assume that the relative entropy is increasing and that
	\[
	\Uc(\mu_{\varepsilon} | \nu)\le \frac{C}{\sqrt{\varepsilon}} \,, 
	\]
	for a constant $C$ independent of $\varepsilon$.
	Then it holds:
	\[
	\limsup_{\varepsilon} \Fc_{\varepsilon}(\mu_0,\mu_1; \mu_{\varepsilon}) = \lim_{\varepsilon} \Fc_{\varepsilon}(\mu_0,\mu_1; \mu_{\varepsilon}) = \Fc(\mu_0,\mu_1; \mu) \,.
	\]
	Therefore $\Fc_\varepsilon(\mu_0,\mu_1;\cdot)$ $\Gamma$-converges to $\Fc(\mu_0,\mu_1;\cdot)$.
	Let us show that the sequence of minimizer $(\mu^\varepsilon_\alpha)_\varepsilon$ is Cauchy.
	For this we observe that $(\Fc_\varepsilon(\mu_0,\mu_1;\mu^\varepsilon_\alpha))_\varepsilon$ is monotonically decreasing as $\varepsilon\rightarrow 0$ since, for $\varepsilon_2>\varepsilon_1$:
	\begin{equation}\label{eq:fepsmonotone}
		\Fc_{\varepsilon_2}(\mu_0,\mu_1;\mu_\alpha^{\varepsilon_2}) = (\varepsilon_2 - \varepsilon_1) \mc{U}(\mu_\alpha^{\varepsilon_2}|\nu) + 	\Fc_{\varepsilon_1}(\mu_0,\mu_1;\mu_\alpha^{\varepsilon_2}) \geq \Fc_{\varepsilon_1}(\mu_0,\mu_1;\mu_\alpha^{\varepsilon_1}).
	\end{equation}
	Since $\Fc_\varepsilon(\mu_0,\mu_1;\cdot)$ are uniformly bounded from below, $\Fc_\varepsilon(\mu_0,\mu_1;\mu^\varepsilon_\alpha)$ converges to a value ${m}$ as $\varepsilon\rightarrow 0$. Hence, we can define $\mc{G}_\varepsilon(\cdot) \coloneqq \Fc_\varepsilon(\mu_0,\mu_1;\cdot) - m\geq 0$. By the same arguments as in the proof of Proposition \ref{prop:existence} and the strong convexity of $\mc{G}_{\varepsilon_1}$ along generalized geodesics, for any $\varepsilon_2>\varepsilon_1$,  
	\[
	\frac{1}{4} W^2_2(\mu_\alpha^{\varepsilon_1},\mu_\alpha^{\varepsilon_2}) \leq \frac{1}{2}\mc{G}_{\varepsilon_1}(\mu_\alpha^{\varepsilon_1}) + \frac{1}{2} \mc{G}_{\varepsilon_1}(\mu_\alpha^{\varepsilon_2}) \leq  \frac{1}{2}\mc{G}_{\varepsilon_1}(\mu_\alpha^{\varepsilon_1}) + \frac{1}{2} \mc{G}_{\varepsilon_2}(\mu_\alpha^{\varepsilon_2})\,
	\]
	where the second inequality is a consequence of  \eqref{eq:fepsmonotone}. Since $\mc{G}_\varepsilon(\mu_\alpha^\varepsilon) \rightarrow 0$ as $\varepsilon\rightarrow 0$ we can conclude that $(\mu_\alpha^\varepsilon)_\varepsilon$ is Cauchy and by the $\Gamma$-convergence showed above, $\mu^\varepsilon_\alpha\rightarrow \mu_\alpha$ in $W_2$.
	
	Finally, by the stability of optimal transport plans {\cite[Theorem 5.20]{villani2009optimal}}, there exist optimal plans $\gamma_{0,\alpha}$ and $\gamma_{1,\alpha}$ from $\mu_0$ to $\mu_\alpha$ and from $\mu_1$ to $\mu_{\alpha}$, respectively, such that (up to the extraction of a subsequence)
	\[
	\gamma_{0,\alpha}^\varepsilon \rightharpoonup 	\gamma_{0,\alpha}\,, \quad 	\gamma_{1,\alpha}^\varepsilon \rightharpoonup 	\gamma_{1,\alpha} \,,
	\]
	weakly, i.e.\ in duality with continuous bounded functions (and also in the Wasserstein sense; in fact, the second moments of $\gamma_{0,\alpha}^\varepsilon$ and $\gamma_{1,\alpha}^\varepsilon$ converge to those of $\gamma_{0,\alpha}$ and $\gamma_{1,\alpha}$ since $\mu_\alpha^\varepsilon \rightarrow\mu_\alpha$ in the Wasserstein sense). 
	As the vector field ${\xi}$ is smooth, passing to the limit in \eqref{eq:defrhoabEps_oc} we obtain \eqref{eq:rhoab_OC}.
	
	\end{proof}
		
	\begin{proposition}
	The metric extrapolation defined via \eqref{eq:metricextraF} verifies the consistency assumption \eqref{eq:consistency} for all $\varphi\in \C^\infty_c(\mathbb{R}^d)$.
	\end{proposition} 
	\begin{proof}
	Using the same notation as in the statement of Lemma \ref{lem:optgamma}, we have that for all $\varphi \in \C^\infty_c(\mathbb{R}^d)$
	\[
	\int \varphi\, ( \mu_\alpha - \alpha \mu_1 + \beta \mu_0)  = \alpha \int ( \varphi(x_1) - \varphi(x_\alpha)) \ed \gamma_{1,\alpha}(x_1,x_\alpha) - \beta  \int ( \varphi(x_0) - \varphi(x_\alpha)) \ed \gamma_{0,\alpha}(x_0,x_\alpha)\,.
	\]
	Using the Taylor expansion of $\varphi$ at $x_\alpha$ in both integrals on the right-hand side, Lemma \ref{lem:optgamma} and the dissipation property \eqref{eq:betadisalp}, we obtain
	\[
	\begin{aligned}
	\left|  \int \varphi\, ( \mu_\alpha - \alpha \mu_1 + \beta \mu_0)  \right| &\leq \frac{1}{2} \|\mathrm{Hess}\varphi\|_\infty  ( \alpha W^2_2(\mu_1,\mu_{\alpha})  + \beta W^2_2(\mu_0,\mu_{\alpha}) ) \\& \leq  \alpha \beta \|\mathrm{Hess}\varphi\|_\infty  W^2_2(\mu_0,\mu_1). 
	\end{aligned}
	\]
	\end{proof}
	
	\begin{remark}[Relation with pressureless fluids]\label{rem:metricextra_L2}
	In dimension one, the Wasserstein distance $W_2$ coincides with the $L^2$ distance between the quantile functions. In particular, the metric extrapolation $\mu_\alpha$ is given by
	\[
	\mu_\alpha = (G_\alpha)_\# \ed x |_{[0,1]} \,, \quad G_\alpha \coloneqq \underset{\substack{G\in L^2([0,1],\mathbb{R})\\ \text{monotone}}}{\mathrm{argmin}} ~ \alpha \| G - F_1^{[-1]}\|^2_{L^2} - \beta \| G - F_0^{[-1]}\|^2_{L^2}\,,
	\]
	where $F_0^{[-1]}$ and $F_1^{[-1]}$ are the quantiles of $\mu_0$ and $\mu_1$, respectively. The solution to this problem coincides with the sticky particle model described in Remark \ref{rem:pressureless}, i.e.\ $G_\alpha = \tilde{X}(\alpha,\cdot)$ with $\tilde{X}$ as in \eqref{eq:pushenvelope}.
	\end{remark}
	
	\begin{remark}[Dual formulation of the metric extrapolation]\label{rem:metricextra_dual}
	Let us recall that the optimal transport problem \eqref{eq:w2} admits the following dual formulation \cite[Theorem 5.10]{villani2009optimal}:
	\begin{equation}\label{eq:w2dual}
	\frac{W^2(\mu_0,\mu_1)}{2} =	\sup_{\phi_0} \left\{ \int \mc{H}_1(\phi_0) \mu_1 - \int \phi_0  \mu_0 ~:~ \frac{|\cdot|^2}{2} + \phi_0(\cdot) ~ \text{ is convex}\right\} \,,
	\end{equation}
	and if $\mu_0$ is absolutely continuous, this admits a unique maximiser $\phi_0$, and $u(\cdot) \coloneqq \frac{|\cdot|^2}{2} + \phi_0(\cdot)$ is the Brenier potential from $\mu_0$ to $\mu_1$.
	However, the associated geodesic from $\mu_0$ to $\mu_1$ can be extended up to time $\alpha>1$ only if \eqref{eq:convexityextra} holds, or equivalently if 
	\begin{equation}\label{eq:alphaconvex} 
	x\mapsto \frac{|x|^2}{2} + \alpha \phi_0(x) ~~\text{is convex}.
	\end{equation} Therefore, in order to construct an extrapolation, one can instead consider the problem
	\begin{equation}\label{eq:metricextra_dual}
	\sup_{\phi_0} \left\{
	\int \mc{H}_1(\phi_0) \mu_1 - \int \phi_0  \mu_0 ~:~  \frac{|\cdot|^2}{2} + \alpha \phi_0(\cdot) ~ \text{ is convex}
	\right\},
	\end{equation}
	and, if $\mu_0$ is absolutely continuous,  set
	\[
	{\sf E}_\alpha(\mu_0, \mu_1) = (\nabla u_\alpha)_\# \mu_0,
	\]
	where $u_\alpha(\cdot)\coloneqq\frac{|\cdot|^2}{2} + \alpha \phi_0(\cdot)$ and $\phi_0$ solves \eqref{eq:metricextra_dual}.
	This extrapolation is well defined and it turns out to be a dual formulation for the metric extrapolation in the spirit of \cite{carlier2008Toland}. However, even if very natural, this dual point of view was not needed for the results presented here, and therefore it will be developed in a future work.
	
	\end{remark}
	
	\subsection{Extrapolation on bounded domains}\label{sec:boundeddom}
	
	So far we only discussed the extrapolation problem on the whole space $\mc{P}_2(\mathbb{R}^d)$. However, even if the EVBDF2 scheme is well-defined using such extrapolations, it can be convenient for numerical reasons to use an extrapolation operator mapping two measures on $\mc{P}(\Omega)$ to an extrapolated one still in $\mc{P}(\Omega)$.
	As mentioned in Remark \ref{rem:omega}, this  can be achieved easily in the case of the metric extrapolation, since one can simply perform the minimization problem \eqref{eq:metricextraF} over $\mc{P}(\Omega)$ rather than $\mc{P}_2(\mathbb{R}^d)$. It is not difficult to check that all the properties discussed in the previous section hold also with this modification.
	
	In general, a straightforward way of defining an extrapolation operator $\extra^\Omega:\mc{P}(\Omega)\times\mc{P}(\Omega)\rightarrow\mc{P}(\Omega)$ is to compose with a $W_2$ projection. Specifically, given an operator $\extra$ and $\mu_0,\mu_1\in\mc{P}(\Omega)$ we can define:
	\[
	\extra^\Omega(\mu_0,\mu_1)\coloneqq \argmin_{\rho\in\mc{P}(\Omega)} W_2^2(\rho,\extra(\mu_0,\mu_1)) = P_\# \extra(\mu_0,\mu_1)\,,
	\] 
	where $P:\mathbb{R}^d\rightarrow \Omega$ is the Euclidean projection on the convex set $\Omega$. Then, if $\extra$ is $\theta$-dissipative and satisfies the consistency assumption \eqref{eq:consistency}, also $\extra^\Omega$ does. In fact, denoting by $\gamma^*$ the optimal plan from $\mu_1$ to $\extra(\mu_0,\mu_1)$, $(\mathrm{Id},P)_\# \gamma^*\in \Pi(\mu_1,\extra^\Omega(\mu_0,\mu_1))$, and therefore one has
	\[
	\begin{aligned}
	W_2^2(\mu_1,\extra^\Omega(\mu_0,\mu_1))
	&\leq \int_{\mathbb{R}^d\times\mathbb{R}^d} |x-{P}(y)|^2 \d \gamma^*(x,y) \\
	&\le \int_{\mathbb{R}^d\times\mathbb{R}^d} |x-y|^2 \d \gamma^*(x,y) = W_2^2(\mu_1,\extra(\mu_0,\mu_1))\,,
	\end{aligned}
	\]
	which implies that $\extra^\Omega$ is $\theta$-dissipative if so is $\extra$.
	Moreover, $\forall \varphi \in \C^\infty_c(\mathbb{R}^d)$ with $\nabla \varphi \cdot n_{\partial \Omega} =0$ on $\partial \Omega$
	\[
	\begin{aligned}
	\left|\int_{\mathbb{R}^d} \varphi\left(\extra^\Omega(\mu_0,\mu_1)-\extra(\mu_0,\mu_1)\right)\right| 
	&= \left|\int_{\mathbb{R}^d} (\varphi\circ P -\varphi) \extra(\mu_0,\mu_1) \right| \\
	&\le \frac12 ||\text{Hess}(\varphi)||_{\infty} W_2^2(\extra^\Omega(\mu_0,\mu_1),\extra(\mu_0,\mu_1)) \\
	&\le \frac12 ||\text{Hess}(\varphi)||_{\infty} W_2^2(\mu_1,\extra(\mu_0,\mu_1))\,,
	\end{aligned}
	\]
	where to pass from the first to the second line we used a Taylor expansion of $\varphi$ together with the fact that $\nabla \varphi(P(x)) \cdot (P(x) -x ) = 0$ on $\mathbb{R}^d$. Hence, using the $\theta$-dissipativity property, we find that if $\extra$ verifies the consistency assumption for all $\varphi \in \C^\infty_c(\mathbb{R}^d)$, then $\extra^\Omega$ also verifies it for all $\varphi \in \C^\infty_c(\Rd)$ such that $\nabla \varphi \cdot n_{\partial \Omega} =0$ on $\partial \Omega$. As a consequence, the convergence result of Theorem \ref{th:convergencefp} holds also when the operator $\extra^\Omega$ is used in the extrapolation step.

	\section{Convergence in the EVI sense}\label{ssec:EVI}
	
	In this section, we make a further assumption on the energy functional $\Ec$. Besides lower semi-continuity, which ensures well-posedness of the scheme (see Section \ref{sec:LJKO2analysis}) we assume that $\Ec$ is $\lambda$-convex in the generalized geodesic sense on $\mc{P}(\Omega)$, for $\lambda\in\R_+$ (see equation \eqref{eq:lambgenconv}, and recall that $\Omega$ is supposed to be convex, so generalized geodesics with endpoints in $\mc{P}(\Omega)$ are well-defined as curves on $\mc{P}(\Omega)$).
	We recall that a curve $\varrho:[0,T]\rightarrow\PcO$, $\varrho(0)=\rho_0$, is a Wasserstein gradient flow in the EVI sense if for any $\nu\in\PcO$ it holds
	\begin{equation}\label{eq:evi}
	\frac{\d}{\d t} \frac{1}{2} \Wc_2^2(\varrho(t),\nu) \le \Ec(\nu)-\Ec(\varrho(t)) - \frac{\lambda}{2} \Wc_2^2(\varrho(t),\nu), \quad \forall t\in(0,T) \,,
	\end{equation}
	or, equivalently,  if for all $r,s\in(0,T)$ with $r\le s$ it holds
	\begin{equation}\label{eq:evi_integral}
	\frac{1}{2} \Wc_2^2(\varrho(s),\nu) -\frac{1}{2} \Wc_2^2(\varrho(r),\nu)\le \Ec(\nu)(s-r)-\int_{r}^{s} \Big(\Ec(\varrho(t)) + \frac{\lambda}{2} \Wc_2^2(\varrho(t),\nu)\Big) \d t \,.
	\end{equation}
	In this section, we show that the limit curve extracted from the time discretization \eqref{eq:bdf2metric} using the metric extrapolation \eqref{eq:metricextraF} (defined on either $\mc{P}(\Omega)$ or $\mc{P}_2(\mathbb{R}^d)$) satisfies the inequality \eqref{eq:evi_integral}.
	
	We first show that for scheme \eqref{eq:bdf2metric}-\eqref{eq:metricextraF} a discrete version of the inequality \eqref{eq:evi_integral} holds. 
	As the Wasserstein distance $\Wc_2^2(\cdot,\rhoab_{n-1})$ is $2$-convex along any generalized geodesic based in $\rhoab_{n-1}$ (see, e.g., the proof of Lemma \ref{lem:abconvex}), the overall functional 
	\begin{equation}
	\Gc(\rho_{n-1},\rho_{n-2};\rho) = \frac{\Wc_2^2(\rho,\rhoab_{n-1})}{2(1-\beta) \tau} + \Ec(\rho) \,,
	\end{equation}
	is $\frac{1}{(1-\beta) \tau}+\lambda>0$ convex along any generalized geodesic on $\mc{P}(\Omega)$ based in $\rhoab_{n-1}$.
	Note that in order to consider the case $\lambda<0$ one should explicitly add a restriction on the time step $\tau$ so that $\frac{1}{(1-\beta) \tau}+\lambda>0$.
	
	\begin{lemma}
	At each step $n$, for all $ \nu \in \Pc(\Omega)$, the following inequality holds:
	\begin{multline}\label{eq:evi_discrete}
	\Big(\frac{1}{2(1-\beta) \tau}+\frac{\lambda}{2}\Big)\Wc^2_2(\rho_{n},\nu) - \alpha \frac{\Wc_2^2(\nu,\rho_{n-1})}{2(1-\beta)\tau} +\beta \frac{\Wc_2^2(\nu,\rho_{n-2})}{2(1-\beta)\tau} \\ \le \Ec(\nu)-\Ec(\rho_{n}) +\alpha\beta \frac{W^2_2(\rho_{n-1},\rho_{n_2})}{2(1-\beta)\tau} -\frac{\Wc_2^2(\rho_{n},\rhoab_{n-1})}{2(1-\beta)\tau}  \,.
	\end{multline}
	\end{lemma}
	\begin{proof}
	By the discussion above, considering the generalized geodesic $\omega$ between $\nu$ and $\rho_{n}$ with base $\rhoab_{n-1}$, and using the optimality of $\rho_{n}$, we obtain
	\[
	\begin{aligned}
	0&\le \Gc(\rho_{n-1},\rho_{n-2};\omega(t)) - \Gc(\rho_{n-1},\rho_{n-2};\rho_{n})  \\
	&\le t (\Gc(\rho_{n-1},\rho_{n-2};\nu) - \Gc(\rho_{n-1},\rho_{n-2};\rho_{n})) -\frac{1}{2}\Big(\frac{1}{(1-\beta) \tau}+\lambda\Big)t(1-t)\Wc^2_2(\rho_{n},\nu).
	\end{aligned}
	\]
	Dividing by $t$ and taking the limit $t\rightarrow0$, this yields
	\[
	\Big(\frac{1}{2(1-\beta) \tau}+\frac{\lambda}{2}\Big)\Wc^2_2(\rho_{n},\nu) -\frac{\Wc_2^2(\nu,\rhoab_{n-1})}{2(1-\beta)\tau}  \le \Ec(\nu)-\Ec(\rho_{n})-\frac{\Wc_2^2(\rho_{n},\rhoab_{n-1})}{2(1-\beta)\tau}  \,.
	\]
	Adding on both side the term $-\frac{1}{2(1-\beta)\tau}\Fc(\rho_{n-1},\rho_{n-2};\rhoab_{n-1})$, using \eqref{eq:rhoab_ineq} on the left-hand side, we obtain
	\begin{multline*}
	\Big(\frac{1}{2(1-\beta) \tau}+\frac{\lambda}{2}\Big)\Wc^2_2(\rho_{n},\nu) - \alpha \frac{\Wc_2^2(\nu,\rho_{n-1})}{2(1-\beta)\tau} +\beta \frac{\Wc_2^2(\nu,\rho_{n-2})}{2(1-\beta)\tau} \\ \le \Ec(\nu)-\Ec(\rho_{n}) -\frac{1}{2(1-\beta)\tau}\Fc(\rho_{n-1},\rho_{n-2};\rhoab_{n-1}) -\frac{\Wc_2^2(\rho_{n},\rhoab_{n-1})}{2(1-\beta)\tau}  \,.
	\end{multline*}
	
	Finally, using \eqref{eq:boundbelow} on the right-hand side we conclude.
	\end{proof}
	
	\begin{proof}[Proof of Theorem \ref{th:convergenceevi}]
	We recall that thanks to the classical estimate \eqref{eq:bound_W} (Lemma \ref{lem:bound_W}), the piecewise constant curve
	\[
	\rho_{\tau}(t) = \sum_{n=1}^{N} \rho_{n-1} \indf_{(t_{n-1},t_{n}]} \,, \quad \rho_{\tau}(0) = \rho_0 \,,
	\]
	converges uniformly in the $\Wc_2$ distance to an absolutely continuous limit curve $\rhoe:[0,T]\rightarrow \mc{P}(\Omega)$ (see Proposition \ref{prop:conv_rho}).
	In order to prove convergence of the scheme in the EVI sense, we show that this curve satisfies inequality \eqref{eq:evi_integral}. Thanks to the uniform convergence in time, the procedure is the same as in \cite[Theorem 5.1]{Matthes2019bdf2}.
	
	For simplicity, assume that given $r,s\in(0,T), r\le s$, there exist $N_{\tau},M_{\tau}\in\mb{N}, N_{\tau}\le M_{\tau}$, such that $r=N_{\tau} \tau, s=M_{\tau} \tau$, $\forall \tau$.
	We multiply by $\tau$ inequality \eqref{eq:evi_discrete} and sum over $n$ from $N_{\tau}$ to $M_{\tau}$ to obtain the discrete integral form of the EVI:
	\begin{multline}\label{eq:EVIt}
	\frac{1}{2(1-\beta)} 	\sum_{n=N_{\tau}}^{M_{\tau}}\left( \Wc_2^2(\rho_{n},\nu) -\alpha \Wc_2^2(\nu,\rho_{n-1}) +\beta \Wc_2^2(\nu,\rho_{n-2}) \right)  \\ 
	\le \Ec(\nu)(t-s)-\sum_{n=N_{\tau}}^{M_{\tau}} \tau \Big(\Ec(\rho_{n})+\frac{\lambda}{2}\Wc_2^2(\rho_{n},\nu)\Big) \\
	+\frac{1}{2(1-\beta)} \sum_{n=N_{\tau}}^{M_{\tau}} \left( \alpha \beta \Wc_2^2(\rho_{n-1},\rho_{n-2})-\Wc_2^2(\rho_{n},\rhoab_{n-1})\right) .
	\end{multline}
	By canceling out terms, the left-hand side is equal to
	\begin{multline}\label{eq:EVIt_left}
	\frac{1}{2(1-\beta)} 	\left(-\alpha \Wc_2^2(\nu,\rho_{N_{\tau}-1}) + \beta \Wc_2^2(\nu,\rho_{N_{\tau}-2})
	+ \beta \Wc_2^2(\nu,\rho_{N_{\tau}-1}) \right. \\ \left.+\Wc_2^2(\rho_{M_{\tau}-1},\nu) +\Wc_2^2(\rho_{M_{\tau}},\nu) -\alpha \Wc_2^2(\nu,\rho_{{M_{\tau}-1}})\right) \,,
	\end{multline}
	and thanks to the uniform convergence in the Wasserstein distance, \eqref{eq:EVIt_left} converges to
	\[
	\frac{1}{2} \Wc_2^2(\varrho(s),\nu) -\frac{1}{2} \Wc_2^2(\varrho(r),\nu) \,,
	\]
	for $\tau\rightarrow0$, where we recall $\alpha-\beta =1$.
	Concerning the right-hand side, thanks again to the uniform convergence in the Wasserstein distance, the lower semi-continuity of $\Ec$ and Fatou's lemma, we have
	\[
	\limsup_{n\rightarrow\infty} -\sum_{n=N_{\tau}}^{M_{\tau}} \tau \Big(\Ec(\rho_{n})+\frac{\lambda}{2}\Wc_2^2(\rho_{n},\nu)\Big) \le 
	-\int_{r}^{s} \Big(\Ec(\rhoe(t)) + \frac{\lambda}{2} \Wc_2^2(\rhoe(t),\nu)\Big) \d\t \,.
	\]
	Finally, owing to bound \eqref{eq:bound_W}, we estimate the last contribution of \eqref{eq:EVIt} as
	\[
	\sum_{n} \alpha \beta \Wc_2^2(\rho_{n-1},\rho_{n-2})-\Wc_2^2(\rho_{n},\rhoab_{n-1}) \le \sum_{n} \alpha \beta  \Wc_2^2(\rho_{n-1},\rho_{n-2}) \le C \tau,
	\]
	which converges to zero. As a consequence, we recover the continuous inequality \eqref{eq:evi_integral}.
	
	\end{proof}

	\section{Finite volume discretization}\label{sec:LJKO2scheme}
	
	In this section we describe a space-time discretization of the proposed approach which yields numerically second-order accuracy both in space and time. 
	We consider a discretization in the Eulerian framework of finite volumes. In this setting, neither the free-flow extrapolation nor the metric one have a straightforward implementation. For this reason, we will construct a discrete extrapolation operator based on formula \eqref{eq:varextra}: in this way the extrapolation step is cast in a variational way allowing for a robust implementation. Although not satisfying the hypotheses of theorem \eqref{th:convergencefp}, this choice leads to a convergent and second order accurate scheme, as we will show numerically.
	As explained in Remark \ref{rem:viscosity+continuity}, the variational step \eqref{eq:varextra} differs from the direct forward integration of the continuity equation. This latter is a viable alternative to define a discrete extrapolation and leads to second order accuracy as well (see \cite{todeschi2021finite}), but it is not clear how to discretize this in a robust way.
	
	The fundamental tool is the solution of JKO steps, which requires the expensive problem of computing the Wasserstein distance. Following \cite{cances2020LJKO,natale2020FVCA}, we linearize the Wasserstein distance obtaining LJKO steps, a more affordable problem to solve. Remarkably, this approach preserves the second order accuracy in time of our time discretization. The discretization in space is based instead on Two-Point Flux Approximation (TPFA) finite volumes with a centered choice for the mobility, which leads to simple and flexible schemes which are second order accurate in space.

	\subsection{Discrete setting}
	
	TPFA finite volumes require a sufficiently regular partitioning of the domain $\Om$, according to \cite[Definition 9.1]{EGH00}.
	For simplicity, we describe the methodology in two dimensions only, although generalizations to arbitrary dimensions are possible, and for $\Om\subset \mathbb{R}^2$ being a polygonal domain.
	The discretization of $\Om$ consists of three sets: the set of cells $K\in\Cs$; the set of edges $\sigma\in\Es$, which is composed of the two subsets of internal edges $\EsIn$ and external edges $\Es\setminus\EsIn$; the set of cell centers ${(\x_K)}_{K\in\Cs}$.
	We will denote the finite volume mesh as $\left(\Cs,\Es, {(\x_K)}_{K\in\Cs}\right)$.
	The fundamental regularity hypothesis we need to construct TPFA schemes is the orthogonality between each internal edge $\sigma=K|L\in\EsIn$ and the segment $\x_L-\x_K$.
	Typical example of meshes that can be used to this end are Cartesian grids, Voronoi tessellations and Delaunay triangulations, by taking the circumcenters of the polygonal cells as cell centers.
	
	For each cell $K\in\Cs$, we denote $\EsK$ and $\EsInK$ the subsets of edges and internal edges belonging to $K$, and by $m_K$ the measure of the cell. The mesh size $h$ is the largest among all cells' diameters, $h\coloneqq \max_{K\in\Cs} \text{diam}(K)$, and characterizes the refinement of the mesh. For every internal edge, the diamond cell $\Delta_\edge$ is the quadrilateral with vertices given by the cell centers, $\x_K$ and $\x_L$, and the vertices of the edge. Denoting by $d_{\edge}\coloneqq |\x_L-\x_K|$ and $m_\sigma$ the measure of the edge, the measure of the diamond cell is equal to $m_{\Delta_\edge}=\frac{m_\edge d_\edge}{d}$, where $d$ stands for the space dimension. Finally, we denote by $d_{K,\edge}$ the Euclidean distance between the cell center $\x_K$ and the midpoint of the edge $\edge\in\EsK$, and by $\n_{K,\edge}$ the outward unit normal of the cell $K$ on the edge $\edge$.
	
	The finite volume methodology introduces two levels of discretization, on cells and edges.
	The first one is used to discretize scalar quantities whereas the second one for vectorial ones. To this end, we introduce three discrete inner product spaces $(\R^{\Cs},\langle \cdot, \cdot \rangle_{\Cs}), \, (\R^{\EsIn},\langle \cdot, \cdot \rangle_{\EsIn})$ and $(\Fs,\langle \cdot, \cdot \rangle_{\mb{F}_{\Cs}})$. The scalar products $\langle \cdot, \cdot \rangle_{\Cs}$ and $\langle \cdot, \cdot \rangle_{\EsIn}$ are defined as
	\[
	\begin{aligned}
	&\langle \cdot, \cdot \rangle_{\Cs}: (\boldsymbol{a},\boldsymbol{b})\in [\R^{\mathcal{T}}]^2 \mapsto \sum_{K\in\mathcal{T}} a_K b_K m_K \,, \\
	&\langle \cdot, \cdot \rangle_{\EsIn}: (\boldsymbol{u},\boldsymbol{v})\in [\R^{\Sigma}]^2 \mapsto \sum_{\sigma\in\Sigma} u_{\sigma} v_{\sigma} m_{\sigma} d_{\sigma} \,.
	\end{aligned}
	\]
	The space $\Fs$ is the space of conservative fluxes, it is defined by
	\begin{equation*}\label{eq:LJKO2space_fluxes}
	\Fs=\{\boldsymbol{F}=(F_{K,\sigma},F_{L,\sigma})_{\sigma\in\Sigma}\in\mb{R}^{2\Sigma}: F_{K,\sigma}+F_{L,\sigma}=0\} \,,
	\end{equation*}
	and its scalar product is
	\[
	\langle \cdot, \cdot \rangle_{\mb{F}_{\Cs}}: (\bs{F},\bs{G})\in [\mb{F}_{\Cs}]^2 \mapsto \sum_{\sigma\in\Sigma} (F_{K,\sigma} G_{K,\sigma}+F_{L,\sigma} G_{L,\sigma}) \frac{m_{\sigma}d_{\sigma}}{2} \,.
	\]
	Note that the space $\Fs$ is defined on internal edges only. This is sufficient, since we are dealing with no flux boundary value problems, and therefore we can neglect the flux variables on the boundary.
	We denote $F_{\sigma} = |F_{K,\sigma}| = |F_{L,\sigma}|$ the modulus of the flux on each internal edge $\sigma=K|L\in\EsIn$ and, by convention, $|\bF| = (F_{\sigma})_{\sigma\in\Sigma} \in \mb{R}^{\Sigma}$ and $|\bF|^2 = (F_{\sigma}^2)_{\sigma\in\Sigma} \in \mb{R}^{\Sigma}$, for $\bF\in\mb{F}_{\Cs}$.
	
	According to finite volumes, the discrete divergence operator $\mathrm{div}_{\Cs}: \mb{F}_\Cs \rightarrow \R^\Cs$ is defined in an integral sense as
	\[
	(\mathrm{div}_{\Cs} \bs{F})_K \coloneqq \mathrm{div}_{K} \bs{F} \coloneqq  \frac{1}{m_K} \sum_{\sigma\in\Sigma_K} F_{K,\sigma} m_{\sigma}\,,
	\]
	that is, for each cell, the discrete divergence is computed as the sum of the fluxes across its boundary. 
	The discrete gradient $\nabla_\Sigma: \mb{R}^{\Cs} \rightarrow \mb{F}_\Cs$ is defined by duality, requiring that $\langle\nabla_\Sigma \bs{a}, \bs{F} \rangle_{\mb{F}_\Cs}= - \langle \bs{a} , \mathrm{div}_\Cs \bs{F}\rangle_{\Cs}$, for all $\bs{a}\in\R^{\Cs}$ and $\bs{F}\in\Fs$. Then, it holds
	\[
	(\nabla_\Sigma \bs{a})_{K,\sigma} \coloneqq \mathrm{\nabla}_{K,\sigma} \bs{a} \coloneqq \frac{a_L -a_K}{d_\sigma} \,.
	\]
	Both the discrete divergence and gradient operators automatically inherit the zero flux boundary condition from the definition of $\Fs$.
	
	The space $(\R^{\EsIn},\langle \cdot, \cdot \rangle_{\EsIn})$ is introduced in order to match the two different discretizations on cells and edges. In order to reconstruct variables defined on cells to the edges, and vice-versa, we need two reconstruction operators.
	We use a centered reconstruction for the mobility in order to attain the second order accuracy in space.
	To this end, we use the weighted arithmetic average operator $\Lc_{\EsIn}:\R^{\Cs}\rightarrow\R^{\EsIn}$ and its adjoint $\Lc^*_{\EsIn}:\R^{\EsIn}\rightarrow\R^{\Cs}$ (with respect to the two scalar products):
	\begin{equation}\label{eq:reconstruction}
	(\Lc_\Sigma \bs{a})_{\sigma}\coloneqq \lambda_{K,\sigma} a_K +\lambda_{L,\sigma} a_L \,, \quad ( \Lc_\Sigma^* \bs{u})_K \coloneqq \sum_{\sigma\in\Sigma_K} \lambda_{K,\edge} u_{\sigma} \frac{m_{\sigma}d_{\sigma}}{m_K}\,,
	\end{equation}
	for $\bs{a}\in\R^{\Cs}$ and $\bs{u}\in\R^{\EsIn}$, with $\lambda_{K,\sigma}+\lambda_{L,\sigma}=1, \forall \edge=K|L\in\EsIn$. Two possible choices for the weights are $(\lambda_{K,\sigma},\lambda_{L,\sigma}) = (\frac{d_{K,\sigma}}{d_{\sigma}},\frac{d_{L,\sigma}}{d_{\sigma}})$ or $(\frac{1}{2},\frac{1}{2})$, both leading to second order accurate schemes in space \cite{natale2020FVCA}. The former choice is possible only if $\x_K\in K$, which may not be always the case for arbitrary admissible meshes.
	
	\begin{remark}
	The definition of the reconstruction operators and the choice of weights may be delicate in general for the discretization of dynamical optimal transport, depending on the discretization chosen for $\Om$. See \cite{gladbach2018scaling,natale2021computation} for details. Notice in particular that the choice $(\lambda_{K,\sigma},\lambda_{L,\sigma})=(\frac12,\frac12)$ may lead to convergence failure in very simple settings \cite[Section 5]{gladbach2018scaling}. Nevertheless, in the context of the discretization of Wasserstein gradient flows the definition of the reconstruction is more flexible, see \cite{cances2020LJKO,forkert2020evolutionary}.
	\end{remark}

	\subsection{Discrete ${\stackrel{.}{H}}{}^{-1}$ norm}
	
	As suggested in \cite{lavenant2018dynamical,erbar2020computation,natale2021computation}, a convenient choice for the time discretization of the Wasserstein distance \eqref{eq:w2dynamic} is to use a staggered time discretization for the velocity and the density on subintervals of the time interval $[0,1]$, and reconstruct the density on intermediate steps via arithmetic average. It has been shown numerically in \cite{cances2020LJKO,natale2020FVCA} that a single step discretization on the whole interval is sufficient in order to preserve the first-order accuracy of the JKO scheme \eqref{eq:jko}.
	Following the same ideas, here we approximate the Wasserstein distance between two measures $\mu,\nu\in\PcO$ as 
	\begin{equation}\label{eq:Hm1norm}
	\frac{1}{2}	W_2^2(\mu,\nu)
	\approx \sup_{\phi} \intO \phi(\mu-\nu)-\frac{1}{2}\intO \Big(\frac{\mu+\nu}{2}\Big) |\nabla \phi|^2 \,.
	\end{equation}
	Formula \eqref{eq:Hm1norm} is obtained by discretizing in one step problem \eqref{eq:w2dynamic} and by applying a duality result thanks to the change of variables $(\omega,v)\mapsto(\omega,\omega v)$. For more details on this construction see \cite{cances2020LJKO,natale2020FVCA}. This approximation consists in replacing the Wasserstein distance with the weighted dual norm $\frac{1}{2}||\mu-\nu||_{\dot{H}_{\frac{\mu+\nu}{2}}^{-1}}$. The choice of the arithmetic average of the two measures as weight is fundamental in order to achieve second order accuracy in time for the scheme we will propose in the following.
	
	Using the finite volume discretization introduced above we can provide a discrete analogous of the weighted norm.
	Given the discrete measures $\bmu,\bnu\in\R^{\Cs}_+$ and for any $\bh\in\R^{\Cs}$, the discrete counterpart of the weighted $\dot{H}^{-1}$ norm squared is
	\begin{equation}\label{eq:Adiss_dual}
	\Ac_{\Cs}\Big(\frac{\bmu+\bnu}{2};\bh\Big) \coloneqq \sup_{\bphi\in\R^{\Cs}} \langle \bh, \bphi \rangle_{\Cs}- \frac{1}{2} \Big\langle \Lc_\Sigma \Big(\frac{\bmu+\bnu}{2}\Big), |\grad_\EsIn \bphi|^2 \Big\rangle_{\EsIn} \,.
	\end{equation}
	A few remarks are in order about such a discretization.
	\begin{itemize}
	\item For any $\brho\in\R^{\Cs}_+$, the function $\Ac_{\Cs}(\brho;\cdot)$ is proper, convex and lower semi-continuous as supremum of convex and lower semi-continuous functions.
	\item The supremum is unbounded if the condition $\langle \bh , \one \rangle_{\Cs}=0$ is not satisfied. On other hand, if $\langle \bh , \one \rangle_{\Cs}=0$, there exists a maximizer $\bphi$, which is however not uniquely defined, since the function maximised in \eqref{eq:Hm1norm} is invariant with respect to addition of a global constant or perturbations sufficiently far from the support of $\bh$, $\bmu$, and $\bnu$.
	\item Setting $\bh = \bnu - \bmu$ in \eqref{eq:Adiss_dual}, with $\bmu$ and $\bnu$ being a discrete approximation of two measures $\mu$ and $\nu$, we obtain a discrete version of $W_2^2(\mu,\nu)/2$. In this case the optimal potential $\bphi$
	can be interpreted as a discrete counterpart of a continuous optimal potential $\phi$, satisfying the Hamilton-Jacobi equation on the time interval $[0,1]$, evaluated at time $1/2$.
	\item The total kinetic energy is discretized on the diamond cells. 
	Notice that due to the definition of the scalar product $\langle\cdot,\cdot\rangle_{\EsIn}$, the measure of each diamond cell is taken $m_{\edge} d_{\edge} = d m_{\Delta_{\edge}}$, i.e. $d$ times the actual measure. This is done in order to compensate for the unidirectional discretization, since each term $|\nabla_{K,\sigma} \bphi|$ is meant as an approximation of the quantity $|\nabla \phi \cdot\n_{K,\edge}|$, and have a consistent discretization. See \cite{natale2021computation} for more details on this construction.
	\end{itemize}

	\subsection{Discrete extrapolation}\label{sec:discreteextra}
	
	We now construct a discrete version of the extrapolation operator $\extra$ at time $\alpha$, by discretizing the procedure described in Section \ref{sec:collisions}, and in particular of equation \eqref{eq:varextra}. The proposed strategy requires three subsequent steps: i) compute the interpolation between the two measures; ii) integrate forward in time the optimal potential; and finally iii) solve a JKO step. 
	
	Let us consider two discrete densities $\bmu,\bnu\in\R^{\Cs}_+$ with the same total discrete mass $\langle \bmu, \one \rangle_{\Cs} =  \langle \bnu, \one \rangle_{\Cs}$.
	The first step requires to solve problem \eqref{eq:Adiss_dual} for $\bh=\bnu-\bmu$ in order to find an optimal potential ${\bphi}$, which approximates the continuous one, solution to the Hamilton-Jacobi equation \eqref{eq:hj}, at the midpoint of the time interval $[0,1]$.

	In the second step, we evolve the optimal potential according to the Hamilton-Jacobi equation until the final time $\alpha$, that is considering a temporal step of length $\frac{1}{2}+\beta=\frac{\alpha+\beta}{2}$. This can be done with an explicit Euler step as follows:
	\begin{equation}\label{eq:HJevolve2}
	\bphiab = {\bphi} - \frac{2}{\alpha+\beta}\frac{1}{2} \Lc^*_{\EsIn} |\grad_\EsIn {\bphi}|^2 \,.
	\end{equation}
	Note that we use the operator $\Lc_{\EsIn}^*$ to reconstruct the square of the gradient of the potential. However, as this step is not variational, it is not mandatory to use the adjoint of the reconstruction $\Lc_{\EsIn}$ and any other (second order) strategy can be adopted.
	
	Finally, for the third step, we approximate problem \eqref{eq:varextra} using again the discrete weighted $\dot{H}^{-1}$ norm. Specifically, we define a discrete extrapolation operator as a map $\extra^{\Cs}: \R^{\Cs}_+\times \R^{\Cs}_+\rightarrow \R^{\Cs}_+$ verifying
	\begin{equation}\label{eq:rhoabLJKO}
	\extra^{\Cs}(\bmu,\bnu) \in \argmin_{\brho\in\R^{\Cs}_+} \frac{1}{\alpha} \Ac_{\Cs} \Big(\frac{\brho+\bmu}{2};\bmu-\brho\Big) - \langle \bphiab, \brho \rangle_{\Cs} \,,
	\end{equation}
	for all $\bmu,\bnu\in \mathbb{R}^\Tc_+$ and where $\bphiab$ is given by equation \eqref{eq:HJevolve2}.
	Due to the definition of $\Ac_{\Cs}$, any solution $\brho$ satisfies $\langle \brho, \one \rangle_{\Cs} =  \langle \bnu, \one \rangle_{\Cs}$. However, since ${\bphi}$ is in general not unique, in order to specify a discrete extrapolation operator one needs to select a specific optimal potential for any $\bmu,\bnu\in \mathbb{R}^\Tc_+$.

	\subsection{A space-time discrete \scheme\ scheme}
	
	We can finally formulate our second order finite volume scheme. Consider a convex discrete energy function $\Ec_{\Cs}:\R^{\Cs}\rightarrow\R$ and the two initial densities $\brho_0,\brho_1\in\R^{\Cs}_+$, with the same total discrete mass.
	We define the subspace of discrete probability measures $\PCs\subset\R^{\Cs}$ as
	\[
	\PCs = \{ \brho\in\R^{\Cs}_+:  \langle \brho, \one \rangle_{\Cs} =  \langle \brho_0, \one \rangle_{\Cs} \} \,.
	\]
	For the time step $\tau>0$, we compute the sequence of densities $(\brho_{n})_{n\ge2}\subset\PCs$ defined by the following recursive scheme:
	\vspace{0.5em}
	\begin{equation}\label{eq:LJKO2th}
	\left\{
	\begin{aligned}
	&\brho^\alpha_{n-1}=\extra^{\Cs}(\brho_{n-2},\brho_{n-1})\,, \\[0.5em]
	&\brho_n\in\argmin_{\brho\in\R^{\Cs}_+} \frac{1}{\tau(1-\beta)} \Ac_{\Cs} \Big(\frac{\brho+\brhoab_{n-1}}{2};\brhoab_{n-1}-\brho\Big) + \Ec_{\Cs}(\brho)\,.
	\end{aligned}
	\right.
	\end{equation}
	The LJKO step in \eqref{eq:LJKO2th} is a well posed convex optimization problem. Uniqueness of the solution at each step is guaranteed if $\Ec_{\Cs}$ is strictly convex. Moreover, due to the definition of $\Ac_{\Cs}$, any solution $\brho$ belongs to $\mb{P}_{\Cs}$.

	\begin{remark}[Efficient implementation via the interior method]
	Problem \eqref{eq:rhoabLJKO} and the LJKO step in \eqref{eq:LJKO2th} can be solved efficiently thanks to an interior point algorithm, as suggested in \cite{natale2020FVCA} (see also \cite{natale2021computation,facca2022efficient}).
	This implies that the density will be always strictly greater than zero, up to the tolerance set for the solver.
	Hence, one can compute the solution ${\bphi}$, required to define $\extra^\Tc$, solving directly the linear system given by the optimality condition of problem \eqref{eq:Adiss_dual}:
	\begin{equation}\label{eq:Adiss_dual_OC}
	\bnu-\bmu + \mathrm{div}_{\Cs} (\Lc \Big(\frac{\bmu+\bnu}{2}\Big) \odot \grad {\bphi}) = 0 \,,
	\end{equation}
	where $\odot$ denotes the component-wise product, which has then a unique solution defined up to a global additive constant.
	\end{remark}
	
	\subsection{Other implementations}\label{ssec:others_order2}
	
	We now propose a discrete version of the extrapolation-based version of the VIM scheme \eqref{eq:VIM2} and the BDF2 scheme \eqref{eq:BDF2_mat} within the same TPFA finite volume setting introduced above.
	We will study these numerically in Section \ref{sssec:1dconvergence} by comparing their solutions to the solutions provided by scheme \eqref{eq:LJKO2th} on one-dimensional test cases.
	
	Our formulation of the VIM scheme  \eqref{eq:VIM2} requires solving a JKO step with time step $\frac{\tau}{2}$ and then computing the $2$-extrapolation. Using the tools introduced above, in the discrete setting this can be formulated as follows. Given the initial density $\brho_0\in\PCs$ and a time step $\tau>0$, construct the sequence of densities $(\brho_{n})_{n\ge1}\subset\PCs$ by solving at each step $n$
	\begin{equation}\label{eq:VIMdiscrete}
	\left\{
	\begin{aligned}
	&\brho_{n-\frac{1}{2}} \in\argmin_{\brho\in\R^{\Cs}_+} \frac{2}{\tau} \Ac_{\Cs} \Big(\frac{\brho+\brho_{n-1}}{2};\brho_{n-1}-\brho\Big) + \Ec_{\Cs}(\brho) \,,\\
	&\brho_n={\sf E}^{\Cs}_2(\brho_{n-1},\brho_{n-\frac{1}{2}})\,.
	\end{aligned}
	\right.
	\end{equation}
	As before, the discrete LJKO steps can be computed thanks to an interior point algorithm.
	From a computational point of view, this scheme is cheaper to compute than \eqref{eq:LJKO2th}, as in this case the value of the optimal potential in the discrete weighted $\dot{H}^{-1}$ norm from $\brho_{n-1}$ to $\brho_{n-\frac{1}{2}}$ is already known from the LJKO step and does not need to be computed. However, in the next section, we will show numerically that the solutions produced by the VIM scheme \eqref{eq:VIMdiscrete} are much more oscillatory than those obtained with the \scheme\ scheme.
	
	We can also propose a naive discretization of the BDF2 scheme \eqref{eq:BDF2_mat} by replacing the Wasserstein distances with discrete weighted $\dot{H}^{-1}$ norms. Consider two initial conditions $\brho_0,\brho_1\in\PCs$ and the time parameter $\tau>0$. At each step $n$, compute $\brho_{n}$ as solution to
	\begin{equation}\label{eq:bdf2discrete}
	\inf_{\brho\in\R^{\Cs}_+} \frac{\alpha}{(1-\beta)\tau} \Ac_{\Cs} \Big(\frac{\brho+\brho_{n-1}}{2};\brho_{n-1}-\brho\Big) -\frac{\beta}{(1-\beta)\tau} \Ac_{\Cs} \Big(\frac{\brho+\brho_{n-2}}{2};\brho_{n-2}-\brho\Big) + \Ec_{\Cs}(\brho) \,.
	\end{equation}
	Problem \eqref{eq:bdf2discrete} is not a convex optimization problem. Notice that it is not even bounded from below in general. Indeed, the function $\Ac_{\Cs} (\frac{\brho+\brho_{n-2}}{2};\brho_{n-2}-\brho)$ is not bounded from above if the density $\brho_{n-2}$ is not supported everywhere.
	We can nevertheless try to compute stationary points of the objective function in \eqref{eq:bdf2discrete} using again an interior point algorithm. Despite not being a robust and completely meaningful strategy, in some cases it is possible to solve the problem, which enables us to compare it to our implementation.
	
	\begin{remark}
	In one dimension, as pointed out in Remark \ref{rem:metricextra_L2}, both the metric extrapolation \eqref{eq:metricextra} and the BDF2 scheme \eqref{eq:BDF2_mat} can be recast as convex optimization problems. In this case it is possible then to design effective discretizations for these (as originally done in \cite{Matthes2019bdf2}). Nevertheless, this approach requires, at least in the Eulerian framework, to be able to switch between discrete densities and discrete quantile functions, and it does not appear obvious how to achieve this while preserving the second order accuracy of the space discretization.
	\end{remark}
	
	\section{Numerical validation of the \scheme\ scheme}\label{sec:numerics}
	
	The objective of this section is to validate our numerical scheme \eqref{eq:LJKO2th}.
	We will first show qualitatively its behavior with simple one-dimensional examples and compare it to the schemes \eqref{eq:VIMdiscrete} and \eqref{eq:bdf2discrete}. We then show that all these three approaches lead to a second order accurate discretization in both time and space.
	We consider for these purposes two specific problems that exhibit a gradient flow structure in the Wasserstein space: the Fokker-Planck equation we presented in Section \ref{ssec:LJKO2convergence} and the porous medium equation. This latter writes
	\begin{equation}\label{eq:porouseqLJKO2}
	\partial_t \rhoe = \Delta \rhoe^{\delta} + \mathrm{div} (\rhoe \nabla V) \,,
	\end{equation}
	and it is a Wasserstein gradient flow with respect to the energy
	\begin{equation}\label{eq:porousenLJKO2}
	\mathcal{E}(\rho) = \int_{\Omega} \frac{1}{\delta-1} \rho^{\delta} + \rho V \,,
	\end{equation}
	for a given $\delta>1$ and with $V\in W^{1,\infty}(\Omega)$ a Lipschitz continuous exterior potential \cite{otto2001geometry}.
	The energy functionals \eqref{eq:FokkerPlanckEnergy} and \eqref{eq:porousenLJKO2} are both of the form $\Ec(\rho)=\intO E(\rho) \d x$ for a strictly convex function $E:\mathbb{R}_+ \rightarrow \mathbb{R}$. They can be straightforwardly discretized as $\Ec_{\Cs} = \sum_{K\in\Cs} E(\rho_K) m_K$.
	Finally, we will test scheme \eqref{eq:LJKO2th} on a more challenging application in order to show its flexibility and robustness, that is an incompressible immiscible multiphase flow in a porous medium.
	
	We remark that when two initial conditions $\brho_0,\brho_1$ are needed, we compute first $\brho_1$ from $\brho_0$ via an LJKO step:
	\[
	\brho_1=\argmin_{\brho\in\R^{\Cs}_+} \frac{1}{\tau} \Ac_{\Cs} \Big(\frac{\brho+\brho_0}{2};\brho_0-\brho\Big) + \Ec_{\Cs}(\brho) \,.
	\]
	In the ODE setting, computing the second initial condition via a first step of implicit Euler scheme ensures the overall second order accuracy \cite{deuflhard2002scientific}. This strategy reveals to be numerically effective also in this setting.

	\subsection{Comparison of the three approaches}\label{ssec:comparison}
	
	We compare the three different approaches on simple one dimensional tests for the diffusion equation and the porous medium equation.
	For both system we set $\Om = [0,1]$, discretized in subintervals of equal length $m_K=0.02$.
	
	We first consider the diffusion equation, which is problem \eqref{eq:FokkerPlanckLJKO2} with zero external potential $V$.  We take as initial condition
	\[
	\rho_0 = \exp\Big(-50\Big(x-\frac{1}{2}\Big)^2\Big) \,,
	\]
	which we discretize as $\brho_0= (\rho_0(\x_K))_{K\in\Cs}$, and the time step $\tau=0.01$. In Figure \ref{fig:1dtest_FP}, we show the density obtained with the three schemes at three different times.
	Using the VIM scheme \eqref{eq:VIMdiscrete}, spurious oscillations appear in the solution and these persist along the integration in time. Such oscillations can be explained as the result of the interaction of the extrapolation step, causing the mass to exit the domain, and the boundary conditions, forcing the mass to stay within $\Omega$. Neither the \scheme\ scheme \eqref{eq:LJKO2th} nor the BDF2 scheme \eqref{eq:bdf2discrete}  suffer from this problem. However, notice that in both cases the dynamics slightly differ from pure diffusion due to the presence of bumps in the solution.
	
	\begin{figure}[t!]
	\centering
	\includegraphics[trim={3.6cm 8.5cm 3.6cm 8cm},clip,width=0.3\textwidth]{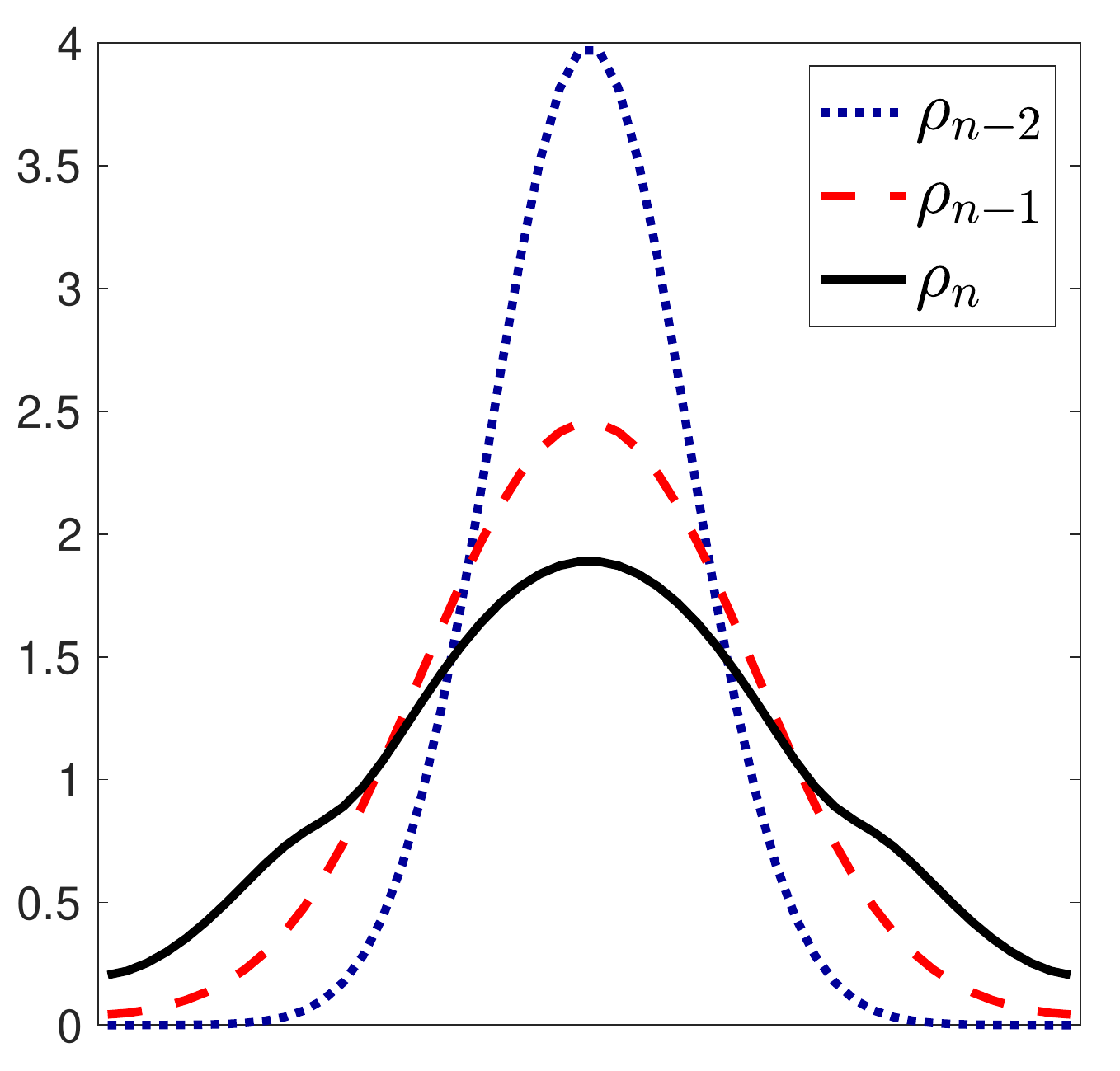}
	\includegraphics[trim={3.6cm 8.5cm 3.6cm 8cm},clip,width=0.3\textwidth]{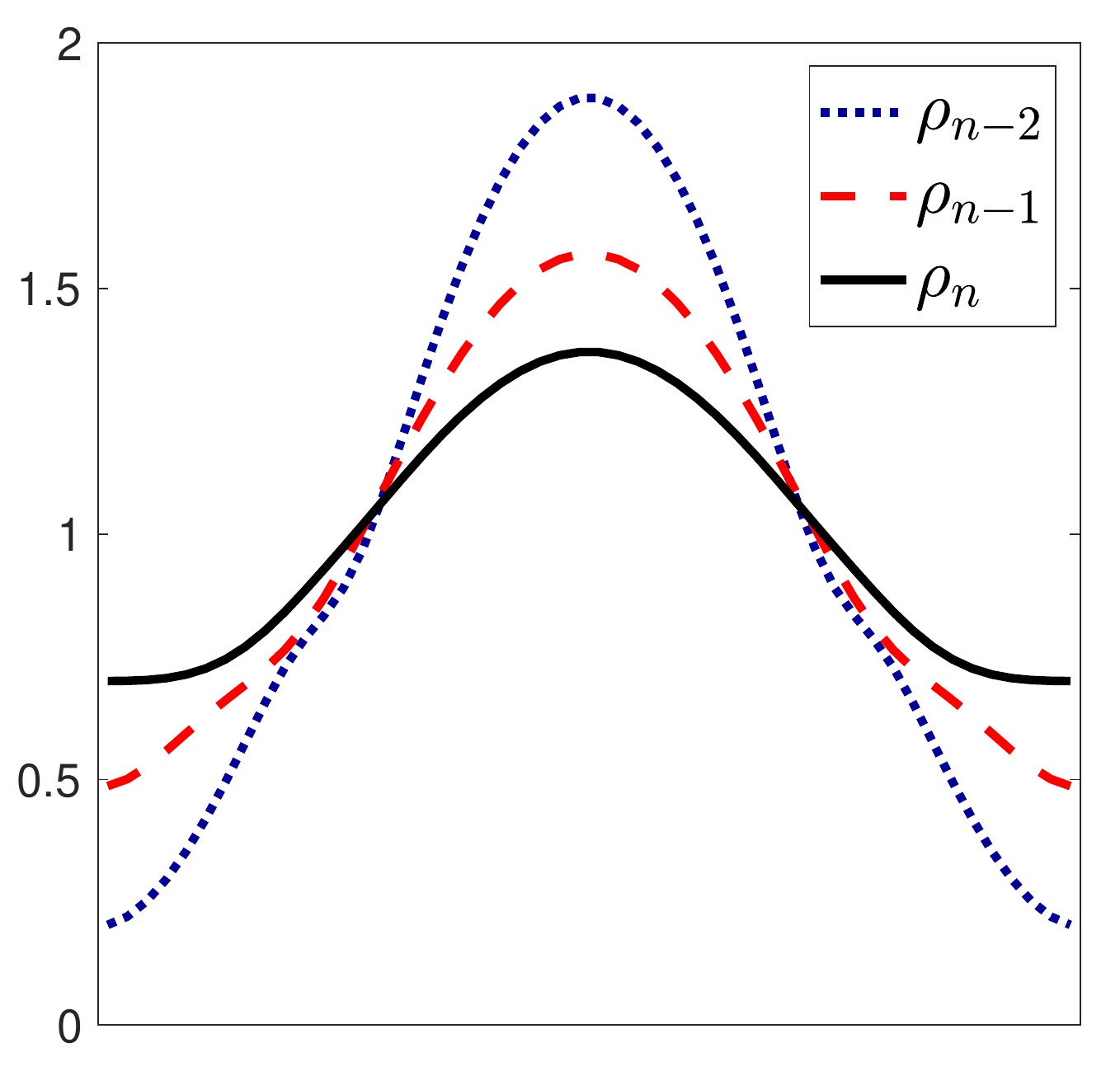}
	\includegraphics[trim={3.6cm 8.5cm 3.6cm 8cm},clip,width=0.3\textwidth]{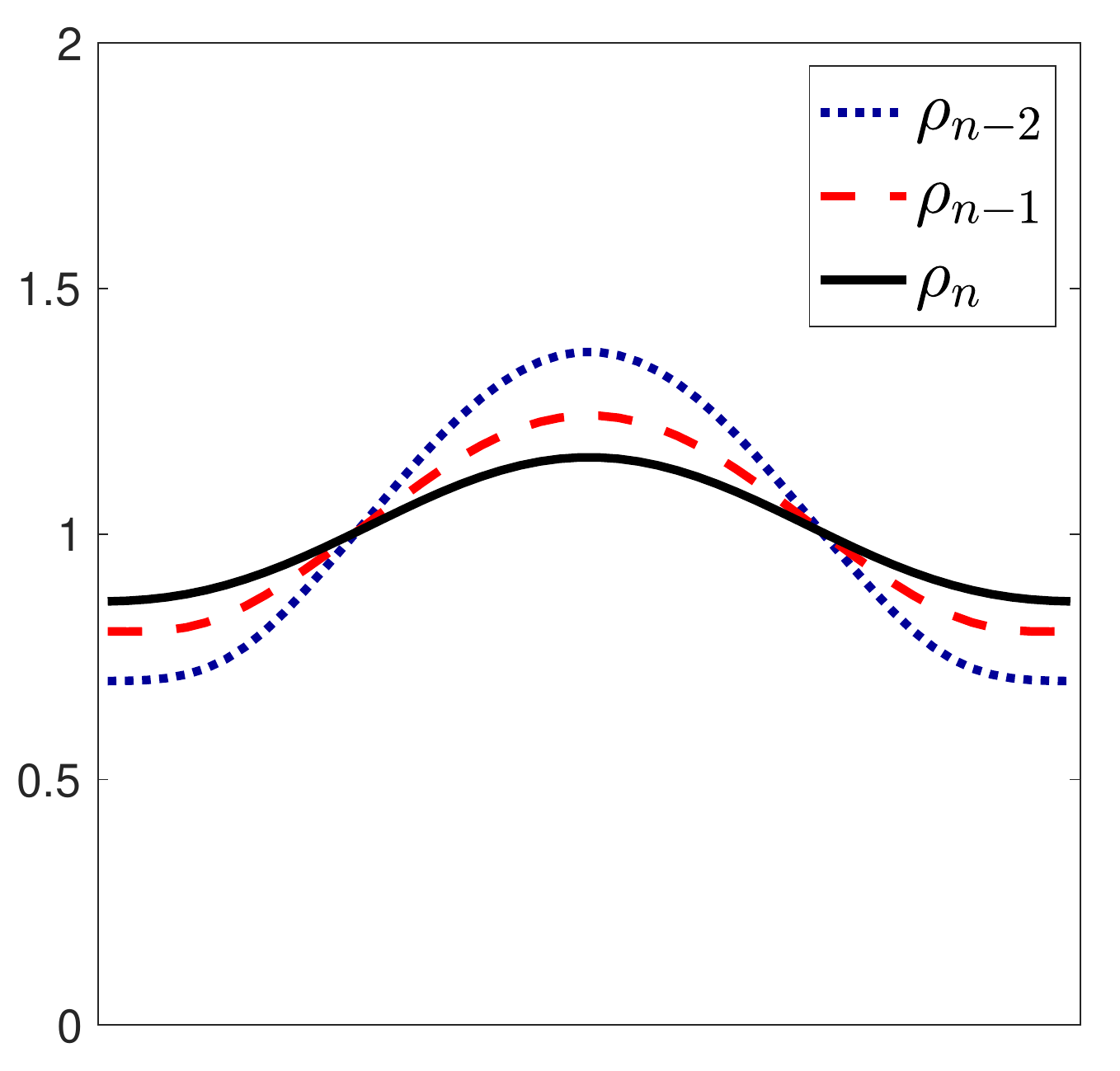} \\
	\includegraphics[trim={3.6cm 8.5cm 3.6cm 8cm},clip,width=0.3\textwidth]{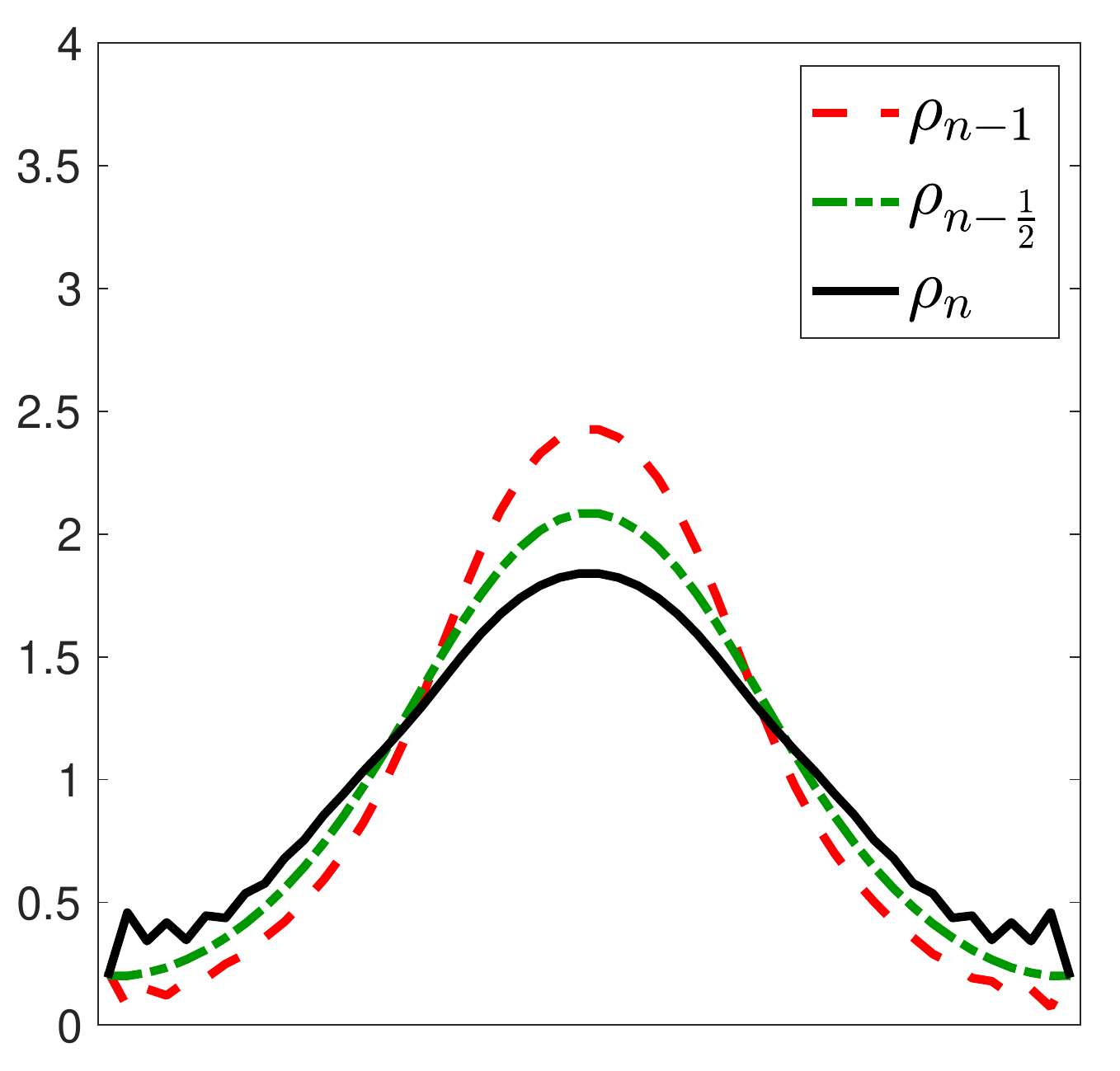}
	\includegraphics[trim={3.6cm 8.5cm 3.6cm 8cm},clip,width=0.3\textwidth]{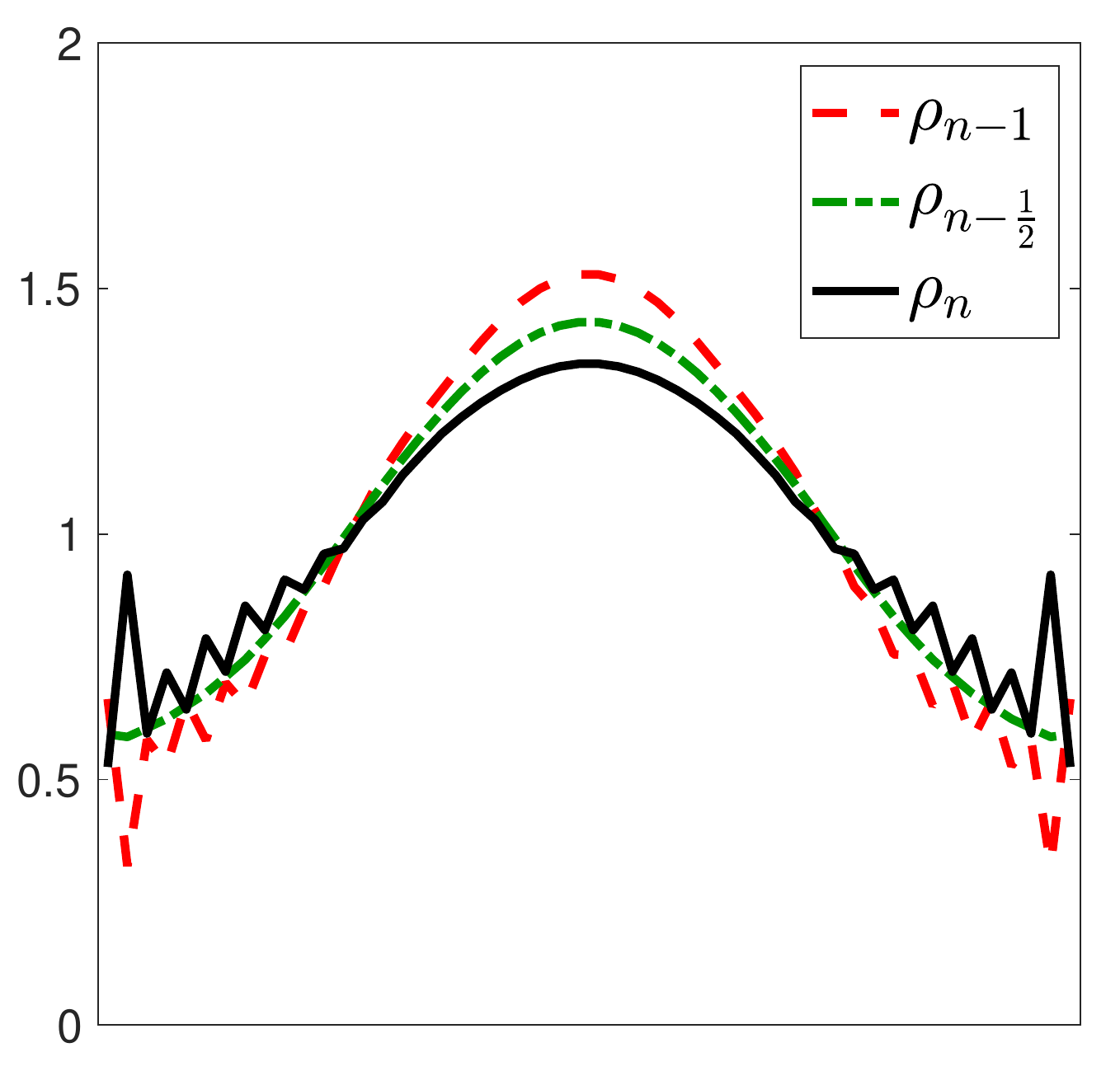}
	\includegraphics[trim={3.6cm 8.5cm 3.6cm 8cm},clip,width=0.3\textwidth]{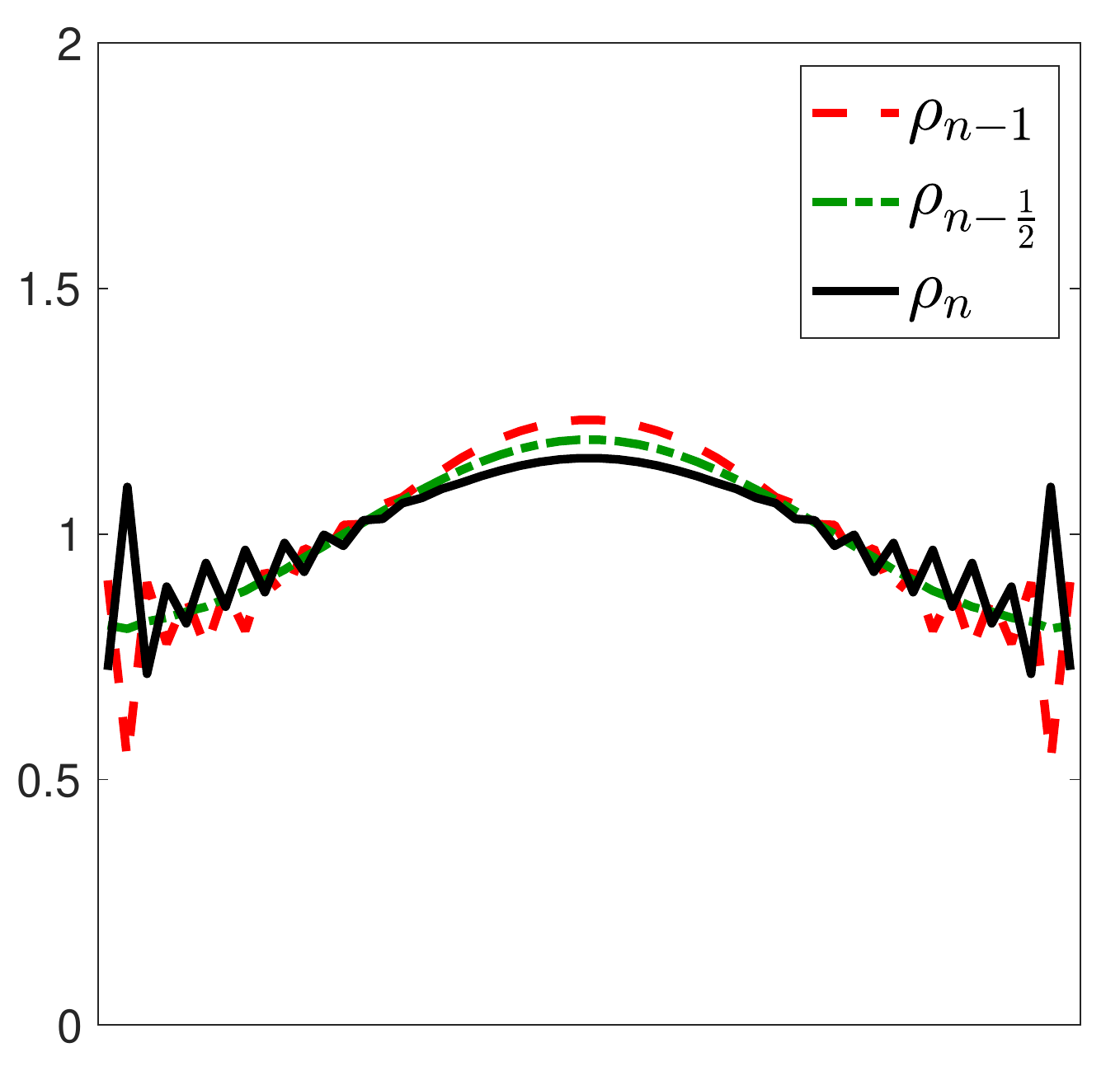} \\
	\includegraphics[trim={3.6cm 8.5cm 3.6cm 8cm},clip,width=0.3\textwidth]{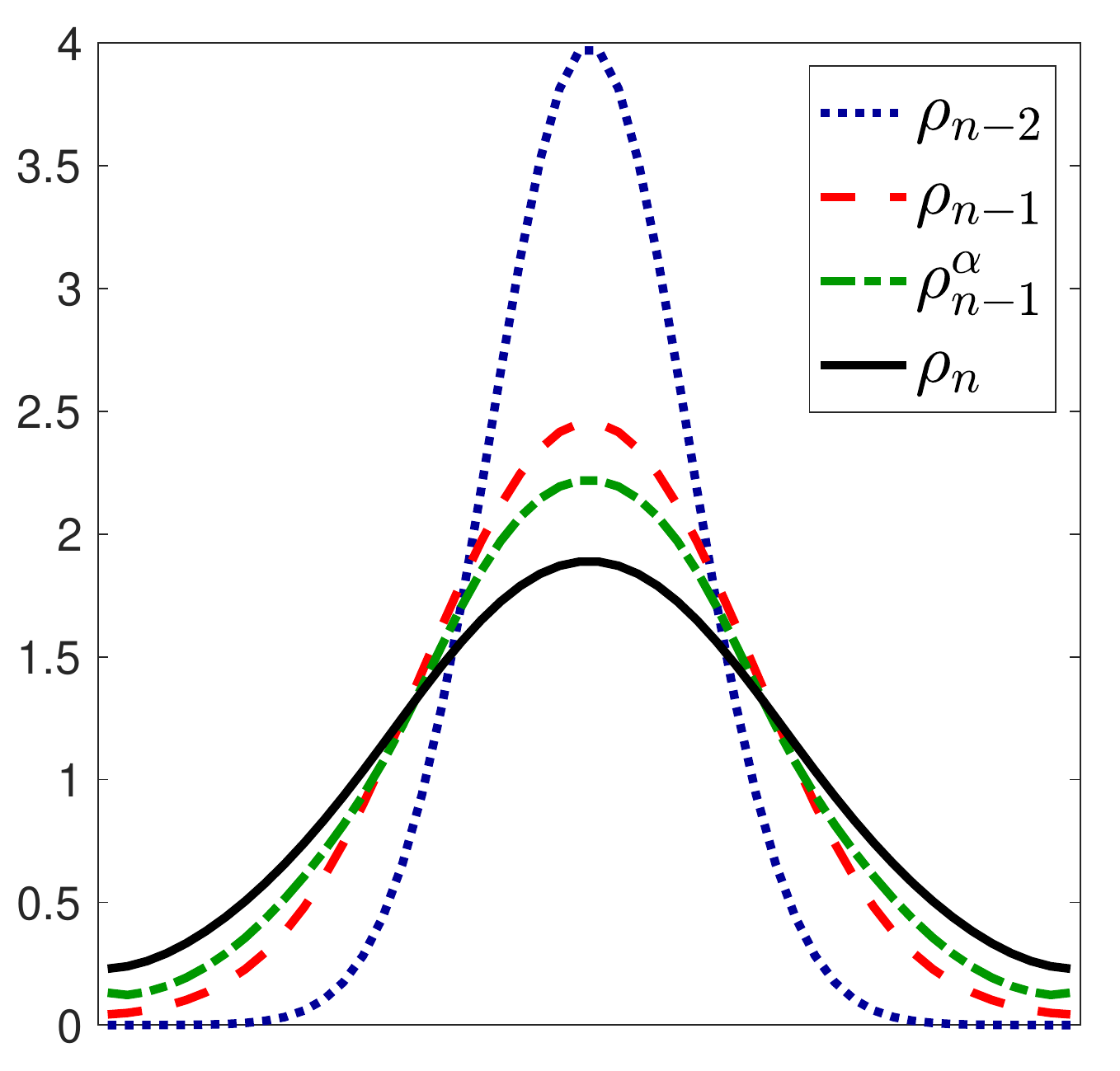}
	\includegraphics[trim={3.6cm 8.5cm 3.6cm 8cm},clip,width=0.3\textwidth]{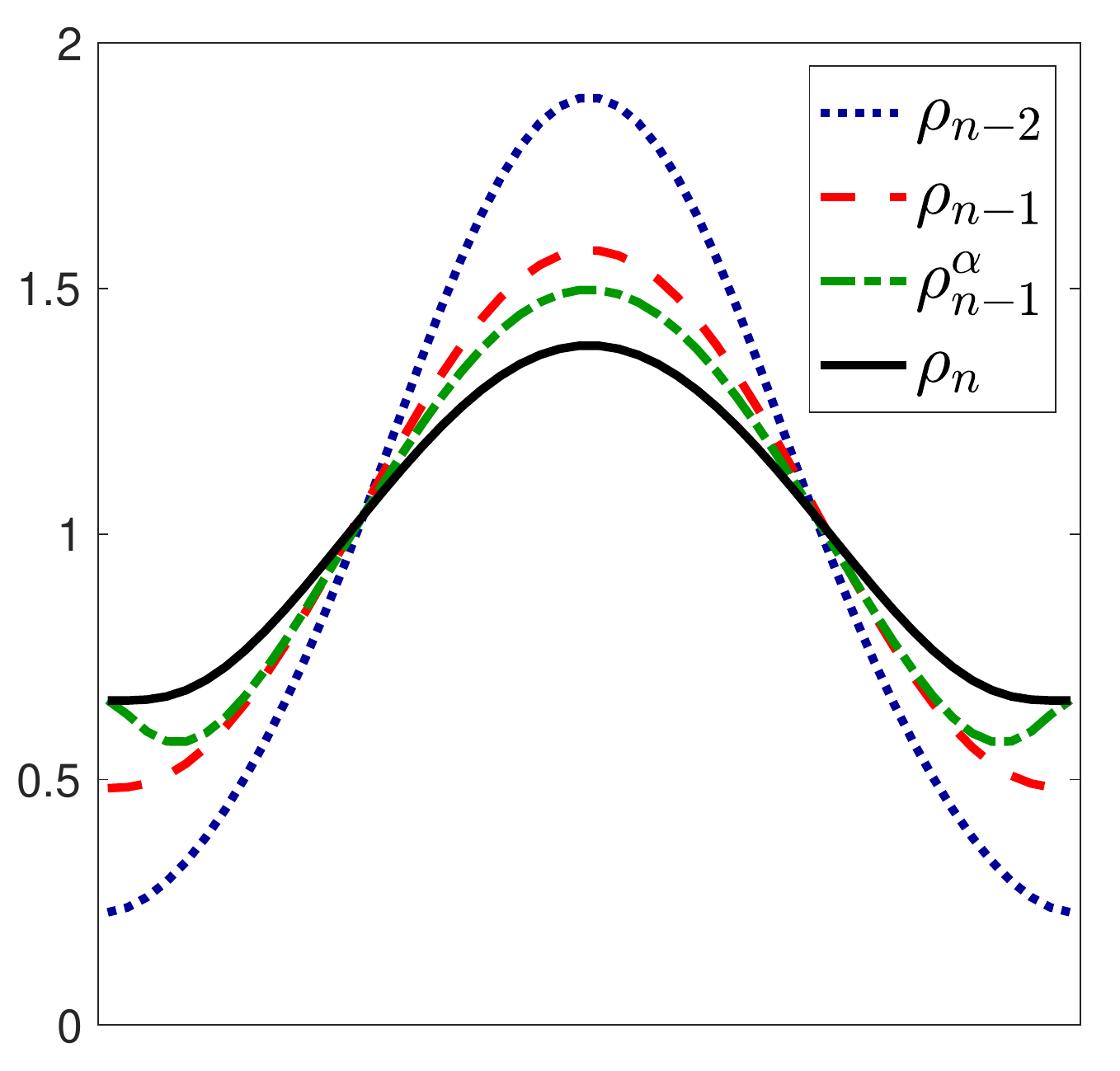}
	\includegraphics[trim={3.6cm 8.5cm 3.6cm 8cm},clip,width=0.3\textwidth]{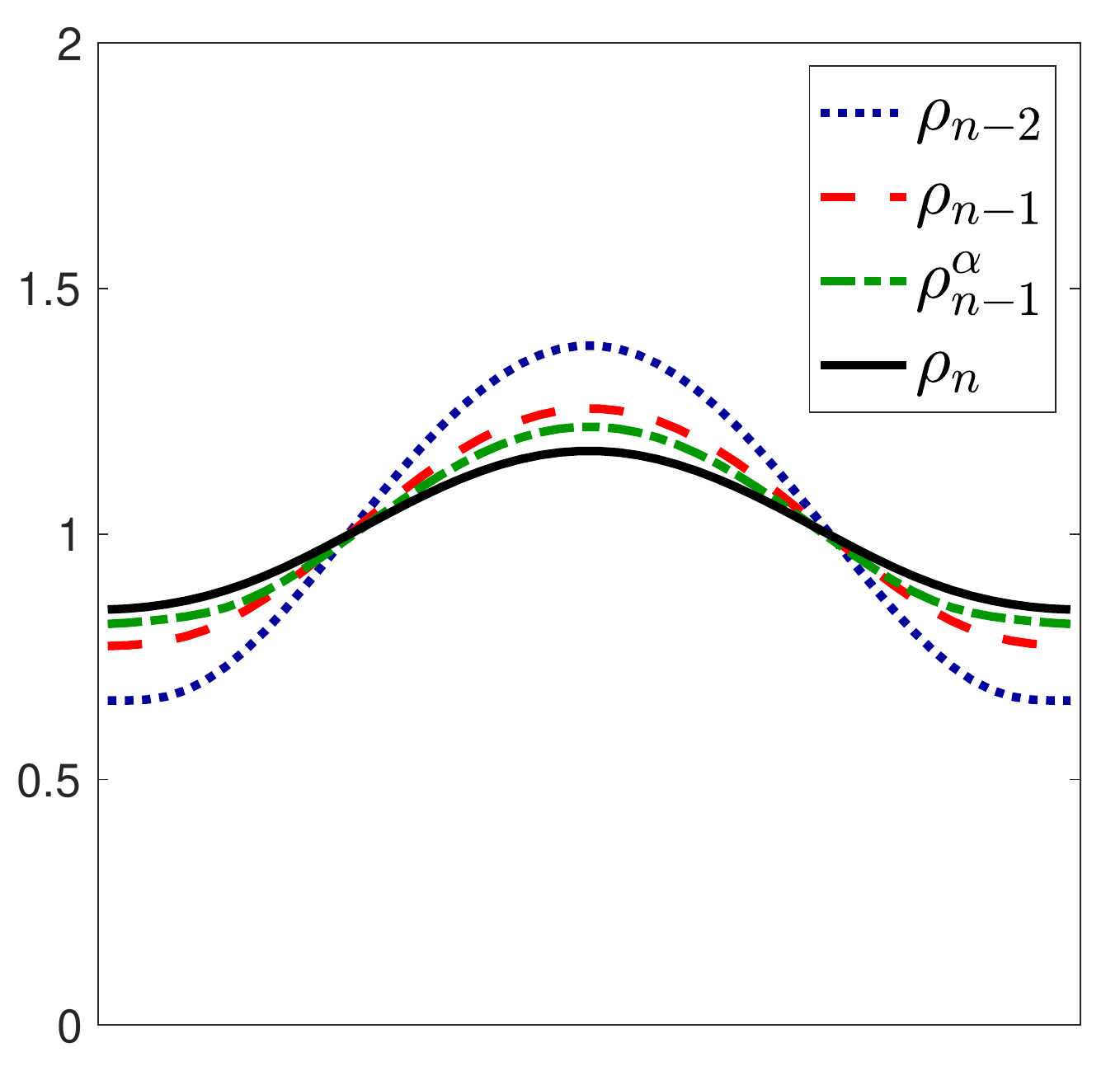}
	\caption{Comparison between the three schemes for the diffusion equation. From top to bottom, the BDF2 scheme \eqref{eq:bdf2discrete}, the VIM scheme \eqref{eq:VIMdiscrete} and the \scheme\ scheme \eqref{eq:LJKO2th}. From left to right, three different time steps: $t=0.02, 0.04, 0.06$.}
	\label{fig:1dtest_FP}
	\end{figure}
	
	Consider now the porous medium equation \eqref{eq:porouseqLJKO2} with $\delta=2$ and the external potential $V(x)=-x$, which causes the mass to drift towards the positive direction. We take as initial condition
	\[
	\rho_0(x) = \indf_{x\le\frac{3}{10}} \,,
	\]
	discretized again as $\brho_0= (\rho_0(\x_K))_{K\in\Cs}$, and the time step $\tau=0.002$. In this case, the naive implementation we proposed for the BDF2 scheme does not converge, which is not surprising since the objective function in \eqref{eq:bdf2discrete} is unbounded from below. The results for the VIM scheme \eqref{eq:VIMdiscrete} and the \scheme\ scheme \eqref{eq:LJKO2th} are shown in Figure \ref{fig:1dtest_PM}. Again, the VIM scheme is unstable whereas the \scheme\ scheme controls and smooths the oscillations generated by the extrapolation step. Note that in this case the oscillations are due to the compact support of the density and the explicit integration in time of the Hamilton-Jacobi equation: in the extrapolation step the mass cannot flow outside the support, which acts then like a boundary.

	Finally, we observe that, as in the continuous setting, we cannot expect any regularity on the measure obtained after the extrapolation, and the JKO step is the only source of regularity for both the \scheme\ and the VIM scheme.
	One may argue that the two schemes perform the same operations up to a temporal shift, which should contradict the different behavior shown in Figure \ref{fig:1dtest_FP}. However, notice that scheme \eqref{eq:LJKO2th} performs a smaller extrapolation and a bigger JKO step with respect to scheme \eqref{eq:VIMdiscrete}. Furthermore, in \eqref{eq:LJKO2th} one needs to compute an extrapolation between two minimizers of the JKO step, whereas in \eqref{eq:VIMdiscrete} the extrapolation is between an extrapolated measure and a JKO minimizer. 
	
	\begin{figure}[t!]
	\centering
	\includegraphics[trim={3.6cm 8.5cm 3.6cm 8cm},clip,width=0.3\textwidth]{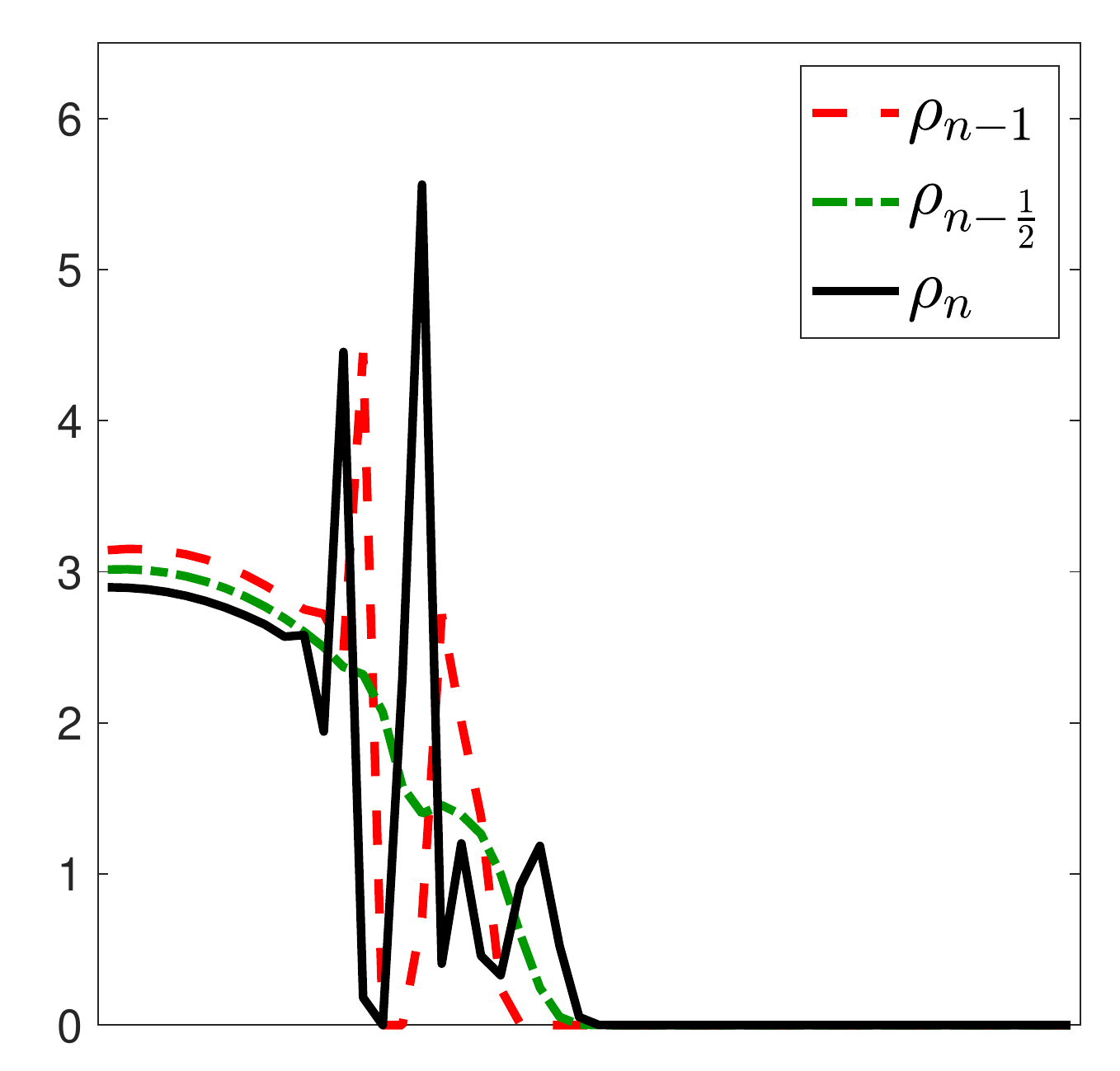}
	\includegraphics[trim={3.6cm 8.5cm 3.6cm 8cm},clip,width=0.3\textwidth]{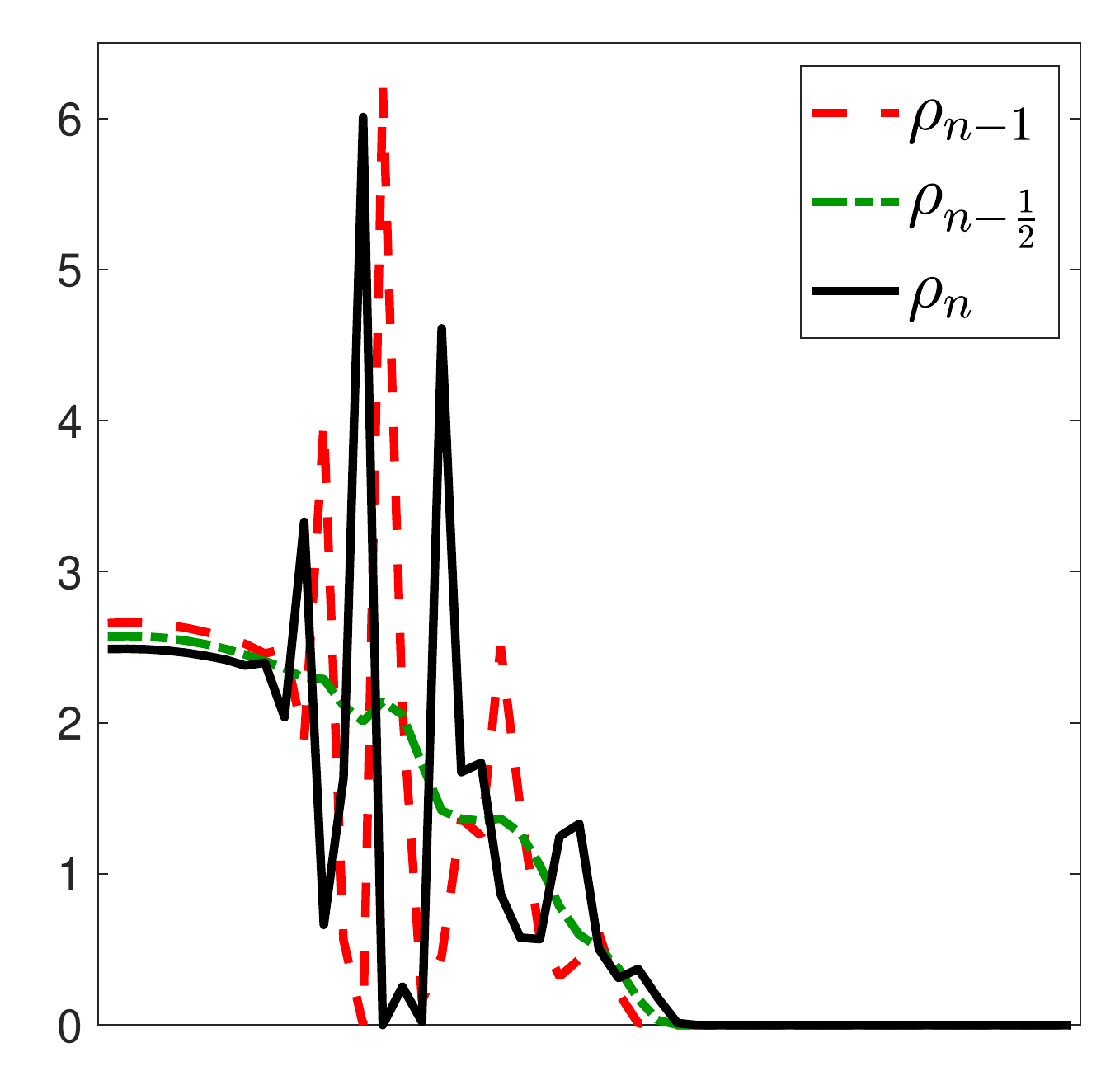}
	\includegraphics[trim={3.6cm 8.5cm 3.6cm 8cm},clip,width=0.3\textwidth]{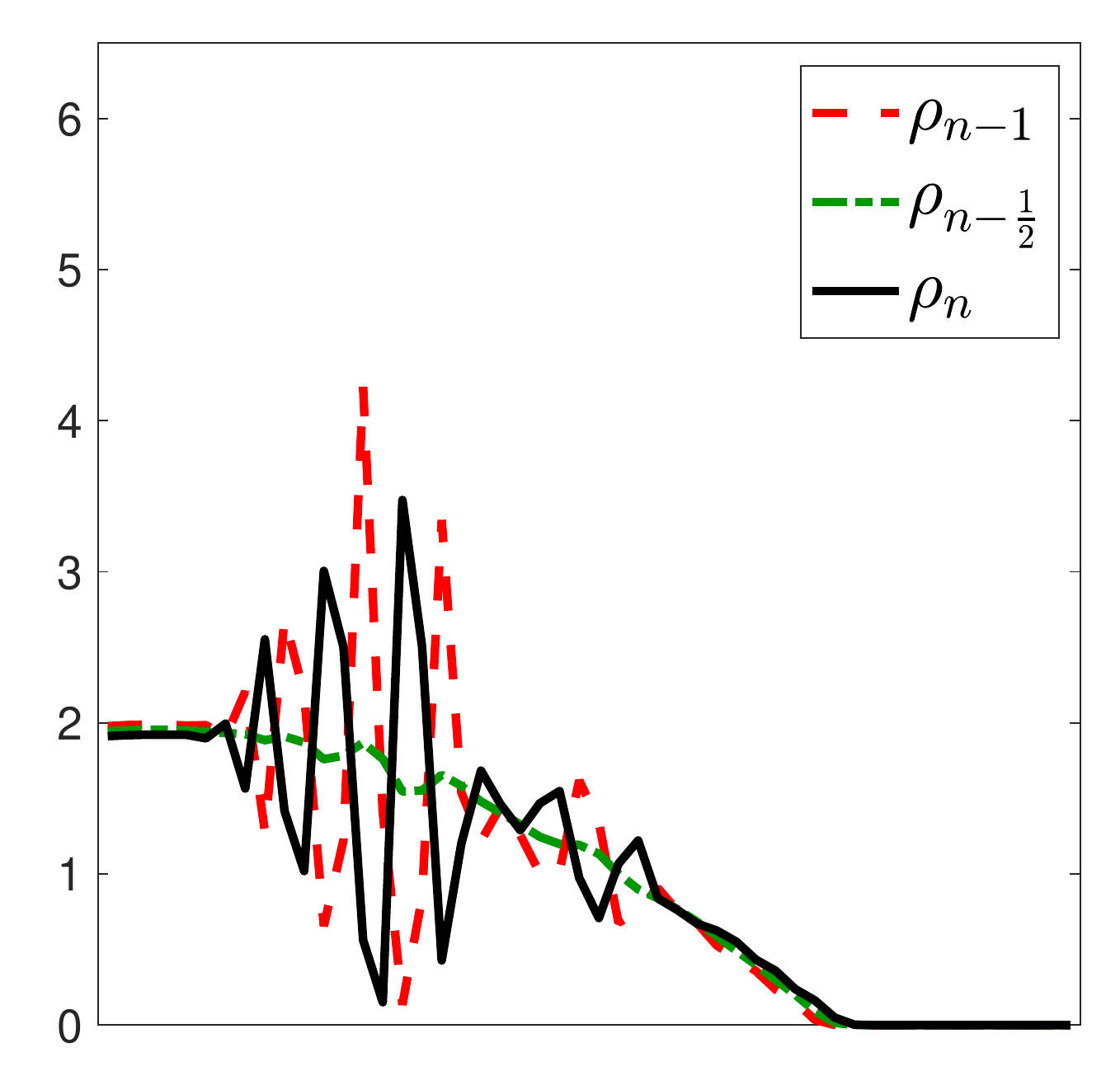} \\
	\includegraphics[trim={3.6cm 8.5cm 3.6cm 8cm},clip,width=0.3\textwidth]{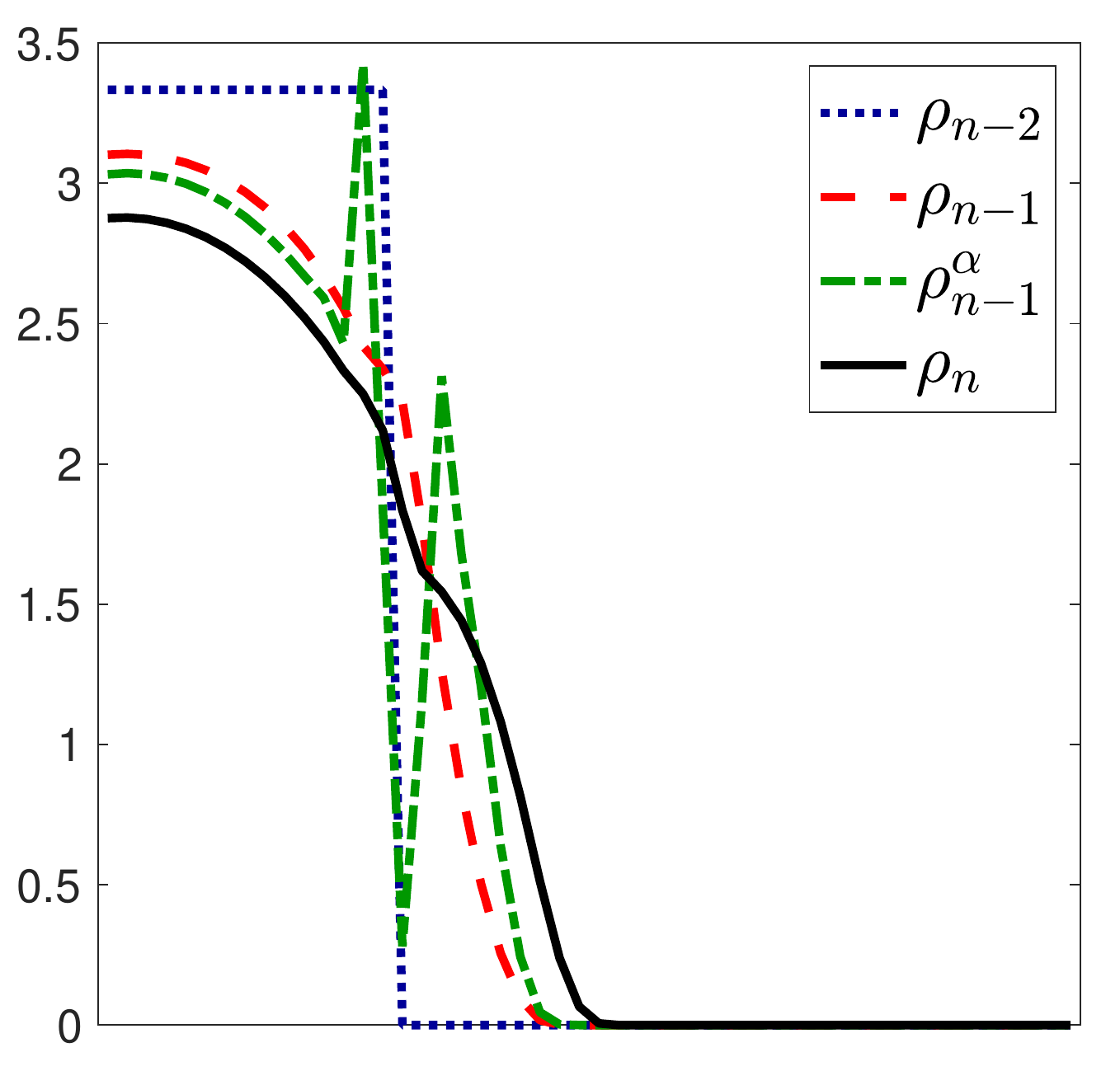}
	\includegraphics[trim={3.6cm 8.5cm 3.6cm 8cm},clip,width=0.3\textwidth]{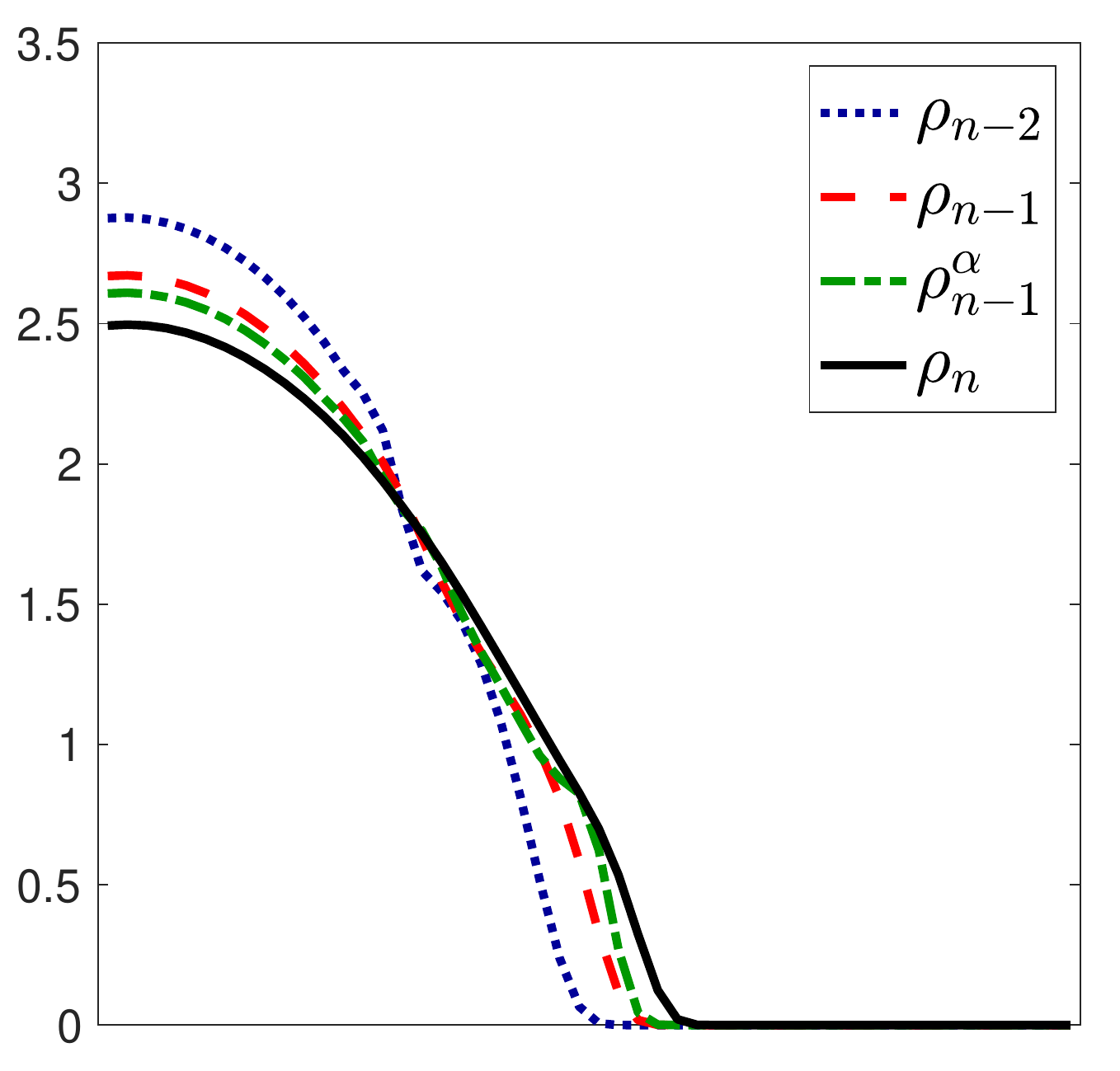}
	\includegraphics[trim={3.6cm 8.5cm 3.6cm 8cm},clip,width=0.3\textwidth]{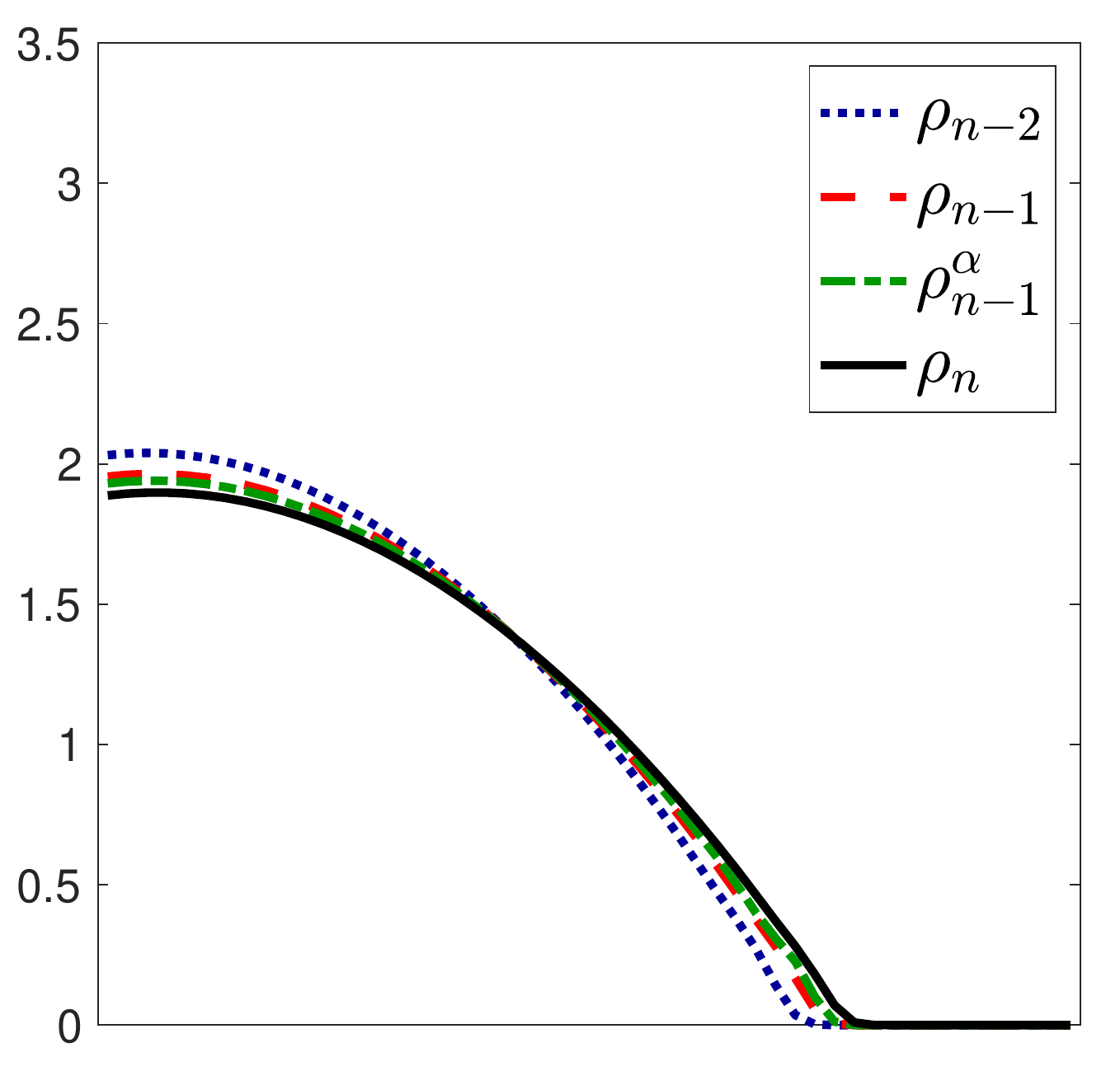} 
	\caption{Comparison between the VIM scheme \eqref{eq:VIMdiscrete} (top row) and the \scheme\ scheme \eqref{eq:LJKO2th} (bottom row) for the porous medium equation. The BDF2 scheme \eqref{eq:bdf2discrete} does not converge in this case. From left to right, three different time steps: $t=0.004, 0.008, 0.020$.}
	\label{fig:1dtest_PM}
	\end{figure}
	
	\subsection{Convergence tests}
	
	We now compare the three schemes in terms of order of convergence with respect to an exact one-dimensional solution of the Fokker-Planck equation  \eqref{eq:FokkerPlanckLJKO2}. For the \scheme\ scheme \eqref{eq:LJKO2th}, we will also perform two dimensional tests using the porous medium equation \eqref{eq:porouseqLJKO2}. For all tests, we consider a sequence of meshes $\left(\Cs_m,\Es_m, {(\x_K)}_{K\in\Cs_m}\right)$ with decreasing meshsize $h_m$ and a sequence of decreasing time steps $\tau_m$ such that $\frac{h_{m+1}}{h_m}=\frac{\tau_{m+1}}{\tau_m}$.
	We solve the discrete problem for each couple $(h_m,\tau_m)$ and evaluate the convergence with respect to the discrete $L^1((0,T);L^1(\Omega))$ error:
	\[
	\epsilon_m = \sum_{n} \tau \sum_{K \in \Cs_m} |\rho_{K,n}-\rhoe(\x_K,n \tau)| m_K \,.
	\]
	We compute the rate of convergence as: \[\frac{\log(\epsilon_{m-1})-\log(\epsilon_m)}{\log(\tau_{m-1})-\log(\tau_m )} \,.
	\]
	
	\subsubsection{One-dimensional tests}\label{sssec:1dconvergence}
	
	On the domain $\Omega=[0,1]$ and for the external potential $V(x)=-gx$, we consider the following exact solution to the Fokker-Planck equation \eqref{eq:FokkerPlanckLJKO2}:
	\begin{equation}\label{eq:FPsol}
	\rhoe(t,x) = \exp\left(-\Big(\pi^2+\frac{g^2}{4}\Big) t+\frac{g}{2}x\right)\left(\pi \cos(\pi x)+\frac{g}{2}\sin(\pi x)\right)+\pi \exp\Big(g\Big(x-\frac{1}{2}\Big)\Big) .
	\end{equation}
	We consider the value $g=1$. For each mesh $\left(\Cs_m,\Es_m, {(\x_K)}_{K\in\Cs_m}\right)$ and time step $\tau_m$, we compute then the discrete solution using the three schemes, starting from the initial condition $\brho_0= (\rhoe(0,\x_K))_{K\in\Cs}$.
	The results are presented in Table \ref{tab:errors1d_order2}. Both the BDF2 and the \scheme\ schemes are second order accurate, whereas the order of convergence is less than one for the VIM scheme. This is due to the presence of oscillations in the solutions obtained with the VIM scheme, which are however only present at the beginning of the time interval $[0,0.25]$. Repeating the test on the interval $[0.05,0.25]$, the convergence significantly improves and attains second order accuracy as well.
	
	\begin{table}
	\centering
	\setlength{\tabcolsep}{5pt}
	\caption{Errors and convergence rates for the three schemes for the Fokker-Planck equation in one dimension. Integration time $[0,0.25]$ for the first three cases, $[0.05,0.25]$ for the last one.}
	\label{tab:errors1d_order2}
	\begin{tabular}{cccccccccc}
	\toprule
	& & \multicolumn{2}{c}{BDF2 \eqref{eq:bdf2discrete}} & \multicolumn{2}{c}{\scheme\ \eqref{eq:LJKO2th}} & \multicolumn{2}{c}{VIM \eqref{eq:VIMdiscrete}} & \multicolumn{2}{c}{VIM \eqref{eq:VIMdiscrete}} \\
	\toprule
	$h_m$ & $\tau_m$ & $\epsilon_m$ & rate & $\epsilon_m$ & rate & $\epsilon_m$ & rate & $\epsilon_m$ & rate \\
	\midrule
	0.100 & 0.050 & 2.091e-02 & /   & 2.217e-02 & /  & 5.895e-02 & / & 4.667e-03 & /  \\
	0.050 & 0.025 & 6.376e-03 & 1.713 & 7.016e-03 & 1.660 & 3.615e-02 & 0.706 & 1.024e-03 & 2.188 \\ 
	0.025 & 0.013 & 1.791e-03 & 1.832 & 2.044e-03 & 1.779 & 2.294e-02 & 0.656 & 2.517e-04 & 2.025  \\ 
	0.013 & 0.006 & 4.849e-04 & 1.885 & 5.653e-04 & 1.854 & 1.468e-02 & 0.644 & 6.264e-05 & 2.007 \\ 
	0.006 & 0.003 & 1.280e-04 & 1.922 & 1.508e-04 & 1.906 & 1.234e-02 & 0.251 & 1.562e-05 & 2.003 \\ 
	0.003 & 0.002 & 3.324e-05 & 1.945 & 3.933e-05 & 1.939 & 9.983e-03 & 0.306 & 3.901e-06 & 2.002 \\
	\bottomrule
	\end{tabular}
	\end{table}
	
	\subsubsection{Two-dimensional tests} We now estimate the order of convergence of the \scheme\ scheme on two-dimensional test cases. Here, we set $\Omega= [0,1]^2$ and use the same sequence of grids that have been used in \cite{cances2020LJKO,natale2020FVCA}, which allows for a direct comparison of the results therein.
	
	We repeat first the test on the Fokker-Planck equation in two dimensions using the same solution \eqref{eq:FPsol} on the domain $\Om=[0,1]^2$. The results are shown in Table \ref{tab:errors2d_FP_LJKO2} and confirm the second order accuracy of the scheme.
	
	\begin{table}[t]
	\centering
	\caption{Errors and convergence rate for the \scheme\ scheme \eqref{eq:LJKO2th} for the Fokker-Planck equation in two dimensions.}
	\label{tab:errors2d_FP_LJKO2}
	\begin{tabular}{cccc}
	\toprule
	$h_m$ & $\tau_m$ & $\epsilon_m$ & rate \\
	\midrule
	0.2986 & 0.0500  & 2.111e-02 & / \\ 
	0.1493 & 0.0250 &  6.800e-03 & 1.634 \\ 
	0.0747 & 0.0125 &  2.017e-03 & 1.754 \\ 
	0.0373 & 0.0063 &  5.669e-04 & 1.831 \\ 
	0.0187 & 0.0031 &  1.535e-04 & 1.884 \\ 
	\bottomrule
	\end{tabular}
	\end{table}

	We also perform a convergence test with respect to an explicit solution of the porous medium equation \eqref{eq:porouseqLJKO2} with zero exterior potential $V$. This equation admits a solution called Barenblatt profile \cite{otto2001geometry}:
	\begin{equation}\label{eq:porous_sol}
	\rhoe(t,{x}) = \frac{1}{t^{d \lambda}}\Big(\frac{\delta-1}{\delta}\Big)^{\frac{1}{\delta-1}} \max\Big(M-\frac{\lambda}{2}\Big|\frac{x-x_0}{t^\lambda}\Big|^2,0\Big)^{\frac{1}{\delta-1}} \,,
	\end{equation}
	where $\lambda=\frac{1}{d(\delta-1)+2}$, $d$ standing for the space dimension, and $x_0$ is the point where the mass is centered. The parameter $M$ can be chosen to fix the total mass. The value
	\[
	M=\Big(\frac{\delta}{\delta-1}\Big)^{-\frac{1}{\delta}} \Big(\frac{\lambda \delta}{2\pi(\delta-1)}\Big)^\frac{\delta-1}{\delta}
	\]
	sets it equal to one.
	The function \eqref{eq:porous_sol} solves \eqref{eq:porouseqLJKO2} on the domain $\Om=[0,1]^d$, with $\x_0$ in the interior of $\Om$, starting from $t_0>0$ and for a sufficiently small time horizon $T$, such that the mass does not reach the boundary of the domain.
	We consider the two-dimensional case and $x_0=(0.5,0.5)$.
	We solve the problem for $\delta=2,3,4$, with initial condition $\brho_0= (\rhoe(t_0,\x_K))_{K\in\Cs}$, starting respectively from $t_0=10^{-4},10^{-5},10^{-6}$ and up to time $T=t_0+10^{-3}$.
	The results are presented in Table \ref{tab:errors2d_PM}. The convergence profile is not clean, probably due to the low precision of the discretization in space. We can nevertheless notice that in the case $\delta=2$ the rate of convergence is approaching order two with refinement. In the cases $\delta=3,4$, where the solution is less regular, the order tends to $1.5$.
	
	\begin{table}[t]
	\centering
	\caption{Errors and convergence rates for the \scheme\ scheme \eqref{eq:LJKO2th} for the porous medium equation.}
	\label{tab:errors2d_PM}
	\begin{tabular}{cccccccc}
	\toprule
	& & \multicolumn{2}{c}{$\delta=2$} & \multicolumn{2}{c}{$\delta=3$} & \multicolumn{2}{c}{$\delta=4$} \\
	\toprule
	$h_m$ & $\tau_m$ & $\epsilon_m$ & rate & $\epsilon_m$ & rate & $\epsilon_m$ & rate \\
	\midrule
	0.2986 & 2.000e-04 & 5.139e-04 & / & 7.515e-04 & / & 9.537e-04 & / \\  
	0.1493 & 1.000e-04 & 1.999e-04 & 1.363 & 2.780e-04 & 1.435 & 3.085e-04 & 1.628 \\ 
	0.0747 & 5.000e-05 & 6.429e-05 & 1.636 & 4.630e-05 & 2.586 & 1.103e-04 & 1.485 \\ 
	0.0373 & 2.500e-05 & 1.471e-05 & 2.127 & 2.903e-05 & 0.674 & 3.847e-05 & 1.519 \\ 
	0.0187 & 1.250e-05 & 4.129e-06 & 1.833 & 7.521e-06 & 1.949 & 1.340e-05 & 1.522 \\ 
	\bottomrule
	\end{tabular}
	\end{table}

	\subsection{Incompressible immiscible multiphase flows in porous media}\label{ssec:multiphase}

	Incompressible immiscible multiphase flows in porous media can be described as Wasserstein gradient flows, as shown in \cite{cances2017multiphase}. We recall quickly the model problem in a simplified way.
	In the porous medium $\Omega$, $N+1$ phases are flowing and we denote by $\boldsymbol{s}=(s_0,...,s_N)$ the saturations of each phase, i.e. the portion of volume occupied by each phase in each point. The evolution of each saturation obeys the following equations:
	\begin{equation}\label{eq:multiphase}
	\left\{
	\begin{aligned}
	&\frac{\partial s_i}{\partial t} + \mathrm{div} (s_i v_i) = 0 \,,\\
	&v_i= - \frac{1}{\mu_i} (\nabla p_i - \rho_i g) \,, \\
	&p_i-p_0 = \pi_i(\boldsymbol{s},x) \, ,
	\end{aligned}
	\right.
	\end{equation}
	$i\in \{0,...,N\}$ for the first two equations, $i\in \{1,...,N\}$ for the third one, plus the total saturation condition $\sum_{i=0}^{N} s_i(t,x) = 1$ and the no-flux boundary conditions. The densities $\rho_i$ and the viscosities $\mu_i$, both constant in the whole domain, are characteristic of each phase. In \eqref{eq:multiphase} the porosity of the medium is considered constant and neglected. The term $\rho_i g$ reflects the influence of the potential energy on the motion ($g$ is the gravitational acceleration), but other types of potential energy could be considered. The model is completed specifying the $N$ capillary pressure relations, described by the functions $\pi_i$.
	
	We introduce the probability spaces
	\[
	\mathcal{P}_i = \Big\{ s_i \in \Pc(\Omega): s_i(\Om) = c_i \Big\}, \quad i\in \{0,...,N\},
	\]
	with the constant $c_i$ denoting the total mass of each phase. Each space $\mathcal{P}_i$ is endowed with the following quadratic Wasserstein distance,
	\[
	W_{2,i}^2(s_i^1,s_i^2) = \min_{\gamma\in \Pi(s_i^1,s_i^2)} \int \mu_i |x-y|^2 \d \gamma(x,y) \,,
	\]
	for $s_i^1,s_i^2\in\mathcal{P}_i$ and we can define the global quadratic Wasserstein distance $\boldsymbol{W}_2$ on $\boldsymbol{\mathcal{P}}:= \mathcal{P}_0 \times ... \times \mathcal{P}_N$ by setting
	\[
	\boldsymbol{W}_2^2(\boldsymbol{s}^1,\boldsymbol{s}^2) = \sum_{i=0}^{N} W_{2,i}^2(s_i^1,s_i^2), \quad \forall \boldsymbol{s}^1, \boldsymbol{s}^2 \in \boldsymbol{\mathcal{P}}.
	\]
	Problem \eqref{eq:multiphase} can then be represented as the gradient flow in the space $\boldsymbol{\mathcal{P}}$ with respect to the (strictly convex) energy functional
	\begin{equation}\label{eq:Fmultiphase}
	\mathcal{E}(\boldsymbol{s}) = \int_{\Omega} \boldsymbol{\Psi} \cdot \boldsymbol{s} + \int_{\Omega} \Pi(\boldsymbol{s},x) + i_{\boldsymbol{\mathcal{S}}}(\boldsymbol{s}) \,,
	\end{equation}
	where $\boldsymbol{\Psi}=(\Psi_0,\ldots,\Psi_N)$ is the exterior gravitational potential given by
	\[
	\Psi_i(x) = -\rho_i g \cdot x, \quad \forall x \in \Omega \,,
	\]
	$\Pi(s,x)$ is a strictly convex potential such that
	\[
	\pi_i(\boldsymbol{s},x) = \frac{\partial \Pi(\boldsymbol{s},x)}{\partial s_i}, \quad i \in \{1,...,N\},
	\]
	and $i_{\boldsymbol{\mathcal{S}}}$ is the indicator function of the set
	\[
	\boldsymbol{\mathcal{S}}=\left\{ \boldsymbol{s} \in \boldsymbol{\mathcal{P}}: \sum_{i=0}^{N} s_i(x)=1,  \text{for a.e.} \,\, x\in \Omega \right\}.
	\]
	
	When applying the \scheme\ scheme to such gradient flow, the extrapolation may be taken in each space $\mathcal{P}_i$ independently, i.e. we define the extrapolation in the space $\boldsymbol{\mathcal{P}}$ as
	\[
	\extra(\boldsymbol{s}^1,\boldsymbol{s}^2)\coloneqq (\extra(s_i^1,s_i^2))_{i=0}^N \,,
	\]
	for all $\boldsymbol{s}^1,\boldsymbol{s}^2\in\boldsymbol{\mathcal{P}}$. This does not guarantee at all that at each step $n$ of the scheme the extrapolation is a feasible point for $\mathcal{E}(\boldsymbol{s})$, that is $\extra(\boldsymbol{s}^1,\boldsymbol{s}^2)\notin\boldsymbol{\mathcal{S}}$ in general even though $\boldsymbol{s}^1,\boldsymbol{s}^2\in\boldsymbol{\mathcal{S}}$. Nevertheless, the resulting scheme is well defined as well as the numerical approach \eqref{eq:LJKO2th}. In our implementation, we linearize each Wasserstein distances independently. The energy functional can be discretized straightforwardly.
	
	As a specific instance of problem \eqref{eq:multiphase}, we consider a two-phase flow, where water ($s_0$) and oil ($s_1$) are competing in the porous medium. We choose the classical Brooks-Corey capillary pressure model,
	\[
	p_1-p_0 = \pi_1(s_1) = \lambda (1-s_1)^{-\frac{1}{2}}\,,
	\]
	and take $g$ acting along the negative direction of the $y$ axis, $|g|=9.81$. We set the model parameter $\lambda=0.05$. The densities and the viscosities of the two fluids are, respectively, $\rho_0=1$ and $\rho_1 = 0.87$, $\mu_0=1$ and $\mu_1=100$. We consider a non convex domain $\Om$ shaped as an hourglass and set an initial condition where the water is distributed uniformly in a layer in the upper part, whereas the oil takes the complementary space (see Figure \ref{fig:twophaseflow_0}). The evolution of the oil saturation $s_1$ is presented in Figure \ref{fig:twophaseflow}.
	
	\begin{figure}[t!]
	\centering
	\begin{minipage}{0.89\textwidth}
	\centering
	\subfloat[][$t=0$\label{fig:twophaseflow_0}]{\includegraphics[trim={4.7cm 7.8cm 4cm 7.25cm},clip,width=0.29\textwidth]{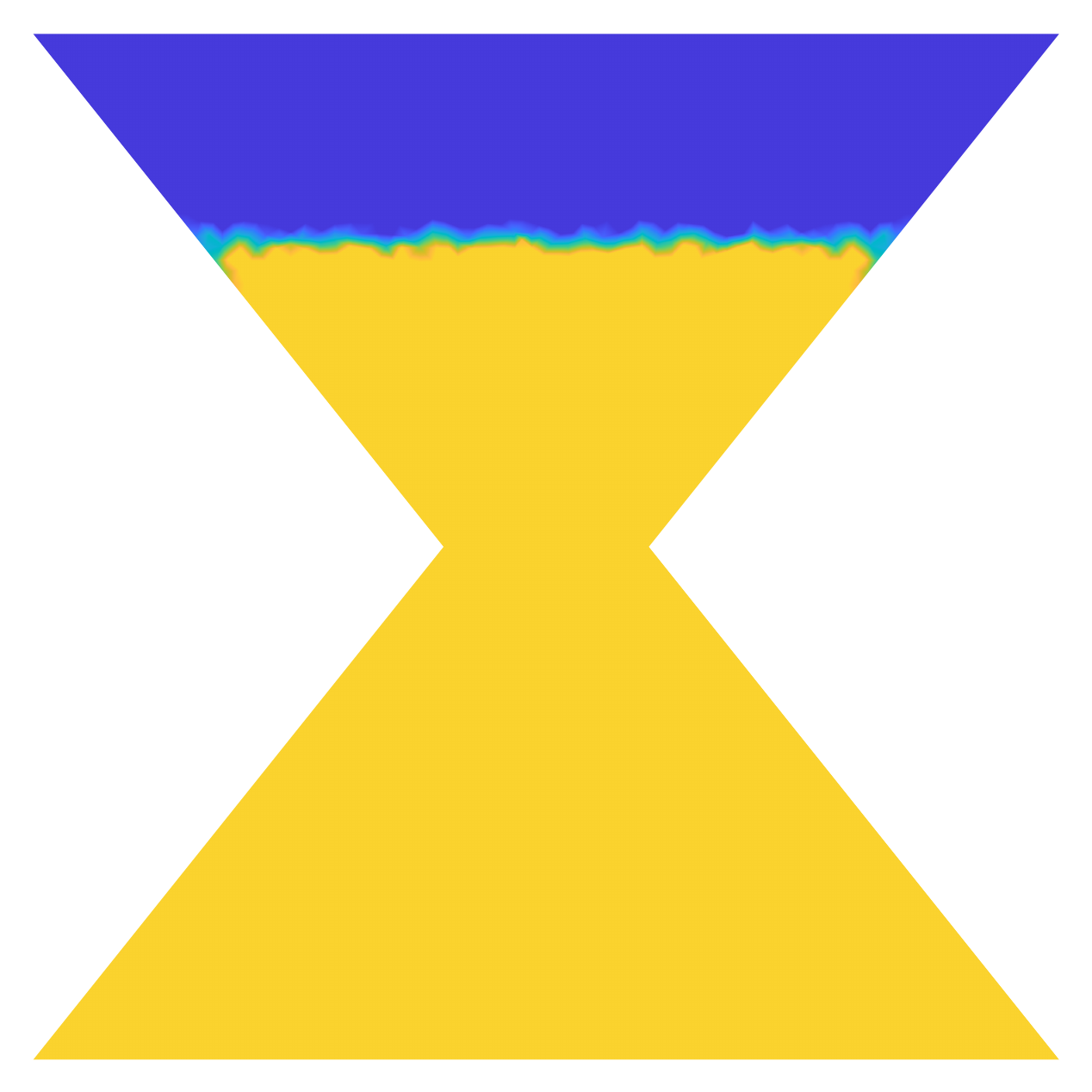}} \,
	\subfloat[][$t=8$]{\includegraphics[trim={4.7cm 7.8cm 4cm 7.25cm},clip,width=0.29\textwidth]{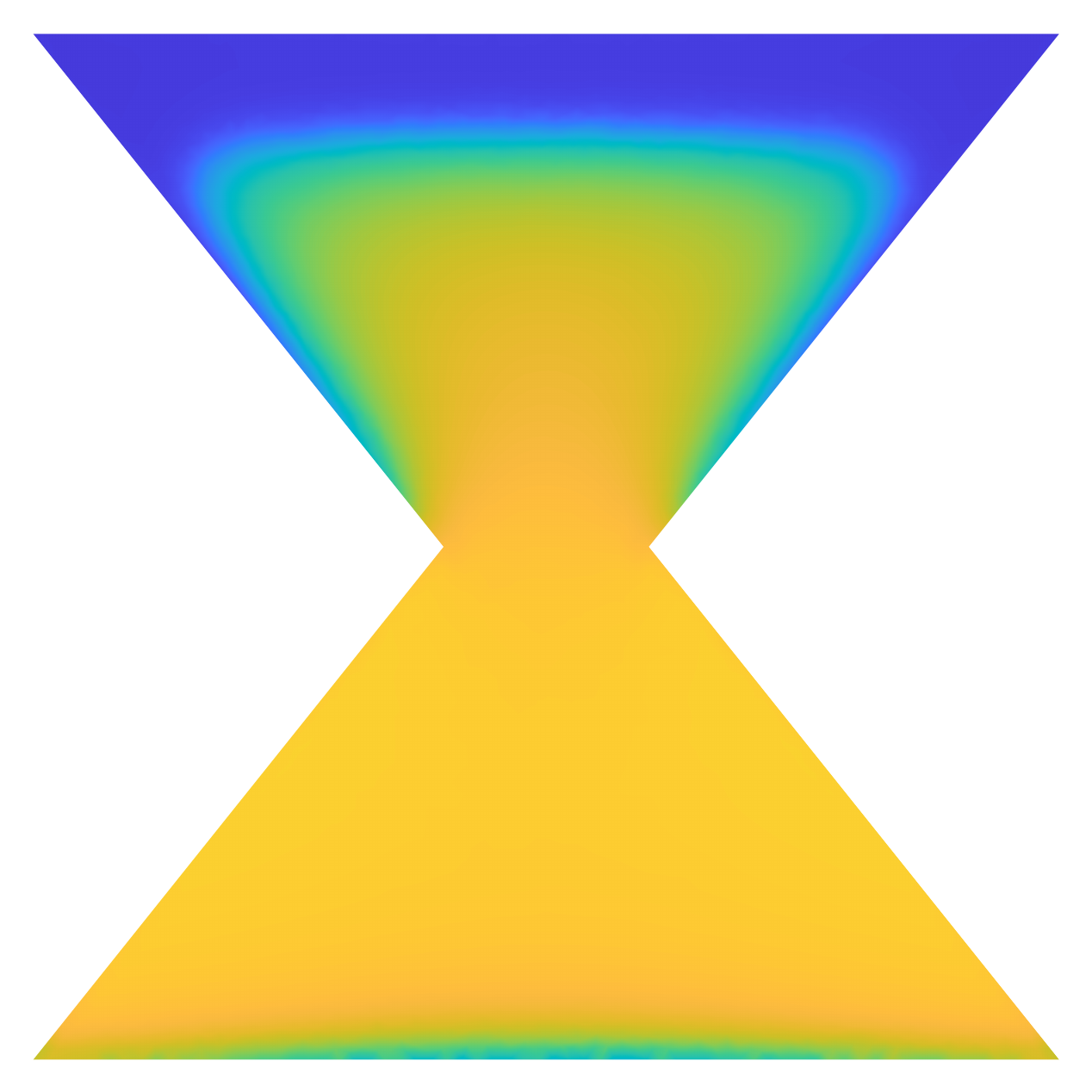}} \,
	\subfloat[][$t=16$]{\includegraphics[trim={4.7cm 7.8cm 4cm 7.25cm},clip,width=0.29\textwidth]{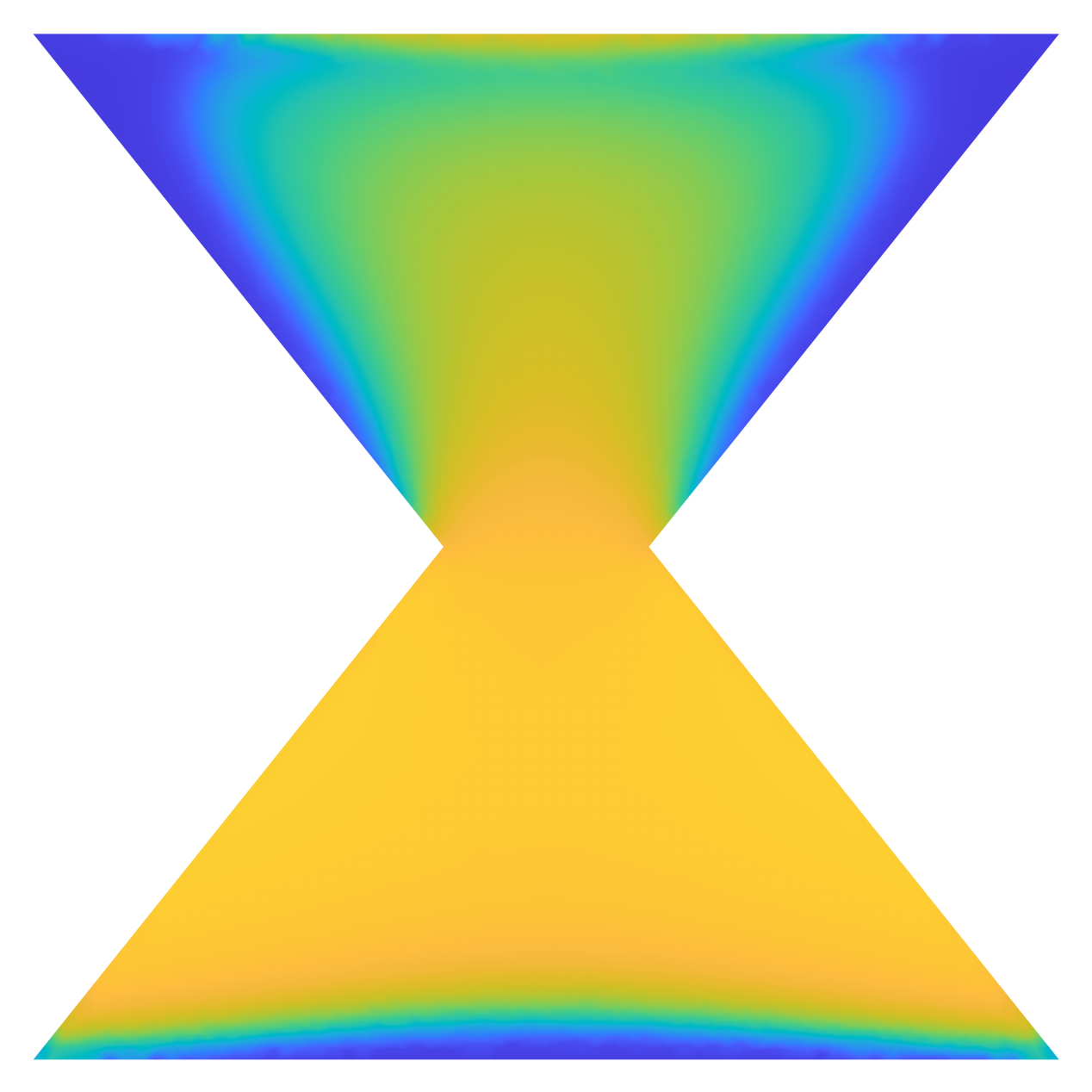}} \\
	\subfloat[][$t=24$]{\includegraphics[trim={4.7cm 7.8cm 4cm 7.25cm},clip,width=0.29\textwidth]{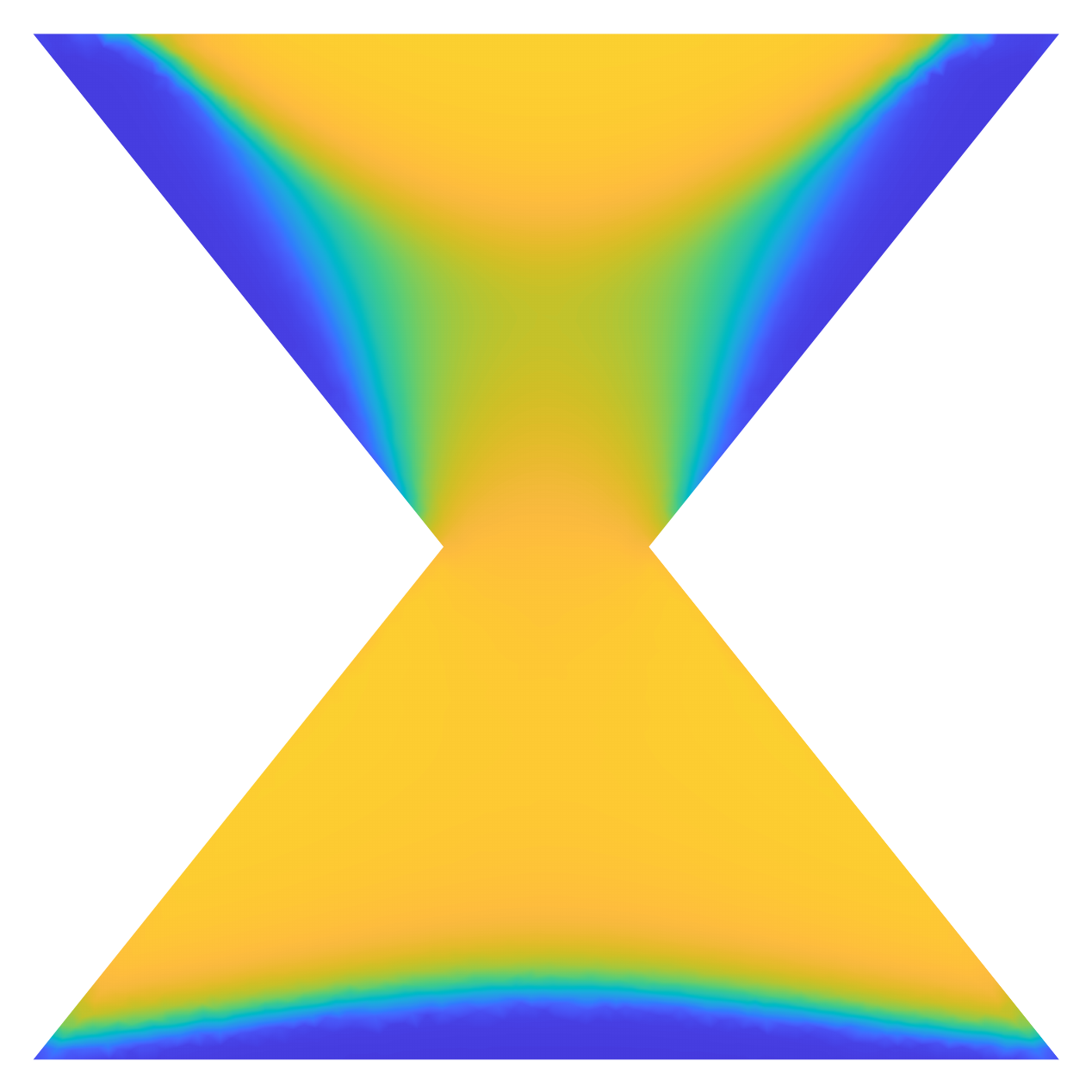}} \,
	\subfloat[][$t=48$]{\includegraphics[trim={4.7cm 7.8cm 4cm 7.25cm},clip,width=0.29\textwidth]{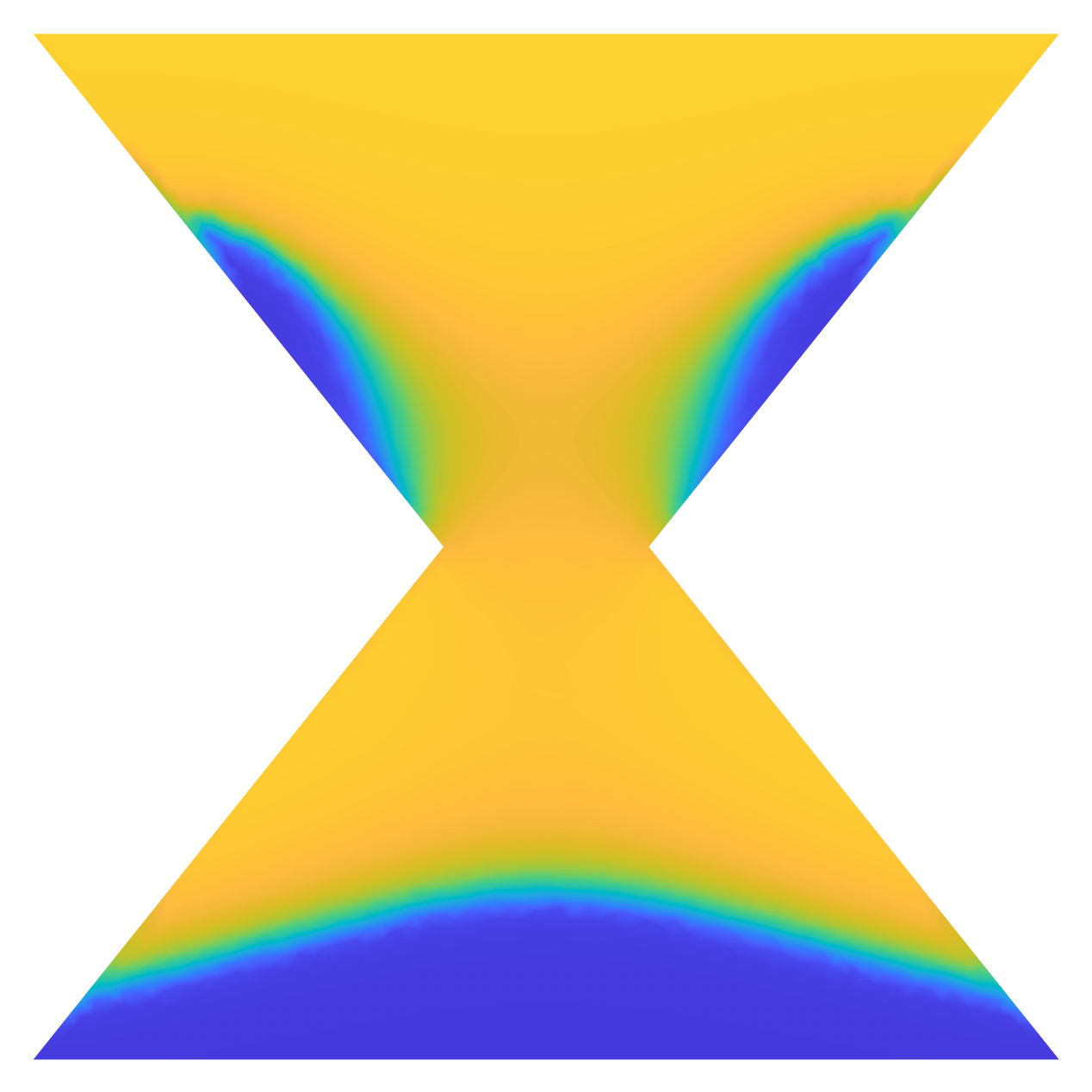}} \,
	\subfloat[][$t=120$]{\includegraphics[trim={4.7cm 7.8cm 4cm 7.25cm},clip,width=0.29\textwidth]{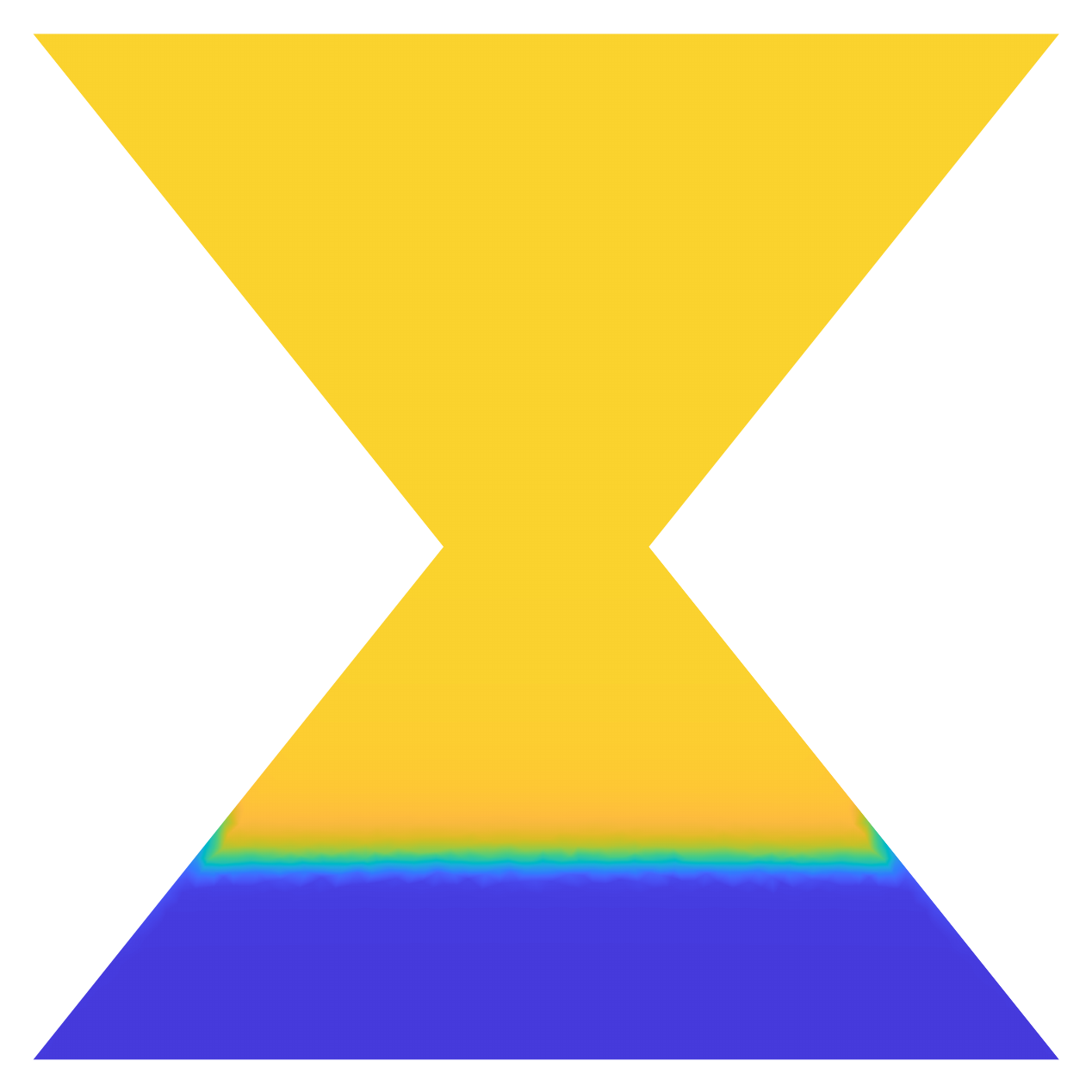}} \\
	\end{minipage}
	\hspace{-2em}
	\begin{minipage}{0.09\textwidth}
	\centering
	\vspace{0.4em}
	\includegraphics[trim={17.2cm 8.6cm 1.7cm 8.1cm},clip,height=0.2\textheight,]{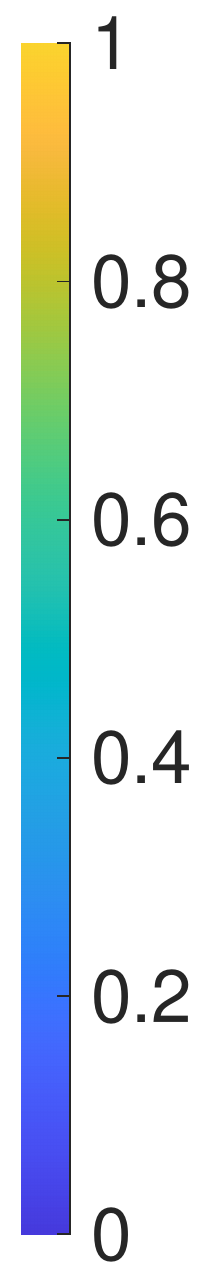}
	\end{minipage}
	\caption{Evolution of the saturation of the oil phase in the hourglass. The evolution of the water is complementary. As expected, the water, the denser phase, flows down the hourglass under the effect of gravity up until reaching the bottom.}
	\label{fig:twophaseflow}
	\end{figure}

	\section{Conclusion}
	
	In this work we proposed and analyzed different notions of extrapolation in the Wasserstein space. We showed how these can be used to construct a second-order time discretization of Wasserstein gradient flows, based on a two-step reformulation of the classical BDF2 scheme. According to the specific notion considered, we could prove different types of convergence guarantees for the scheme. We also proposed a fully-discrete version of the method, and demonstrated numerically its second-order accuracy in space and time. The possibility to provide an implementable scheme is in fact the main advantage of our approach compared to previous works also based on the BDF2 scheme \cite{Matthes2019bdf2}, or on the midpoint rule \cite{Legendre2017VIM}. The different type of extrapolations and their properties are summarized in Table \ref{tab:extra}.
	
	\begin{table}
	\centering
	\caption{Summary of the different types of extrapolation proposed in the present work.}
	\label{tab:extra}
	\renewcommand{\arraystretch}{1.3}
	\begin{tabular}{C{3.5cm} C{3cm} C{3cm} C{3cm}}
	\toprule
	& Free-flow extrapolation \eqref{eq:freeflow} & Viscosity extrapolation \eqref{eq:omegahj} & Metric extrapolation \eqref{eq:metricextraF}  \\
	\toprule
	Fokker-Planck conv. & \cmark & \textbf{?} & \cmark  \\ \hline
	EVI conv.& \textbf{?}  & \textbf{?} & \cmark  \\ \hline
	Implementation & \textbf{?} & \cmark &  \textbf{?}\\ \hline 
	Second order & \textbf{?} & \cmark & \textbf{?}  \\
	\toprule
	\end{tabular}
	\end{table}
	
	In order to provide our fully discrete scheme, we worked in the framework of Eulerian discretizations and considered an extrapolation based on viscosity solutions of the Hamilton-Jacobi equation. The resulting scheme is robust and allows to achieve  second order of accuracy both in space and time, but it does not verify the hypotheses of our convergence results. The free-flow extrapolation could be implemented straightforwardly in the framework of Lagrangian discretizations (see, e.g., \cite{matthes2014convergence,calvez2014particle} for Lagrangian discretizations of Wasserstein gradient flows), although in this setting it would be challenging to achieve second order accuracy in space. 
	The metric extrapolation enjoys the nicest mathematical structure, and in principle one could exploit its dual formulation \eqref{eq:metricextra_dual}, which is a convex optimization problem, to implement it numerically. However, dealing with the strong-convexity constraint on the Brenier potential requires the development of dedicated tools. We will investigate this direction in a future work.

	\section*{Acknowledgements} 
	This work was partly supported by the Labex CEMPI (ANR-11-LABX-0007-01).
	TOG acknowledges the support of the french Agence Nationale de la Recherche through the project MAGA (ANR-16-CE40-0014).
	GT acknowledges that this project has received funding from the European Union’s Horizon 2020 research and innovation programme under the Marie Skłodowska-Curie grant agreement No 754362.
	The authors would like to thank Clément Cancès and Guillaume Carlier for fruitful discussions and suggestions on the topic.
	\begin{center}
	\vspace{0.5em}
	\includegraphics[width=0.15\textwidth]{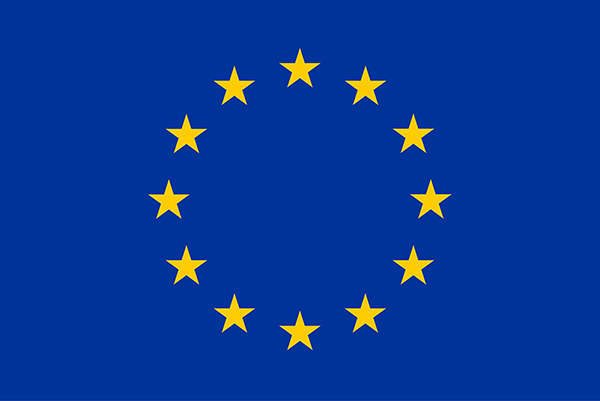}
	\end{center}

	\bibliographystyle{plain}      
	\bibliography{refs}   

\end{document}